\documentclass[11pt, a4paper, reqno]{amsart}

\usepackage{lmodern}
\usepackage[utf8]{inputenc}
\usepackage{amsmath}
\usepackage{graphicx}
\usepackage{amssymb}
\usepackage{esint}
\usepackage[dvipsnames]{xcolor}
\usepackage{tikz}
\usepackage{xxcolor}
\usepackage{floatrow}

\usepackage{color}
\usepackage{amsthm}
\usepackage{epsfig}
\usepackage[english]{babel}
\usepackage{hyperref}
\usepackage[normalem]{ulem}
\usepackage{mathrsfs}
\usepackage{tikz}
\usetikzlibrary{decorations.markings,backgrounds}
\usetikzlibrary{arrows.meta}
\usepackage{pgfplots}
\usepackage{subfigure}
\usepackage{caption}
\usepackage{bbm}

\usepackage{stmaryrd}

\usepackage{float}
\usepackage{pst-plot}

\hypersetup{
	colorlinks,
	linkcolor={red!80!black},
	citecolor={blue!50!black},
	urlcolor={blue!80!black}
}
 \usepackage{diagbox}

\usepackage[a4paper,top=3.5 cm,bottom=3 cm,left=2.5 cm,right=2.5 cm]{geometry}

\theoremstyle{plain}
\newtheorem{lemma}{Lemma}[section]
\newtheorem{prop}[lemma]{Proposition}
\newtheorem{theorem}{Theorem}[section]

\theoremstyle{definition}

\theoremstyle{remark}
\newtheorem{remark}[lemma]{Remark}

\numberwithin{equation}{section}

\def\to{\rightarrow}

\usepackage{dsfont}

\renewcommand{\div}{\mathrm{div} \hspace{0.5mm}}        

\newcommand{\supp}{\mathrm{supp}}

\newcommand{\nchi}{{\raise.3ex\hbox{$\chi$}}}

\def\XXint#1#2#3{{\setbox0=\hbox{$#1{#2#3}{\int}$ }
		\vcen{\hbox{$#2#3$ }}\kern-.6\wd0}}

\renewcommand{\det}{{\rm det}}

\usepackage{mathtools}

\makeatletter
\newcommand{\justified}{%
	\rightskip\z@skip%
	\leftskip\z@skip}
\makeatother
\usepackage{amsfonts}
\usepackage{aurical}

\newcommand\restr[2]{{
		\left.\kern-\nulldelimiterspace 
		#1 
		\vphantom{\big|} 
		\right|_{#2} 
}}

\newcommand{\var}{\varepsilon}

\DeclareFontFamily{U}{mathx}{\hyphenchar\font45}
\DeclareFontShape{U}{mathx}{m}{n}{<-> mathx10}{}
\DeclareSymbolFont{mathx}{U}{mathx}{m}{n}
\DeclareMathAccent{\widebar}{0}{mathx}{"73}

\renewcommand{\i}{\ifmmode\mathit{\mathchar"7010 }\else\char"10 \fi}
\renewcommand{\j}{\ifmmode\mathit{\mathchar"7011 }\else\char"11 \fi}

\renewcommand{\le}{\leq} 
\renewcommand{\ge}{\geq}

\def\char{{1\!\mbox{\rm l}}}

\usepackage{dsfont}

\definecolor{orange}{rgb}{1,.549,0}
\definecolor{GreenYellow }{rgb}{ 0.15,   0.69, 0}
\definecolor{Yellowone}{rgb}{ 0, 1., 0} \definecolor{Goldenrod }{rgb}{  0, 0.10, 0.84}
\definecolor{Dandelion }{rgb}{ 0, 0.29, 0.84} 
\definecolor{Apricot }{rgb}{ 0, 0.32, 0.52}
\definecolor{Peach }{rgb}{ 0, 0.50, 0.70} 
\definecolor{GreenYellow}{cmyk}{0.15,0,0.69,0}
\definecolor{RoyalPurple}{cmyk}{0.75,0.90,0,0}
\definecolor{Yellow}{cmyk}{0,0,1,0}
\definecolor{BlueViolet}{cmyk}{0.86,0.91,0,0.04}
\definecolor{Goldenrod}{cmyk}{0,0.10,0.84,0}
\definecolor{Periwinkle}{cmyk}{0.57,0.55,0,0}
\definecolor{Dandelion}{cmyk}{0,0.29,0.84,0}
\definecolor{CadetBlue}{cmyk}{0.62,0.57,0.23,0}
\definecolor{Apricot}{cmyk}{0,0.32,0.52,0}
\definecolor{CornflowerBlue}{cmyk}{0.65,0.13,0,0}
\definecolor{Peach}{cmyk}{0,0.50,0.70,0}
\definecolor{MidnightBlue}{cmyk}{0.98,0.13,0,0.43}
\definecolor{Melon}{cmyk}{0,0.46,0.5,0}
\definecolor{NavyBlue}{cmyk}{0.94,0.54,0,0}
\definecolor{YellowOrange}{cmyk}{0,0.42,1,0}
\definecolor{RoyalBlue}{cmyk}{1,0.50,0,0}
\definecolor{Orange}{cmyk}{0,0.61,0.87,0}
\definecolor{Blue}{cmyk}{1,1,0,0}
\definecolor{BurntOrange}{cmyk}{0,0.51,1,0}
\definecolor{Cerulean}{cmyk}{0.94,0.11,0,0}
\definecolor{Bittersweet}{cmyk}{0,0.75,1,0.24}
\definecolor{Cyan}{cmyk}{1,0,0,0}
\definecolor{RedOrange}{cmyk}{0,0.77,0.87,0}
\definecolor{ProcessBlue}{cmyk}{0.96,0,0,0}
\definecolor{Mahogany}{cmyk}{0,0.85,0.87,0.35}
\definecolor{SkyBlue}{cmyk}{0.62,0,0.12,0}
\definecolor{Maroon}{cmyk}{0,0.87,0.68,0.32}
\definecolor{Turquoise}{cmyk}{0.85,0,0.20,0}
\definecolor{BrickRed}{cmyk}{0,0.89,0.94,0.28}
\definecolor{TealBlue}{cmyk}{0.86,0,0.34,0.02}
\definecolor{Red}{cmyk}{0,1,1,0}
\definecolor{Aquamarine}{cmyk}{0.82,0,0.30,0}
\definecolor{OrangeRed}{cmyk}{0,1,0.50,0}
\definecolor{BlueGreen}{cmyk}{0.85,0,0.33,0}
\definecolor{RubineRed}{cmyk}{0,1,0.13,0}
\definecolor{Emerald}{cmyk}{1,0,0.50,0}
\definecolor{WildStrawberry}{cmyk}{0,0.96,0.39,0}
\definecolor{JungleGreen}{cmyk}{0.99,0,0.52,0}
\definecolor{Salmon}{cmyk}{0,0.53,0.38,0}
\definecolor{SeaGreen}{cmyk}{0.69,0,0.50,0}
\definecolor{CarnationPink}{cmyk}{0,0.63,0,0}
\definecolor{Green}{cmyk}{1,0,1,0}
\definecolor{Magenta}{cmyk}{0,1,0,0}
\definecolor{ForestGreen}{cmyk}{0.91,0,0.88,0.12}
\definecolor{VioletRed}{cmyk}{0,0.81,0,0}
\definecolor{PineGreen}{cmyk}{0.92,0,0.59,0.25}
\definecolor{Rhodamine}{cmyk}{0,0.82,0,0}
\definecolor{LimeGreen}{cmyk}{0.50,0,1,0}
\definecolor{Mulberry}{cmyk}{0.34,0.90,0,0.02}
\definecolor{YellowGreen}{cmyk}{0.44,0,0.74,0}
\definecolor{RedViolet}{cmyk}{0.07,0.90,0,0.34}
\definecolor{SpringGreen}{cmyk}{0.26,0,0.76,0}
\definecolor{Fuchsia}{cmyk}{0.47,0.91,0,0.08}
\definecolor{OliveGreen}{cmyk}{0.64,0,0.95,0.40}
\definecolor{Lavender}{cmyk}{0,0.48,0,0}
\definecolor{RawSienna}{cmyk}{0,0.72,1,0.45}
\definecolor{Thistle}{cmyk}{0.12,0.59,0,0}
\definecolor{Sepia}{cmyk}{0,0.83,1,0.70}
\definecolor{Orchid}{cmyk}{0.32,0.64,0,0}
\definecolor{Brown}{cmyk}{0,0.81,1,0.60}
\definecolor{DarkOrchid}{cmyk}{0.40,0.80,0.20,0}
\definecolor{Tan}{cmyk}{0.14,0.42,0.56,0}
\definecolor{Purple}{cmyk}{0.45,0.86,0,0}
\definecolor{Gray}{cmyk}{0,0,0,0.50}
\definecolor{Plum}{cmyk}{0.50,1,0,0}
\definecolor{Black}{cmyk}{0,0,0,1}
\definecolor{Violet}{cmyk}{0.79,0.88,0,0}
\definecolor{White}{cmyk}{0,0,0,0}
\definecolor{rltred}{rgb}{0.75,0,0}
\definecolor{rltgreen}{rgb}{0,0.5,0}
\definecolor{oneblue}{rgb}{0,0,0.75}
\definecolor{marron}{rgb}{0.64,0.16,0.16}
\definecolor{forestgreen}{rgb}{0.13,0.54,0.13}
\definecolor{purple}{rgb}{0.62,0.12,0.94}
\definecolor{dockerblue}{rgb}{0.11,0.56,0.98}
\definecolor{freeblue}{rgb}{0.25,0.41,0.88}
\definecolor{myblue}{rgb}{0,0.2,0.4}
\definecolor{Melon}{rgb}{ 0.46, 0.50, 0}
\definecolor{Melone}{rgb}{ 0, 0.46, 0.50}

\allowdisplaybreaks
\usepackage[colorinlistoftodos]{todonotes}
\usepackage{url}

\usepackage{graphicx,tikz} 
\begin{document}


    	\title{Relaxation limit and asymptotic stability for the Euler-Navier-Stokes equations}

\author{Mingwen Fei, Ling-Yun Shou, and  Houzhi  Tang}



	\date{}
	
\maketitle

\begin{abstract}
The Euler-Navier-Stokes (E-NS) system arises as a macroscopic description of kinetic-fluid interactions, derived from the local-Maxwellian closure of the Vlasov-Fokker-Planck-Navier-Stokes flow.  In this paper, we investigate the singular limit of the system in $\mathbb{R}^d$ ($d\ge2$) when the relaxation parameter $\varepsilon>0$ tends to zero. In contrast to the Euler system with velocity damping, the E-NS model features only a weaker relaxation of the relative velocity, which makes it challenging to analyze its dynamics as $\varepsilon\rightarrow 0$. We develop an energy argument to show global-in-time error estimates between the E-NS system and its limit system, the so-called Kramers-Smoluchowski-Navier-Stokes (KS-NS) system. These error estimates enable us to prove the global existence and uniform-in-$\varepsilon$ regularity of the strong solution to the E-NS system in a hybrid critical Besov space with a sharp frequency threshold of order $\mathcal{O}(\varepsilon^{-1})$ separating the low- and high-frequency regimes. Moreover, the large-time asymptotic stability of the global solution to the E-NS system is established. 
More precisely, we derive the optimal decay rates 
of the solution uniformly in $\varepsilon$, and the enhanced decay rates 
for the difference between the densities of the E-NS system and the KS-NS system.

	\end{abstract}

	\section{Introduction}

    \subsection{Presentation of the models}
		Recently, coupled kinetic-fluid models have received  a bulk of attention due to their wide range of applications in modeling reaction flows of sprays, atmospheric pollution modeling,  chemical engineering, wastewater treatment, and dust collecting units \cite{MR2226800,MR2041452,MR709743,1981Collective,2006Large,Williams1958Spray}.  The Euler-Navier-Stokes (E-NS) system describes the motion of the compressible Euler equations coupled with an incompressible viscous fluid derived from kinetic-fluid interactions, which captures the two-phase coupling through a drag (Brinkman) force in $\mathbb{R}^d$ ($d\ge 2$). This takes the form
	\begin{equation}\label{main1}
		\left\{
		\begin{aligned}
			&\partial_t\rho^\var+\frac{1}{\varepsilon}\text{div}(\rho^\var w^\var)=0,\\
			&\partial_t(\rho^\var w^\var)+\frac{1}{\varepsilon}\text{div}(\rho^\var w^\var\otimes w^\var)+\frac{1}{\varepsilon}\nabla \rho^\var=-\frac{1}{\varepsilon^2}\rho^\var (w^\var-\varepsilon u^\var),\\
			&\partial_tu^\var+u^\var\cdot\nabla u^\var+\nabla P^\var=\mu\Delta u+\frac{1}{\varepsilon}\rho^\var(w^\var-\varepsilon u^\var),\\
			&{\rm div} u^\var=0,
		\end{aligned}
		\right.
	\end{equation}
	which governs the evolution of the density $\rho^\var=\rho^\var(t,x)$ and the velocity $w^\var=w^\var(t,x)$ for the isothermal compressible Euler equations, as well as the velocity $u^\var=u^\var(t,x)$ for the incompressible Navier-Stokes equations. Here $\var>0$ denotes the relaxation‐time parameter of the particle velocity, measuring the characteristics  of light particles in the surrounding fluid (cf. \cite{MR2106333}), and $\mu$ is the viscosity coefficient. 
	
	We consider the Cauchy problem for the E-NS system \eqref{main1} subject to the initial conditions
	\begin{align}\label{main1d}
		(\rho^\var, w^\var, u^\var)(t,x)=(\rho_0^\var, w_0^\var, u_0^\var)(x)\rightarrow (\bar{\rho},0,0)\quad\text{as}\quad|x|\rightarrow\infty
	\end{align}
	with a constant background density $\bar{\rho}>0$. Without loss of generality, we set $\bar{\rho}=1$ in this paper.

The E-NS system \eqref{main1} can be derived from a kinetic-fluid model describing Brownian particles immersed in a viscous incompressible fluid, known as the Vlasov-Fokker-Planck-Navier-Stokes (VFP-NS) equations. 
Let $f^\var(t,x,\xi)$ denote the distribution function of particles located at $x\in\mathbb{R}^d$ with velocity $\xi\in\mathbb{R}^d$ at time $t\in\mathbb{R}_+$, and $u^\var(t,x)$ be the velocity field in the incompressible Navier-Stokes equations. 
In the light-particle regime, the coupled kinetic-fluid system is given by
\begin{equation}\label{NSVFP}
\left\{
\begin{aligned}
&\partial_t f^\var
   + \frac{1}{\varepsilon}\,\xi\cdot\nabla_x f^\var
   + \frac{1}{\varepsilon^{2}}\, \operatorname{div}_\xi\!\left((\varepsilon u^\var - \xi) f^\var - \nabla_\xi f^\var\right) = 0,\\[2mm]
&\partial_t u^\var + u^\var\cdot\nabla u^\var + \nabla P^\var
   = \mu \Delta u 
     + \frac{1}{\varepsilon}\int_{\mathbb{R}^d} (\xi - \varepsilon u^\var)\, f^\var \, d\xi,\\[2mm]
&\operatorname{div} u^\var = 0 .
\end{aligned}
\right.
\end{equation}
Define 
\[
\rho^\var(t,x):=\int_{\mathbb{R}^d} f^\var(t,x,\xi)\, d\xi,
\qquad
\rho^\var\, w^\var(t,x):=\int_{\mathbb{R}^d} \xi\, f^\var(t,x,\xi)\, d\xi .
\]
and assume that $f^\var$ is a local Maxwellian of the form
\[
f^\var(t,x,\xi)
   := \frac{\rho^\var(t,x)}{(2\pi)^{d/2}}
     \exp\!\left(-\frac{\lvert \xi - \varepsilon u(t,x)\rvert^2}{2}\right).
\]
Under this Maxwellian ansatz, the system \eqref{NSVFP} leads to the E-NS system \eqref{main1}. There have been many important mathematical works related to the kinetic-fluid model \eqref{NSVFP}, see \cite{MR3227296,MR3403400,MR2729436,MR2106333,MR2106334,MR4076066,MR3073216} and references therein.

Formally, suppose that $(\rho^\varepsilon, w^\varepsilon, u^\varepsilon)$ is a sufficiently regular solution to \eqref{main1} and 
\begin{align}
\lim_{\varepsilon\to 0}~\Big(\rho^\varepsilon, \tfrac{1}{\varepsilon} w^\varepsilon, u^\varepsilon, P^\var\Big)
   = (\rho^*, W^*, u^*, \tilde{P}^*),
   \qquad 
\lim_{\varepsilon\to 0}~(\rho_0^\varepsilon, u_0^\varepsilon)
   = (\rho_0^*, u_0^*),\label{limit111}
\end{align}
then we deduce from $\eqref{main1}_1$ that 
\begin{align}\label{YUI1}
\partial_t\rho^*+\text{div}(\rho^* W^*)=0.
\end{align}
Taking the limit in $\eqref{main1}_2$ as $\var\rightarrow0$ and using \eqref{limit111}, we derive a new \emph{Darcy-type law} 
\begin{align}\label{darcy}
\rho^* W^* = - \nabla \rho^* + \rho^* u^* .
\end{align}
Combining \eqref{YUI1} and \eqref{darcy} leads to 
\begin{align}\label{YUI3}
\partial_t\rho^*+\text{div}(\rho^* u^*)=\Delta \rho^*.
\end{align}
Taking the limit in $\eqref{main1}_3$ as $\var\rightarrow0$, then using \eqref{darcy} yields that
\begin{align}\label{YUI4}
&\partial_t u^* + u^* \cdot \nabla u^* + \nabla P^* = \mu \Delta u^*,
\end{align}
with the new pressure $P^*:=\tilde{P}^*+\rho^*$.

We combine \eqref{YUI3}, \eqref{YUI4} and $\eqref{main1}_4$ together to obtain the Kramer-Smoluchowski equation coupled with the incompressible Navier-Stokes equations, namely the KS-NS system:
\begin{equation}\label{main2}
\left\{
\begin{aligned}
&\partial_t \rho^* + \operatorname{div}(\rho^* u^*) = \Delta \rho^*,\\
&\partial_t u^* + u^* \cdot \nabla u^* + \nabla P^* = \mu \Delta u^*,\\
&\operatorname{div} u^* = 0,
\end{aligned}
\right.
\end{equation}
supplemented with the initial condition
\begin{align}\label{main2d}
(\rho^*, u^*)(0,x) = (\rho_0^*, u_0^*)(x)
   \longrightarrow (1,0) 
   \quad \text{as } |x| \to \infty.
\end{align}

To the best of our knowledge, no rigorous results are currently available concerning the justification of the convergence from the E-NS system \eqref{main1} to the KS-NS system \eqref{main2}, nor the derivation of the Darcy-type law \eqref{darcy}.

        \subsection{State of the art}
We recall some related developments that are closely connected to the present work.

\medskip
\noindent
{\textbf{Euler equations with damping.}}
The system~\eqref{main1} is strongly related to the following classical compressible Euler equations with damping under the diffusive scaling:
\begin{equation}\label{euler}
	\left\{
	\begin{aligned}
		&\partial_t\rho^\var+\frac{1}{\varepsilon}\operatorname{div}(\rho^\var w^\var)=0,\\[2pt]
		&\partial_t(\rho^\var w^\var)+\frac{1}{\varepsilon}\operatorname{div}(\rho^\var w^\var\otimes w^\var)
		+\frac{1}{\varepsilon}\nabla P(\rho^\var)
		=-\frac{1}{\varepsilon^2}\rho^\var w^\var.
	\end{aligned}
	\right.
\end{equation}
Indeed, when $u^\var=0$, the system~\eqref{main1} reduces to~\eqref{euler} in the isothermal case of $P(\rho^\var)=\rho^\var$.

In the absence of the damping term $\varepsilon^{-2}\rho^\var w^\var$, it is well-known that local classical 
solutions of the Euler equations may develop singularities (such as shock waves), even for smooth 
and small initial data (see Dafermos~\cite{MR3468916} and Lax~\cite{MR165243}).  
When the damping term is present, Wang and Yang~\cite{MR1834121}, and later Sideris, Thomases, and Wang~\cite{MR1978315}, 
proved that such blowup can be prevented and global classical solutions exist for small perturbations 
in Sobolev spaces $H^{s}(\mathbb{R}^{d})$ with $s> d/2+1$.  
We also refer to extensive works \cite{MR2754341,MR2349349,MR4951948,MR2107774,MR2058165} on global existence and large-time asymptotics for general hyperbolic systems 
satisfying the Shizuta-Kawashima condition. Xu and Kawashima~\cite{MR3149065,MR3360739} extended the well-posedness and stability theory for general partially dissipative systems to the 
inhomogeneous critical space $B^{d/2+1}_{2,1}$.  
More recently, Danchin and Crin-Barat~\cite{MR4419369,MR4612415} refined this framework by improving 
the control of low frequencies from $L^2$ to $\dot{B}^{d/2}_{2,1}$, or more generally 
$\dot{B}^{d/p}_{p,1}$ for $2\le p\le4$.  
We also mention that the second author, together with Xu and Zhang~\cite{xuzhangshou}, 
developed an $L^{p}$ characterization of sharp upper and lower decay rates for partially dissipative systems.

Formally, if $(\rho^\varepsilon, \varepsilon^{-1} w^\varepsilon)$ converges to $(\rho^*, W^*)$ as $\varepsilon\to0$, 
then the limit dynamics of \eqref{euler} are governed by the porous medium equation
\[
\partial_t\rho^* - \Delta P(\rho^*) = 0,
\]
along with the classical {\emph{Darcy law}}
\[
\rho^* W^* = -\nabla P(\rho^*).
\]

The rigorous justification of such relaxation limit has a long history.  
In one spatial dimension, the results can be traced back to 
Liu~\cite{MR872145}, Marcati, Milani and Secchi~\cite{MR920759}, and Marcati and Milani~\cite{MR1042662}.  
For large $BV$ data, Junca and Rascle~\cite{MR1900673} proved the relaxation limit from the isothermal Euler system 
to the heat equation. In several dimensions, Coulombel, Goudon, and Lin~\cite{MR2255190,MR3057138} 
established uniform-in-$\varepsilon$ estimates of smooth solutions for the isentropic Euler equations~\eqref{euler} 
and justified the corresponding weak relaxation limit.  
See also Xu and Wang~\cite{MR3003282} for an analysis in inhomogeneous critical spaces, 
and Peng and Wasiolek~\cite{MR3519534} for more general hyperbolic systems. Crin-Barat and Danchin~\cite{MR4419369} obtained uniform-in-$\var$ regularity estimates in critical spaces by rescaling the system to the normalized case $\varepsilon=1$.  They further introduced a damped mode that allows global-in-time error estimates with explicit convergence rates.  when chemotactic effects are included, Crin-Barat, He and the second author~\cite{MR4643014} tracked the parameter in energy estimates and justified the relaxation limit toward the Keller-Segel model, together with uniform large-time behavior. 
Related functional frameworks have also been used to study hyperbolic systems that are not covered by the classical Shizuta–Kawashima theory (see for instance
\cite{MR3356581,MR4878172,MR4548999,MR4752555,MR4970471}).  


 \vspace{2mm}
 
 \noindent
 \textbf{Euler-Navier-Stokes equations.}  Note that although the relative velocity dissipation in the E-NS system is weaker than the direct velocity dissipation of the damped Euler system, it is enough to prevent finite-time blow up. Indeed, for the E-NS system \eqref{main1} in the case $\var=1$, Choi \cite{MR3546341} first proved the global existence and exponential stability of classical solutions for small initial perturbations in $H^s(\mathbb{T}^3)$ with  $s>7/2$. Then, Huang et al. \cite{MR4776378} proved the global well-posedness and  optimal time-decay estimates of the classical solution to the E-NS system in $\mathbb{R}^3$ with high Sobolev regularity. When the pressure term $\nabla\rho$ in $\eqref{main1}_2$ vanishes, the coupled system is reduced to the pressureless E-NS  system. Choi and Jung \cite{MR4316126} first applied the weighted energy method to investigate the global well-posedness of solutions to the pressureless E-NS system and proved the convergence toward the equilibrium state at almost optimal decay rates. Later, the optimal decay rates were achieved in \cite{choijung24,htz24}.  
		
When the coupled Navier–Stokes phase is compressible, there are many important works. Choi \cite{MR3546341} established the global existence and uniqueness of strong solutions in the whole space and in the periodic domain, and also proved exponential time stability in the periodic case. Choi and Jung \cite{MR4344262} showed the global existence and uniqueness of solutions near the equilibrium state in a  bounded domain. In addition, Wu, Zhang and Zhou \cite{MR4175837}, and Tang and Zhang \cite{MR4262061} obtained the optimal algebraic time-decay rates  in the whole space using the spectral analysis method. Later, Li and Shou \cite{MR4596744} established the global well-posedness and optimal time-decay estimates in the critical Besov space. Recently, without the pressure in the Euler phase, the global dynamics of strong solutions near the equilibrium state were studied in \cite{MR3487272}, and the finite-time blow-up phenomenon of classical solutions was investigated in \cite{MR3723164}.  


 \vspace{2mm}


There are few results on hydrodynamic limits from the VFP-NS equations to the E-NS or KS-NS system. Goudon, Jabin and Vasseur \cite{MR2106333} first studied the hydrodynamic limit for the VFP-NS model to the KS-NS equations in this context. Inspired by this work, Fang, Qi and Wen \cite{fqw2024} studied the global strong convergence from the VFP-NS system to the KS-NS system for small initial perturbations under the Sobolev framework by the method of Hilbert expansion.  When local alignment forces are included in the VFP-NS system, the rigorous convergence of the system to the E-NS system via the relative entropy method was established by Carrillo, Choi and Karper~\cite{MR3465376}. 

 \vspace{2mm}

In conclusion, the existing literature concerns the hydrodynamic limits of the VFP-NS system toward either the E-NS system or the KS-NS system under different scaling regimes, cf. \cite{MR3465376,fqw2024,MR2106333}.  
This naturally raises the following question  (see Figure \ref{fig:1}): \emph{Can one rigorously justify the singular limit from the E-NS system to the KS-NS system?}

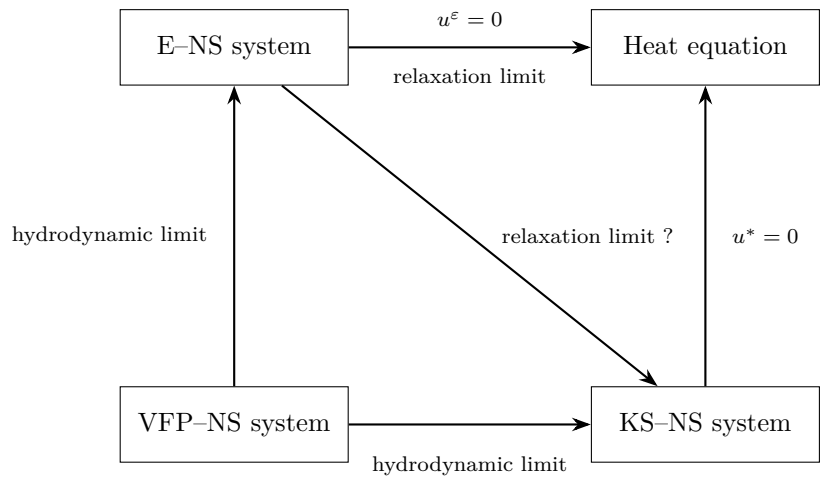
\begin{figure}[!ht]
\centering
\begin{tikzpicture}[
  system/.style={draw, rectangle, minimum width=3cm, minimum height=1cm, align=center, font=\small},
  arrow/.style={-{Stealth[length=2.5mm, width=1.8mm]}, thick, line width=0.8pt}
]

  \def\side{6}          
  \def\height{5.2}

  \node[system] (ENS)   at (0,\height-0.2) {E--NS system};
  \node[system] (VFPNS) at (0,0) {VFP--NS system};

  \node[system] (NSKS) at (\side/2+3.2,0) {KS--NS system};
  \node[system] (HEAT) at (\side/2+3.2,\height-0.2) {Heat equation};


  \draw[arrow] (VFPNS) -- 
      node[midway, left=6pt] {\scriptsize hydrodynamic limit} 
      (ENS);

  \draw[arrow] (ENS) -- 
      node[midway, right=8pt] {\scriptsize relaxation limit~?} 
      (NSKS);

  \draw[arrow] (VFPNS) -- 
      node[midway, below=8pt] {\scriptsize hydrodynamic limit} 
      (NSKS);

  \draw[arrow] (ENS) -- 
      node[midway, above=4pt] {\scriptsize $u^\varepsilon = 0$}
      node[midway, below=4pt] {\scriptsize relaxation limit}
      (HEAT);

  \draw[arrow] (NSKS) -- 
      node[midway, right=6pt] {\scriptsize $u^* = 0$}
      (HEAT);

\end{tikzpicture}
\caption{Relations between the E--NS, VFP--NS and KS--NS systems}
\label{fig:1}
\end{figure}

        \subsection{Contributions, difficulties and main ideas}

The first objective of this paper is to close the missing link in the macroscopic description of two-phase flow (see Figure~\ref{fig:1}).  
We establish the singular limit from the E-NS system \eqref{main1} to the KS-NS system \eqref{main2} and the Darcy-type law \eqref{darcy}.  
Our convergence result is global in time and accommodates ill-prepared initial data (see Remark \ref{r11.5}).

Secondly, we aim to prove the global well-posedness of the E-NS system \eqref{main1} for small initial perturbations in the critical space. Precisely, we assume
\[
\rho^\varepsilon_0-1\in \dot{B}^{\frac{d}{2}-1}_{2,1}\cap \dot{B}^{\frac{d}{2}+1}_{2,1},\quad 
 w^\varepsilon_0\in \dot{B}^{\frac{d}{2}}_{2,1}\cap \dot{B}^{\frac{d}{2}+1}_{2,1},\quad 
\mathcal{P} w^\varepsilon_0\in \dot{B}^{\frac{d}{2}-1}_{2,1},\quad
u^\varepsilon_0\in \dot{B}^{\frac{d}{2}-1}_{2,1}.
\]
Here $\mathcal{P}$ denotes the incompressible projection (see \eqref{AT1}). We emphasize that the different regularity assumptions on $\rho_0^\varepsilon,w_0^\varepsilon$ and $u_0^\varepsilon$ yield distinct qualitative estimates for the solution. The regularity index $d/2+1$ is regarded as critical since $\dot{B}^{d/2+1}_{2,1}$ is continuously 
embedded into the space of globally Lipschitz functions.  
Controlling the Lipschitz norm is essential in preventing finite-time blowup in hyperbolic systems 
(see \cite{MR3468916}). On the other hand, the space $\dot{B}_{2,1}^{d/2-1}$ is the classical Fujita–Kato scaling-critical space for the incompressible Navier–Stokes equations (e.g., cf. \cite{MR1753481}), and the incompressible component $\mathcal{P} w_0^\varepsilon$ belongs to the same space as $u_0^\varepsilon$ due to the relative velocity coupling.

Our third objective is to characterize the large-time behavior of solutions in the singular limit.  
Under the additional assumption that the initial perturbations are uniformly bounded in $\dot{B}^{\sigma_1}_{2,\infty}$ ($-d/2\le \sigma_1 < d/2-1$), we obtain  uniform-in-$\varepsilon$ time-decay estimates, with the same rates as those of the limit KS-NS system.  
Moreover, we establish an enhanced stability property for the error: the difference between the densities of the E-NS system and the limit KS-NS system decays at faster rates and converges with an explicit $\mathcal{O}(\varepsilon)$ rate.  
This reveals a diffusion phenomenon: the E-NS system is asymptotically equivalent to a diffusive KS-NS system  for both long-time behavior and the singular limit process.

\vspace{2mm}

However, there are two main analytical difficulties.  
First, its dissipation mechanism is fundamentally different from that of the damped Euler system \eqref{euler}, particularly in the low-frequency regime.  
For the damped Euler equations, the velocity decays faster than the density due to direct damping, for example, cf.\cite{MR3360739}.
In contrast, the E-NS system exhibits a mixed hyperbolic–parabolic nature: the velocity enjoys parabolic dissipation inherited from the Navier-Stokes operator, whereas the variables $(\rho^\varepsilon,w^\varepsilon)$ are only weakly damped through nonlinear coupling.  
Consequently, all components decay at heat-like rates. 

The second difficulty stems from the intricate coupling between the Euler and Navier-Stokes parts in \eqref{main1}.  
The transport effects in the Euler phase interact strongly with the viscous dissipation in the Navier-Stokes phase, which produces nonlinear cross-interaction terms without direct coercivity.  
This prevents the direct application of classical hyperbolic relaxation theories \cite{MR4419369,MR4643014,MR4878172,MR3003282} and requires a delicate balance between hyperbolic and parabolic modes.

To overcome these challenges, we develop a refined energy method via Littlewood–Paley theory, which is fundamentally different from the relative entropy method and the Hilbert expansion used in previous works.  With the help of spectral analysis, a delicate frequency partition implies that as $\var\rightarrow0$, the high frequencies  $|\xi|\gtrsim \var^{-1}$ vanish while the low frequencies $|\xi|\lesssim \var^{-1}$, consistent with the limit diffusion regularity, cover all other frequencies. Thus, our result reveals the singular limit process from hyperbolic type equations to parabolic type equations. To carry it out in the energy framework, our key idea is to introduce two damped modes
\begin{align}\label{B1}
		\mathcal{Z}^\varepsilon=\mathcal{P}^{\top}w^\varepsilon+\frac{\varepsilon}{\rho^\varepsilon}\nabla \rho^\varepsilon\quad \text{and} \quad \mathcal{R}^\varepsilon=\mathcal{P}w^\varepsilon-\varepsilon u^\varepsilon.
	\end{align} 
These allow us to uncover the limit KS-NS structure with remainders involving $(\mathcal{Z}^\varepsilon, \mathcal{R}^\varepsilon)$ (see \eqref{main5}) to establish the error estimates in the low-frequency regime $|\xi|\lesssim \var^{-1}$.
    This, combined with a hypocoercivity energy argument in the high-frequency regime $|\xi|\gtrsim  \var^{-1}$, enables us to construct global solutions to the E-NS system by {\emph{performing direct error estimates}} with an explicit convergence rate, supplemented by the regularity properties of the limit system.  
The resulting error analysis not only yields the singular limit but may also be of independent interest, as it offers a new perspective on the relative velocity relaxation in two-phase flow.

	\subsection{Main results}

	Note that the KS-NS system \eqref{main2} is scaling-invariant under the transform
	\begin{align}
		( \rho^*_\lambda, u^*_\lambda,P^*_\lambda)(t,x)=(\rho^*, \lambda u^*,\lambda^2 P^*)(\lambda^2 t, \lambda x).
	\end{align}
	Since the pair of homogeneous Besov spaces $\dot{B}^{d/2}_{2,1}\times \dot{B}^{d/2-1}_{2,1}$ indeed possesses this invariance, we have global well-posedness in the optimal functional space for the system \eqref{main2}. To align with the regularity of $u^*_0$, it is natural to assume the $\dot{B}^{d/2-1}_{2,1}$-regularity for $\rho_0^*-1$. The first result about the global well-posedness  and large-time behavior of system \eqref{main2} is stated as follows, and the proof is given in Appendix A.
	
	\begin{prop}\label{thm1}
		Let $d\geq 2$. There is a positive constant $\alpha_0$ such that if 
		\begin{align}\label{a0}
			E_0:=\|\rho_0^*-1\|_{\dot{B}_{2,1}^{\frac{d}{2}-1}\cap \dot{B}^{\frac{d}{2}}_{2,1}}+\|u_0^*\|_{\dot{B}_{2,1}^{\frac{d}{2}-1}}\leq \alpha_0,
		\end{align}	
		then the Cauchy problem \eqref{main2}-\eqref{main2d} admits a unique global solution $(\rho^*, u^*)$ satisfying that  
        $$
        \rho^*-1\in C(\mathbb{R}_+; \dot{B}_{2,1}^{\frac{d}{2}-1}\cap \dot{B}^{\frac{d}{2}}_{2,1}),\quad  u^*\in C(\mathbb{R}_+; \dot{B}_{2,1}^{\frac{d}{2}-1})
        $$
        and for $t>0$ 
		\begin{align}
			&\|\rho^*-1\|_{L_t^\infty(\dot{B}_{2,1}^{\frac{d}{2}-1}\cap \dot{B}^{\frac{d}{2}}_{2,1})}+\|\rho^*-1\|_{L_t^1( \dot{B}_{2,1}^{\frac{d}{2}+2}\cap \dot{B}^{\frac{d}{2}+1}_{2,1})}\nonumber\\	
			&\quad+\|u^*\|_{L_t^\infty(\dot{B}_{2,1}^{\frac{d}{2}-1})}	
			+\|u^*\|_{L_t^1(\dot{B}_{2,1}^{\frac{d}{2}+1})}\leq CE_0,\label{r0}
		\end{align}
		where the positive constant $C$ is independent of time.
        
          Moreover, let $\sigma_1\in [-\frac{d}{2},\frac{d}{2}-1)$ be a fixed constant. If $
			\rho_0^*-1, u_0^*\in \dot{B}_{2,\infty}^{\sigma_1}$,	
		then the solution $(\rho^*, u^*)$ to the KS-NS system \eqref{main2}-\eqref{main2d} has the following decay rates:
		\begin{align}\label{decayNSKS}
			\|\Lambda^\sigma (\rho^*-1,u^*)(t)\|_{L^2}\leq C (1+t)^{-\frac{1}{2}(\sigma-\sigma_1)},\quad \forall\,\, \sigma>\sigma_1,\ t\geq 1,
		\end{align}
		with $C>0$ a constant independent of time.
	\end{prop}	
	\begin{remark}
       We achieve the optimal time-decay rates of the density fluctuation and the velocity for any order of derivatives. This reveals a heat-like smoothing effect of the KS-NS system.
	\end{remark}

    Before stating the main results, we introduce several  functionals. If $(\rho^\var, w^\var,u^\var)$ is a solution to \eqref{main1}-\eqref{main1d}, we define the energy functional
	\begin{align}
		E^\varepsilon(t):&=\|\rho^\var-1\|_{L_t^\infty(\dot{B}_{2,1}^{\frac{d}{2}-1}\cap \dot{B}^{\frac{d}{2}}_{2,1})}+\| u^\varepsilon\|_{L_t^\infty(\dot{B}_{2,1}^{\frac{d}{2}-1})}\nonumber\\
		&\quad+\|\mathcal{P}w^\varepsilon\|_{ L_t^\infty(\dot{B}_{2,1}^{\frac{d}{2}-1})}^{\ell,\var}+\|w^\varepsilon\|_{L_t^\infty(\dot{B}_{2,1}^{\frac{d}{2}})}^{\ell,\var}+\varepsilon\| w^\varepsilon\|_{L_t^\infty(\dot{B}_{2,1}^{\frac{d}{2}+1})}^{h,\varepsilon},\label{DTT}
	\end{align}
	and the dissipation functional 
	\begin{align}
		D^\varepsilon(t)&:=
		\|\rho^\var-1\|_{L_t^1( \dot{B}^{\frac{d}{2}+1}_{2,1})}+\|\rho^\varepsilon-1\|_{L_t^1(\dot{B}_{2,1}^{\frac{d}{2}+2})}^{\ell,\var}+\| u^\varepsilon\|_{L_t^1(\dot{B}_{2,1}^{\frac{d}{2}+1})}\nonumber\\
		&\quad+\frac{1}{\varepsilon}\|w^\varepsilon\|_{L_t^2(\dot{B}_{2,1}^{\frac{d}{2}})}
		+\frac{1}{\varepsilon}\| w^\varepsilon\|_{L_t^1(\dot{B}_{2,1}^{\frac{d}{2}+1})}\nonumber\\
		&\quad+\frac{1}{\varepsilon^2}\|\mathcal{Z}^\varepsilon\|_{L_t^1(\dot{B}_{2,1}^{\frac{d}{2}})}+\frac{1}{\varepsilon^2}\|\mathcal{R}^\varepsilon\|_{L_t^1(\dot{B}_{2,1}^{\frac{d}{2}-1}\cap\dot{B}_{2,1}^{\frac{d}{2}})},\label{ETT}
	\end{align}
	where $\mathcal{Z}^\varepsilon$ and $\mathcal{R}^\varepsilon$ are defined in \eqref{B1}.

	We also define the initial error energy functional
	\begin{align}\label{deltaE0}
		\delta E_0&:=\frac{1}{\var}\|\rho^\varepsilon_0-\rho_0^*\|_{\dot{B}_{2,1}^{\frac{d}{2}-1}}^{\ell,\var}
		+ \frac{1}{\varepsilon}\| u^\varepsilon_0-u^*\|_{\dot{B}_{2,1}^{\frac{d}{2}-1}}\nonumber\\
		&\quad+\|\mathcal{P}w^\varepsilon_0\|_{\dot{B}_{2,1}^{\frac{d}{2}-1}}^{\ell,\var}+\| w_0^\varepsilon\|_{\dot{B}_{2,1}^{\frac{d}{2}}}^{\ell,\varepsilon}+\varepsilon\|(\rho^\varepsilon_0-1,w^\varepsilon_0)\|_{\dot{B}_{2,1}^{\frac{d}{2}+1}}^{h,\varepsilon}.
	\end{align}  
	
	Our first main theorem concerns the global well-posedness and singular limit of the Cauchy problem \eqref{main1}-\eqref{main1d}, which is stated as follows.
	
	\begin{theorem}\label{thm2}
		Let $(\rho^*, u^*)$ be the global solution of the Cauchy problem \eqref{main2}-\eqref{main2d} with the initial data $(\rho^*_0,u_0^*)$.
		Then there are two generic positive constants $\var^*$ and $\delta_0$ such that if $0<\var\leq \var^*$ and
		\begin{align}\label{a1}
			\delta E_0\leq \delta_0,
		\end{align}
		the Cauchy problem \eqref{main1}-\eqref{main1d} admits a unique global solution $(\rho^\varepsilon,u^\varepsilon,w^\varepsilon)$ satisfying for all $t\geq0$ that
		\begin{align}
			E^\varepsilon(t)+D^\varepsilon(t)\leq C( E_0+\delta E_0),	\label{r1}
		\end{align} 
		and 
		\begin{align}
			&\|\rho^\varepsilon-\rho^*\|_{L_t^\infty(\dot{B}_{2,1}^{\frac{d}{2}-1})}+\|\rho^\varepsilon-\rho^*\|_{L_t^1(\dot{B}_{2,1}^{\frac{d}{2}+1})}+\|\partial_t(\rho^\varepsilon-\rho^*)\|_{L_t^1(\dot{B}_{2,1}^{\frac{d}{2}-1})} \nonumber\\
            &\quad+ \|u^\varepsilon-u^*\|_{L_t^\infty(\dot{B}_{2,1}^{\frac{d}{2}-1})}+\|u^\varepsilon-u^*\|_{L_t^1(\dot{B}_{2,1}^{\frac{d}{2}+1})}+\|\partial_t(u^\varepsilon-u^*)\|_{L_t^1(\dot{B}_{2,1}^{\frac{d}{2}+1})} \nonumber\\
            &\quad    +\|\var^{-1}w^\var-W^*\|_{L_t^1(\dot{B}_{2,1}^{\frac{d}{2}}+\dot{B}_{2,1}^{\frac{d}{2}+1})}\nonumber\\
            &\leq C \varepsilon( E_0+\delta E_0),
		\end{align}
		where $W^*$ is given by \eqref{darcy} and $C>0$ is a positive constant independent of $t$ and $\var$.
	\end{theorem}

    \begin{remark}\label{r1.5}
We first rigorously prove the validity of the singular limit from the E-NS system to the KS-NS system, with a global rate of convergence as $\var$ goes to zero, for initial data of critical regularity near equilibrium.
Importantly, the smallness assumptions on $\rho_0^\var$ and $u_0^\var$ depend {\emph{only}} on their limit profiles $\rho^*_0$ and $u^*_0$. To the best of our knowledge, our approach—based on the limit system and error estimates—is new in the analysis of Euler-type coupled models. Moreover, the smallness assumptions on $\rho_0^\var$ and $u_0^\var$ in Proposition \ref{thm1} may be removed in certain cases,  for example, in two spatial dimensions. 
\end{remark}

\begin{remark}\label{r11.5}
The initial data $(\rho_0^\varepsilon, w_0^\varepsilon, u_0^\varepsilon)$ considered in this paper are of \emph{ill-prepared} type in the sense that the Darcy-type compatibility condition 
\[
\varepsilon^{-1}\rho_0^\varepsilon w_0^\varepsilon + \nabla \rho_0^\varepsilon - \rho_0^\varepsilon u_0^\varepsilon \to 0 \quad \text{as } \varepsilon \to 0
\]
is not required to hold. 
In fact, we only assume that $w_0^\varepsilon = \mathcal{O}(1)$ (see \eqref{deltaE0} and \eqref{a1}).
    \end{remark}

	\begin{remark}
		By a suitable modification, our method can be extended to a more general $L^p$ framework in low frequencies (see \cite{MR4612415}). When $\var$ is a positive constant (not necessarily small), we can also prove the global existence of the E-NS system in critical Besov space by standard frequency-localized energy estimates (cf. \cite{{MR4596744}}). Here, we leave the details to the interested readers.
	\end{remark}

      	\begin{remark}
In this work, our focus is on the singular limit  from \eqref{main1} to the KS–NS system \eqref{main2} in the critical Besov space.
In the Sobolev setting with high regularity, Fang, Qi and Wen \cite{fqw2024} justified the limit from the VFP-NS system to \eqref{main2} via the Hilbert expansion method.
The case of low-regularity initial data presents additional challenges, which we will study in the future.
	\end{remark}

	Then, we study the uniform large-time behavior of the global solution $(\rho^\var,u^\var)$ to the Cauchy problem \eqref{main1}-\eqref{main1d}. We establish the asymptotic stability of $(\rho^\var,u^\var)$ toward the solution $(\rho^*,u^*)$ to the KS-NS system \eqref{main2}-\eqref{main2d} for general ill-prepared initial data. 
	
\begin{theorem}\label{thm3}
Let $(\rho^*, u^*)$ and $(\rho^\varepsilon, w^\varepsilon, u^\varepsilon)$ be the global solutions to the Cauchy problems \eqref{main2}–\eqref{main2d} and \eqref{main1}–\eqref{main1d} obtained in Proposition~\ref{thm1} and Theorem \ref{thm2}, respectively.  
In addition to the assumptions \eqref{a0} and \eqref{a1}, suppose that there exists a constant $\beta_1>0$ {\rm(}not necessarily small{\rm)}, independent of $\varepsilon$, such that for 
$-\frac{d}{2}\le \sigma_1<\frac{d}{2}-1$,
\begin{align}\label{a2}
\rho_0^\varepsilon=\rho_0^*, \qquad 
u_0^\varepsilon=u_0^*, \qquad 
\|(\rho_0^*-1, w_0^\varepsilon, u_0^\varepsilon)\|_{\dot{B}^{\sigma_1}_{2,\infty}}
\le \beta_1.
\end{align}
Then, for all $t\ge 1$, the solution $(\rho^\varepsilon, w^\varepsilon, u^\varepsilon)$ satisfies
\begin{align}
\|(\rho^\varepsilon-1,w^\varepsilon, u^\varepsilon)(t)\|_{\dot{B}^{\sigma}_{2,1}}&\le C(1+t)^{-\frac{1}{2}(\sigma-\sigma_1)},
\qquad\quad\,\, \sigma_1<\sigma\le\frac{d}{2}.\label{decay0}
\end{align}
Moreover, the error $(\rho^\varepsilon-\rho^*,\, u^\varepsilon-u^*)$ and the relative velocity $w^\varepsilon-\varepsilon u^\varepsilon$ have the following improved decay estimates:
\begin{align}
\|(\rho^\varepsilon-\rho^*)(t)\|_{\dot{B}^{\sigma}_{2,1}}
&\le C\varepsilon(1+t)^{-\frac{1}{2}(\sigma-\sigma_1)},
\qquad ~~\, \sigma_1<\sigma\le\frac{d}{2}-1,\label{decay1}
\\
\|(u^\varepsilon-u^*)(t)\|_{\dot{B}^{\sigma}_{2,1}}
&\le C\varepsilon(1+t)^{-\frac{1}{2}(\sigma-\sigma_1)},
\qquad ~~\, \sigma_1<\sigma\le\frac{d}{2},\label{decay11}
\\
\|(w^\varepsilon-\varepsilon u^\varepsilon)(t)\|_{\dot{B}^{\sigma}_{2,1}}
&\le C\varepsilon(1+t)^{-\frac{1}{2}(\sigma-\sigma_1+1)},
\qquad \sigma_1<\sigma\le\frac{d}{2}-1,\label{decay2}
\end{align}
where $C=C(\alpha_0,\delta_0,\beta_1)$ is independent of $t$ and $\varepsilon$. 

If $-\frac{d}{2}\le \sigma_1<\frac{d}{2}-2$, then for all $t\geq1$, the decay rate of the density difference may be improved to
\begin{align}
\|(\rho^\varepsilon-\rho^*)(t)\|_{\dot{B}^{\sigma}_{2,1}}
\le C\varepsilon (1+t)^{-\frac{1}{2}(\sigma-\sigma_1+1)},
\qquad \sigma_1<\sigma\le\frac{d}{2}-1.\label{decay3}
\end{align}
\end{theorem}

	\begin{remark}
By \eqref{decay0} and the standard embedding $\dot{B}^{\frac{d}{2}}_{2,1}\hookrightarrow L^{\infty}(\mathbb{R}^d)$, one has 
		\begin{align*}
			\|(\rho^\var-1,w^\var,u^\var)(t)\|_{L^{\infty}}\leq C(1+t)^{-\frac{1}{2}(\frac{d}{2}-\sigma_1)},\quad t\geq1.
		\end{align*}
		Moreover, due to $L^1(\mathbb{R}^d)\hookrightarrow \dot{B}^{-\frac{d}{2}}_{2,1}$, if choosing $\sigma_1=-\frac{d}{2}$ in \eqref{decay0}, our results are consistent with previous works on incompressible two-phase flows subject to the $L^1$-assumption (cf. \cite{MR4776378}), that is 
		\begin{align*}
			\|\Lambda^\sigma(\rho^\var-1,w^\var, u^\var)(t)\|_{L^{2}}\leq C(1+t)^{-\frac{d}{4}-\frac{\sigma}{2}},\quad  -\frac{d}{2}<\sigma\le\frac{d}{2},\quad t\geq1.
		\end{align*}
	\end{remark}

	\subsection{Strategies}
    \subsubsection{Reformulation}

     Let $a^*=\rho^*-1$. Then, the system \eqref{main2} reduces to 
	\begin{equation}\label{main4}
		\left\{
		\begin{aligned}
			&\partial_ta^*-\Delta a^*=-\text{div}(a^* u^*),\\
			&\partial_tu^*-\mu\Delta u^*+ \nabla P^*=-u^*\cdot\nabla u^*,\\
			&{\rm div} u^*=0,\\
			&(a^*,u^*)(0,x)=(a_0^*,u_0^*).
		\end{aligned}
		\right.
	\end{equation}
	Denote $a^\varepsilon=\rho^\varepsilon-1$. The E-NS system \eqref{main1} becomes 
	\begin{equation}\label{main3}
		\left\{
		\begin{aligned}
			&\partial_ta^\varepsilon+\frac{1}{\varepsilon}w^\varepsilon\cdot\nabla a^\varepsilon+\frac{1}{\varepsilon}(1+a^\varepsilon)\text{div}w^\varepsilon=0,\\
			&\partial_tw^\varepsilon+\frac{1}{\varepsilon}(1+h(a^\varepsilon))\nabla a^\varepsilon+\frac{1}{\varepsilon^2}(w^\varepsilon-\varepsilon u^\varepsilon)=-\frac{1}{\varepsilon}w^\varepsilon\cdot\nabla w^\varepsilon,\\
			&\partial_tu^\varepsilon-\mu \Delta u^\varepsilon+\nabla P^\varepsilon+\frac{1}{\varepsilon}(1+a^\varepsilon)(\varepsilon u^\varepsilon-w^\varepsilon)=-u^\varepsilon\cdot\nabla u^\varepsilon,\\
			&{\rm div} u^\varepsilon=0,\\
			&(a^\varepsilon,w^\varepsilon,u^\varepsilon)(0,x)=(a_0^\varepsilon,w_0^\varepsilon,u_0^\varepsilon),
		\end{aligned}
		\right.
	\end{equation}
	where $h(a^\varepsilon)$ is given by
	\begin{align}
		h(a^\varepsilon)=-\frac{a^\varepsilon}{1+a^\varepsilon}.	
	\end{align}

	\subsubsection{Spectral behavior and frequency threshold}\label{subsection:sp}
	

    	To help the reader understand the energy estimates, we provide the following asymptotic analysis of the eigenvalues associated with the linearized system for \eqref{main3}. For that, we employ the Hodge decomposition to decompose it into two systems. Let the incompressible projector $\mathcal{P}$ and the compressible projection $\mathcal{P}^{\top}$ be given by 
	\begin{align}\label{AT1}
		\mathcal{P}=\rm{Id}-\nabla(-\Delta)^{-1}\text{div}\quad\text{and}\quad \mathcal{P}^{\top}={\rm Id}-\mathcal{P}=\nabla(-\Delta)^{-1}\text{div}.
	\end{align}
	Then, we have
	\begin{equation}\label{Main:incom}
		\left\{
		\begin{aligned}
			&\partial_ta^\varepsilon+\frac{1}{\varepsilon}{\rm{div}}(\mathcal{P}^{\top} w^\varepsilon)=-\frac{1}{\varepsilon}\text{div}(a^\varepsilon w^\varepsilon), \\
			&\partial_t \mathcal{P}^{\top}w^\varepsilon+\frac{1}{\varepsilon}(1+h(a^\varepsilon))\nabla a^\varepsilon+\frac{1}{\var^2}\mathcal{P}^{\top} w^\varepsilon=-\frac{1}{\varepsilon}\mathcal{P}^{\top}(w^\varepsilon\cdot\nabla w^\varepsilon),\\
			&(a^\varepsilon,\mathcal{P}^{\top}w^\varepsilon)(0,x)=(a_0^\varepsilon,\mathcal{P}^{\top}w_0^\varepsilon),
		\end{aligned}
		\right.
	\end{equation}
	and 
	\begin{equation}\label{Main:com}
		\left\{
		\begin{aligned}
			&\partial_t \mathcal{P} w^\varepsilon+\frac{1}{\var^2}(\mathcal{P} w^\varepsilon-\varepsilon u^\varepsilon)=-\frac{1}{\varepsilon}\mathcal{P}(w^\varepsilon\cdot\nabla w^\varepsilon),\\
			&\partial_t u^\varepsilon-\mu\Delta u^\varepsilon+\frac{1}{\var}(\varepsilon u^\varepsilon-\mathcal{P}w^\varepsilon)=-\mathcal{P}(u^\varepsilon\cdot\nabla u^\varepsilon)+\frac{1}{\var}\mathcal{P}(a^\var(w^\var-\var u^\var)),\\
			&(\mathcal{P}w^\varepsilon,u^\varepsilon)(0,x)=(\mathcal{P}w^\varepsilon_0,u_0^\varepsilon).
		\end{aligned}
		\right.
	\end{equation}


    \noindent
    \textbf{Compressible part}. We first consider the linearized system of \eqref{Main:incom} as follows:
	\begin{equation}\label{Lmain1}
		\left\{
		\begin{aligned}
			&\partial_ta^\varepsilon+\frac{1}{\varepsilon}{\rm{div}}(\mathcal{P}^{\top} w^\varepsilon)=0,\\
			&\partial_t (\mathcal{P}^{\top}w^\varepsilon)+\frac{1}{\varepsilon}\nabla a^\varepsilon+\frac{1}{\var^2}\mathcal{P}^{\top} w^\varepsilon=0,\\
			&(a^\varepsilon,\mathcal{P}^{\top}w^\varepsilon)(0,x)=(a_0^\varepsilon,\mathcal{P}^{\top}w_0^\varepsilon).
		\end{aligned}
		\right.
	\end{equation}
	which can be viewed as the damped compressible Euler equations. Let $G_1(t,x)$ be the Green function of \eqref{Lmain1}. After taking the Fourier transform of \eqref{Lmain1} in $x$, we have
	$$
	\partial_t\hat{G}_1(t,\xi)+B_1\hat{G}_1(t,\xi)=0,\quad \hat{G}_1(0,\xi)={\rm{I}},
	$$
	where
	\begin{equation}
		B_1=\left(\begin{array}{cc}
			0 & \frac{i}{\varepsilon}\xi^t \\
			\frac{i}{\varepsilon}\xi &\frac{1}{\varepsilon^2}\\
		\end{array}\right).
	\end{equation}
The eigenvalues of $B_1$ are determined from
\begin{equation}
	\det(\lambda \mathrm{I}+B_1)
	= \left(\lambda+\frac{1}{\varepsilon^{2}}\right)^{d-1}
	  \left(\lambda^{2}+\frac{1}{\varepsilon^{2}}\lambda
	  +\frac{|\xi|^{2}}{\varepsilon^{2}}\right)
	=0.
\end{equation}
A straightforward calculation yields
\begin{align*}
	\lambda_{0} &= -\frac{1}{\varepsilon^{2}}
	\quad \text{(with multiplicity $d-1$)}, \\
	\lambda_{1} &= -\frac{1}{2\varepsilon^{2}}
		+ \frac{1}{2\varepsilon^{2}}
		  \sqrt{\,1 - 4\varepsilon^{2}|\xi|^{2}\,},\\
	\lambda_{2} &= -\frac{1}{2\varepsilon^{2}}
		- \frac{1}{2\varepsilon^{2}}
		  \sqrt{\,1 - 4\varepsilon^{2}|\xi|^{2}\,}.
\end{align*}

        \noindent
    \textbf{Incompressible part}.
	Considering the linearized system of \eqref{Main:com}, we have the following coupled two-phase system
	\begin{equation}\label{Lmain2}
		\left\{
		\begin{aligned}
			&\partial_t (\mathcal{P} w^\varepsilon)+\frac{1}{\var^2}(\mathcal{P} w^\varepsilon-\varepsilon u^\varepsilon)=0,\\
			&\partial_t u^\varepsilon-\mu\Delta u^\var+\frac{1}{\var}(\varepsilon u^\varepsilon-\mathcal{P}w^\varepsilon)=0,\\
			&(\mathcal{P}w^\varepsilon,u^\varepsilon)(0,x)=(\mathcal{P}w^\varepsilon_0,u_0^\varepsilon).
		\end{aligned}
		\right.
	\end{equation}
	Let $G_2(t,x)$ be the Green function of \eqref{Lmain2}. After taking the Fourier transform in $x$, we have
	$$
	\partial_t\hat{G}_2(t,\xi)+B_2\hat{G}_2(t,\xi)=0,\quad \hat{G}_2(0,\xi)={\rm{I}},
	$$
	where 
	\begin{equation}
		B_2=\left(\begin{array}{cccccc}
			\frac{1}{\varepsilon^2} {\rm{I}}&- \frac{1}{\varepsilon} 	{\rm{I}}\\
			-\frac{1}{\varepsilon}{\rm{I}}& (1+\mu|\xi|^2){\rm{I}}\\
		\end{array}\right).
	\end{equation}
	The eigenvalues of $B_2$ are computed from the determinant
	\begin{equation}
		\det{(\lambda {\rm{I}}+B_2)}=\Big(\lambda^2+(\frac{1}{\varepsilon^2}+1+\mu|\xi|^2)\lambda+\frac{\mu|\xi|^2}{\varepsilon^2}\Big)^d=0,
	\end{equation}
	which gives two eigenvalues with multiplicity $d$:
	\begin{align*}
		&\lambda_3=-\frac{1}{2}(\frac{1}{\varepsilon^2}+1+\mu|\xi|^2)+\frac{1}{2}\sqrt{(1+\frac{1}{\varepsilon^2})^2-\frac{2\mu|\xi|^2}{\varepsilon^2}+2\mu|\xi|^2+\mu^2|\xi|^4} ,\\
		&\lambda_4=-\frac{1}{2}(\frac{1}{\varepsilon^2}+1+\mu|\xi|^2)-\frac{1}{2}\sqrt{(1+\frac{1}{\varepsilon^2})^2-\frac{2\mu|\xi|^2}{\varepsilon^2}+2\mu|\xi|^2+\mu^2|\xi|^4}.
	\end{align*}
	
	Firstly, through a simple calculation, we investigate the asymptotic behavior of the eigenvalues $\lambda_i (i=1,2,3,4)$ as follows:  
    
\noindent For the low-frequency part $|\xi|\ll\frac{1}{\varepsilon}$, the eigenvalues satisfy
		\begin{equation*}
			\begin{aligned}
				&\lambda_1=-|\xi|^2+\varepsilon^2\mathcal{O}(|\xi|^4),\\
				&\lambda_2=-\frac{1}{\varepsilon^2}+|\xi|^2+\varepsilon^2\mathcal{O}(|\xi|^4),\\ 
				&\lambda_3=-\frac{\mu|\xi|^2}{1+\varepsilon^2}+\frac{\varepsilon^2}{1+\varepsilon^2} \mathcal{O}(|\xi|^4),\\
				&\lambda_4=-(1+\frac{1}{\varepsilon^2})-\frac{\mu\varepsilon^2}{1+\varepsilon^2}|\xi|^2+\frac{\varepsilon^2}{1+\varepsilon^2}\mathcal{O}(|\xi|^4).
			\end{aligned}
		\end{equation*}
		For the high-frequency part  $|\xi|\gg\frac{1}{\varepsilon}$, we have
		\begin{align*}
			&\lambda_1=-\frac{1}{2\varepsilon^2}+ {\rm i}\frac{|\xi|}{\varepsilon}+\varepsilon^{-3}\mathcal{O}(|\xi|^{-1}),\\
			&\lambda_2=-\frac{1}{2\varepsilon^2}- {\rm i}\frac{|\xi|}{\varepsilon}+\varepsilon^{-3}\mathcal{O}(|\xi|^{-1}),\\
			&\lambda_3=-\frac{1}{\varepsilon^2}+(1+\frac{1}{\varepsilon^2})^2\mathcal{O}(|\xi|^{-2}),\\
			&\lambda_4=-1-\mu|\xi|^2+(1+\frac{1}{\varepsilon^2})^2\mathcal{O}(|\xi|^{-2}),
		\end{align*}
        with ${\rm i}=\sqrt{-1}$.

The spectral analysis reveals that the threshold between low-frequency and high-frequency regimes should be $\mathcal{O}(\var^{-1})$. 
At low frequencies, the modes associated with $\lambda_1$ and $\lambda_3$ exhibit a heat–type behavior, which is consistent with the singular limit equation \eqref{main2}. In contrast, the high-frequency modes decay at exponential rates. The corresponding frequency splitting is illustrated in Figures~\ref{fig:2} and \ref{fig:3}. These observations naturally lead us to adopt the Littlewood–Paley decomposition and hybrid Besov spaces to capture the sharp regularity structure of the system.

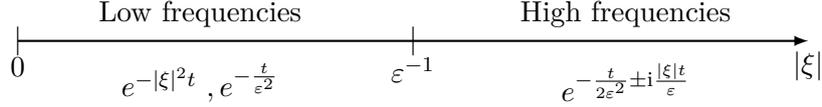
\begin{figure}[!ht]
\centering
\begin{tikzpicture}[xscale=0.4,yscale=0.4, thick]
\draw [-latex] (0,0) -- (26,0) ;
\draw  (26,0)  node [below]  {$|\xi|$} ;
\draw  (13,-0.2)  node [below]  {$\var^{-1}$};
\draw  (0,-0.2)  node [below]  {$0$} ;

\draw (13,0) node {$|$};
\draw (0,0) node {$|$};

\draw (20,0.1) node [above] {High frequencies};
\draw (20,-0.4) node [below] {$e^{-\frac{t}{2\var^2}\pm {\rm i}\frac{|\xi|t}{\var}}$};

\draw (6,0.1) node [above] {Low frequencies};
\draw (6,-0.4) node [below] {$ e^{-|\xi|^2t}\,\,, e^{-\frac{t}{\var^2}} $};
\end{tikzpicture}
\caption{Frequency partition for compressible modes}
\label{fig:2}
\end{figure}

\begin{figure}[!ht]
\centering
\begin{tikzpicture}[xscale=0.4,yscale=0.4, thick]
\draw [-latex] (0,0) -- (26,0) ;
\draw  (26,0)  node [below]  {$|\xi|$} ;
\draw  (13,-0.2)  node [below]  {$\var^{-1}$};
\draw  (0,-0.2)  node [below]  {$0$} ;

\draw (13,0) node {$|$};
\draw (0,0) node {$|$};

\draw (20,0.1) node [above] {High frequencies};
\draw (20,-0.4) node [below] {$e^{-\frac{t}{\var^2}},\,\, e^{-(1+\mu|\xi|^2)t}$};

\draw (6,0.1) node [above] {Low frequencies};
\draw (6,-0.4) node [below] {$e^{-\frac{\mu|\xi|^2}{1+\var^2}},\,\, e^{-(1+\frac{1}{\var^2})t}$};
\end{tikzpicture}
\caption{Frequency partition for incompressible modes}
\label{fig:3}
\end{figure}
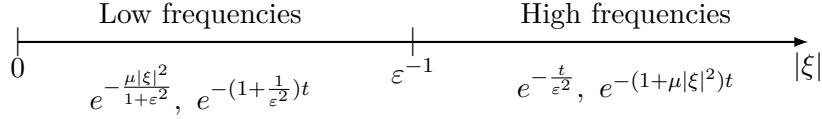

\vspace{-2mm}

	\subsubsection{Key ingredients}
We now outline the main strategies in the proofs of Theorems~\ref{thm1}–\ref{thm3}.  
A crucial step in the global existence analysis for the Cauchy problem 
\eqref{main1}–\eqref{main1d} is to study the evolution of the difference between 
$(\rho^\varepsilon,u^\varepsilon)$ and $(\rho^*,u^*)$. 
Compared to the solution itself, this difference enjoys better regularity properties 
and converges at an enhanced rate.  
However, a direct subtraction of the two systems is not feasible, as their 
structures do not align in a way that yields a closed equation.  
To address this structural mismatch, we introduce two auxiliary damped modes
\begin{align}\label{RT1}
	\mathcal{Z}^\varepsilon
	 = \mathcal{P}^{\top}w^\varepsilon
	   + \varepsilon(1+h(a^\varepsilon))\nabla a^\varepsilon,
	\qquad
	\mathcal{R}^\varepsilon
	 = \mathcal{P}w^\varepsilon - \varepsilon u^\varepsilon,
\end{align}
which satisfy
\begin{align*}
	\mathcal{Z}^\varepsilon+\mathcal{R}^\varepsilon
	&= w^\varepsilon + \varepsilon(1+h(a^\varepsilon))\nabla a^\varepsilon
	   - \varepsilon u^\varepsilon,\\
	w^\varepsilon
	&= \mathcal{Z}^\varepsilon+\mathcal{R}^\varepsilon
	   - \varepsilon(1+h(a^\varepsilon))\nabla a^\varepsilon
	   + \varepsilon u^\varepsilon.
\end{align*}
Here, $\mathcal{P}$ and $\mathcal{P}^{\top}$ are defined by \eqref{AT1}. Using the identity
\begin{align*}
	\mathcal{P}\!\left(\varepsilon a^\varepsilon (1+h(a^\varepsilon))\nabla a^\varepsilon\right)=0,
\end{align*}
we have the following key observation
\begin{align}
	\frac{1}{\varepsilon}\mathcal{P}\big(a^\varepsilon(w^\varepsilon-\varepsilon u^\varepsilon)\big)
	= \frac{1}{\varepsilon}\mathcal{P}\big(a^\varepsilon(\mathcal{Z}^\varepsilon+\mathcal{R}^\varepsilon)\big).
\end{align}
Noting that 
\begin{align*}
 Y^\var:=- \operatorname{div}
		    \big((1+a^\varepsilon)(\mathcal{Z}^\varepsilon+\mathcal{R}^\varepsilon)\big)
            =-w^\varepsilon\cdot\nabla a^\varepsilon-(1+a^\varepsilon)\text{div}w^\varepsilon-\var\Delta  a^\var+\var u^\var\cdot\nabla a^\var,
\end{align*}
we can obtain the following reformulation of \eqref{main3}:
\begin{equation}\label{main5}
	\left\{
	\begin{aligned}
		&\partial_t a^\varepsilon-\Delta a^\var
		= - u^\varepsilon\cdot \nabla a^\varepsilon
		  + \frac{1}{\varepsilon}Y^\var,\\[1mm]
		&\partial_t u^\varepsilon-\mu\Delta u^\var
		= -\mathcal{P}(u^\varepsilon\cdot\nabla u^\varepsilon)
		  + \frac{1}{\varepsilon}\mathcal{P}\big(a^\varepsilon(\mathcal{Z}^\varepsilon+\mathcal{R}^\varepsilon)\big)
		  + \frac{1}{\varepsilon}\mathcal{R}^\varepsilon.
	\end{aligned}
	\right.
\end{equation}

Applying $\mathcal{P}$ to $\eqref{main4}_2$ yields the limit system 
\begin{equation}\label{s2}
	\left\{
	\begin{aligned}
		&\partial_t a^*-\Delta a^*=-u^*\cdot\nabla a^*,\\
		&\partial_t u^*-\mu\Delta u^*=-\mathcal{P}(u^*\cdot\nabla u^*),\\
		&\operatorname{div}u^*=0.
	\end{aligned}
	\right.
\end{equation}
Introducing the differences
\begin{align}\label{TR1}
	\delta a=a^\varepsilon-a^*,\qquad
	\delta u=u^\varepsilon-u^*,
\end{align}
we obtain
\begin{equation}\label{main6}
	\left\{
	\begin{aligned}
		&\partial_t (\delta a)-\Delta (\delta a)
		= - u^\varepsilon\cdot \nabla \delta a
		  - \delta u\cdot\nabla a^*
		  + \frac{1}{\varepsilon}Y^\varepsilon,\\[1mm]
		&\partial_t (\delta u)-\mu\Delta (\delta u)
		= -\mathcal{P}\big(u^\varepsilon\cdot\nabla \delta u
		   +\delta u\cdot\nabla u^*\big)
		  + \frac{1}{\varepsilon}\mathcal{P}\big(a^\varepsilon
		    (\mathcal{Z}^\varepsilon+\mathcal{R}^\varepsilon)\big)
		  + \frac{1}{\varepsilon}\mathcal{R}^\varepsilon,
	\end{aligned}
	\right.
\end{equation}
where $\mathcal{Z}^\varepsilon$ and $\mathcal{R}^\varepsilon$ satisfy the following damped equations
\begin{equation}\label{T8}
	\left\{
	\begin{aligned}
		&\partial_t\mathcal{Z}^\varepsilon
		 +\frac{1}{\varepsilon^2}\mathcal{Z}^\varepsilon
		 = \varepsilon(1+h(a^\varepsilon))\nabla\partial_ta^\varepsilon
		   + \varepsilon h'(a^\varepsilon)\nabla a^\varepsilon\partial_ta^\varepsilon
		   -\frac{1}{\varepsilon}\mathcal{P}^{\top}(w^\varepsilon\cdot\nabla w^\varepsilon),\\[1mm]
		&\partial_t\mathcal{R}^\varepsilon
		  +\Big(1+\frac{1}{\varepsilon^2}\Big)\mathcal{R}^\varepsilon
		  = -\varepsilon\mu\Delta u^\var
		    +\varepsilon\mathcal{P}(u^\varepsilon\cdot\nabla u^\varepsilon)
		    -\mathcal{P}(a^\varepsilon(\mathcal{Z}^\varepsilon+\mathcal{R}^\varepsilon))
		    -\frac{1}{\varepsilon}\mathcal{P}(w^\varepsilon\cdot\nabla w^\varepsilon).
	\end{aligned}
	\right.
\end{equation}
The above reformulation exhibits a clear structure:  
two modes $(\delta a,\delta u)$ correspond to the diffusive behavior of the limit KS–NS system,  
while the remaining two modes $(\mathcal{Z}^\var,\mathcal{R}^\var)$ are rapidly damped at a rate $\varepsilon^{-2}$.  
This combination of diffusion and fast relaxation occurs naturally in the regime $2^j\lesssim \varepsilon^{-1}$, which is consistent with 
the spectral analysis in Section~\ref{subsection:sp}.  


These spectral features motivate the construction of a parameter-dependent 
energy functional $\delta X(t)$ (see \eqref{Fn1}) adapted to the Littlewood-Paley decomposition and allow 
the original and difference systems to be handled within a unified framework. By energy-dissipation estimates of the diffusion error equations \eqref{main6} and the damped equations \eqref{T8}, we are able to bound   
$(\delta u,\mathcal{P}w^\varepsilon,\mathcal{R}^\varepsilon)$ (Proposition~\ref{P3.2}) and then obtain $\dot{B}_{2,1}^{d/2}$ estimates for  
$(\delta a,\mathcal{P}^{\top}w^\varepsilon,\mathcal{Z}^\varepsilon)$  
(Proposition~\ref{P3.3}). Observe that the threshold $2^{J_\var}=\mathcal{O}(\var^{-1})$ (see \eqref{Jvar}) allows us to absorb higher-order terms in $\eqref{main6}_2$ and $\eqref{T8}_2$ by the dissipation bounds. In the high-frequency regime, we obtain the higher-order $\dot{B}_{2,1}^{d/2+1}$ estimates by constructing high-frequency Lyapunov functional inequalities (cf. Lemma~\ref{P3.4}). It should be noted that, to overcome the loss of derivatives, we need to consider the original system for the solution  $(a^\varepsilon,w^\varepsilon)$ and employ hyperbolic symmetrization and commutator estimates in high frequencies. Combining these results with the definition of $\delta X(t)$ yields
\[
   \delta X(t) \lesssim E_0 + \delta E_0 
      + \varepsilon\,\delta X(t) + (\delta X(t))^2.
\]
If $E_0$, $\delta E_0$ and $\var$ are suitably chosen, then  
$\delta X(t)$ remains uniformly bounded, which allows us to close the 
\emph{a priori} estimate of the error. With the help of the global well-posedness of the limit system \eqref{main2} in Proposition \ref{thm1}, we are able to recover the desired regularity of the solution and complete the proof of 
Theorem~\ref{thm2}.

  
For the large-time behavior in Theorem~\ref{thm3}, we employ a purely energy-based approach, independent of any detailed spectral analysis. To this end, we introduce a time-weighted functional $\delta\mathcal{D}(t)$ (see~\eqref{deltaDY}), involving Besov norms of higher regularity than those appearing in $\delta X(t)$. Optimal decay for $\delta\mathcal{D}(t)$ then follows from interpolation inequalities together with uniform bounds on the differences in $\dot{B}^{\sigma_1}_{2,\infty}$. Moreover, owing to the intrinsic damping of $(\mathcal{R}^\varepsilon,\mathcal{Z}^\varepsilon)$ in~\eqref{T8}, we also obtain improved decay rates, and in particular, an enhanced decay for the relative velocity. Furthermore, a key observation is that the density error $\delta a$ lies in the lower-regularity space $\dot{B}^{\sigma_1-1}_{2,\infty}$, which indicates that the rate of $\delta \rho$ in $\dot{B}^{\sigma}_{2,1}$ can be improved to $\var(1+t)^{-\frac{1}{2}(\sigma-\sigma_1+1)}$ ($\sigma_1<\sigma\leq d/2-1$) by analyzing the modified time-weighted functional $\delta\mathcal{D}^*(t)$ given in \eqref{DT*}.

Our argument is partly inspired by~\cite{MR4970471,MR4596744}. The main novelty here lies in time weighted energy estimates that capture the enhanced decay \emph{at the level of the error system rather than the original system itself}. By combining these refined bounds for the error with the known decay of the limit system, we eventually recover the optimal decay rates for the solutions to the original E-NS system, and thus conclude the proof of Theorem~\ref{thm3}.


            \subsection{Outline of the paper}

In Section \ref{S2}, we recall the Littlewood-Paley decomposition, Besov spaces and some usual analytic tools. In Section \ref{S4}, we mainly establish the global well-posedness of the system \eqref{main1} in Besov space, where the key issue lies in the construction of the energy functional. In Section \ref{S5}, we aim to obtain the large-time behavior of the solutions to the system \eqref{main1}. Due to the different decay rates of the low and high frequencies of the solutions in Besov space, we establish the time-decay estimates by using the pure energy method instead of the complicated spectral analysis method. In Appendix A, we provide the proof of Proposition \ref{thm1}.

\vspace{2mm}

Throughout the paper, we write $f_{1}\lesssim f_{2}$ if there exists a constant
$C>0$, independent of the functions under consideration, such that
$f_{1}\leq C\,f_{2}$. We write $f_{1}\sim f_{2}$ if $f_{1}\lesssim f_{2}$ and
$f_{2}\lesssim f_{1}$. For any Banach spaces $X$ and $Y$, we set $\|(f,g)\|_{X}:=\|f\|_{X}+\|g\|_{X}$ and $\|f\|_{X\cap Y}:=\|f\|_{X}+\|f\|_{Y}$. The sum space $X+Y$ is defined by $
X+Y:=\{\,f=f_{1}+f_{2}\mid f_{1}\in X,\ f_{2}\in Y\,\}$ equipped with the norm   $\|f\|_{X+Y}:=\inf_{f=f_{1}+f_{2}}\bigl(\|f_{1}\|_{X}+\|f_{2}\|_{Y}\bigr)$.

	\section{Preliminaries}\label{S2}

    	\subsection{Functional spaces}
        
	In this section, we present functional spaces and some useful lemmas.

We briefly recall the Littlewood-Paley decomposition and the associated Besov
spaces, and refer the readers to \cite[Chapter~2]{bahouri1} for a systematic
presentation. Let $\chi\in\mathcal{C}^\infty(\mathbb{R}^d)$ be radial, equal to
$1$ on $B(0,3/4)$ and supported in $B(0,4/3)$, and set
\[
\varphi(\xi):=\chi(\xi/2)-\chi(\xi),\qquad
\sum_{j\in\mathbb{Z}}\varphi(2^{-j}\cdot)\equiv 1,
\qquad 
\supp\, \varphi\subset\{\tfrac34\le|\xi|\le\tfrac83\}.
\]
For $j\in\mathbb{Z}$ we then define the homogeneous dyadic blocks and the
low-frequency cut-off by
\[
\dot\Delta_j u := \mathcal{F}^{-1}\!\big(\varphi(2^{-j}\cdot)\mathcal{F}u\big),
\qquad
\dot S_j u := \mathcal{F}^{-1}\!\big(\chi(2^{-j}\cdot)\mathcal{F}u\big),
\]
and write $u_j:=\dot\Delta_j u$ for simplicity.

Let $\mathcal{S}'_h$ denote the subspace of tempered distributions $u$ such that
$u\in\mathcal{S}'$ and $\lim_{j\to-\infty}\|\dot S_j u\|_{L^\infty}=0$. Then
\[
u=\sum_{j\in\mathbb{Z}}u_j,
\qquad
\dot S_j u=\sum_{j'\le j-1}u_{j'},\quad\text{in }\mathcal{S}'_h.
\]

The homogeneous Besov space $\dot B^s_{p,r}$ ($s\in\mathbb{R}$,
$1\le p,r\le\infty$) is defined by
\[
\dot B^s_{p,r}:=
\Big\{u\in\mathcal{S}'_h:\ 
\|u\|_{\dot B^s_{p,r}}
:=\big\|\{2^{js}\|u_j\|_{L^p}\}_{j\in\mathbb{Z}}\big\|_{\ell^r}<\infty\Big\}.
\]
Similarly, the Chemin-Lerner-type norm on $(0,T)$ is given by
\[
\widetilde L^\varrho(0,T;\dot B^s_{p,r})
:=\Big\{u\in L^\varrho(0,T;\mathcal{S}'_h):\
\|u\|_{\widetilde L^\varrho_T(\dot B^s_{p,r})}
:=\big\|\{2^{js}\|u_j\|_{L^\varrho_T(L^p)}\}\big\|_{\ell^r}<\infty\Big\}.
\]
By Minkowski’s inequality,
\[
\|u\|_{\widetilde L^\varrho_T(\dot B^s_{p,r})}
\le \|u\|_{L^\varrho_T(\dot B^s_{p,r})}\quad\text{if } r\ge\varrho,
\qquad
\|u\|_{\widetilde L^\varrho_T(\dot B^s_{p,r})}
\ge \|u\|_{L^\varrho_T(\dot B^s_{p,r})}\quad\text{if } r\le\varrho.
\]

To separate low and high frequencies, we fix an integer $J_\varepsilon$
(called the threshold) and set for $s\in\mathbb{R}$
\[
\|u\|_{\dot B^s_{p,r}}^{\ell,\varepsilon}
:=\big\|\{2^{js}\|u_j\|_{L^p}\}_{j\le J_\varepsilon}\big\|_{\ell^r},
\qquad
\|u\|_{\dot B^s_{p,r}}^{h,\varepsilon}
:=\big\|\{2^{js}\|u_j\|_{L^p}\}_{j\ge J_\varepsilon-1}\big\|_{\ell^r},
\]
and define $\widetilde L^\varrho_T(\dot B^s_{p,r})^{\ell,\varepsilon}$,
$\widetilde L^\varrho_T(\dot B^s_{p,r})^{h,\varepsilon}$ in the same manner.
For $u\in\mathcal{S}'_h$ let
\[
u^{\ell,\varepsilon}:=\sum_{j\le J_\varepsilon-1}u_j,
\qquad
u^{h,\varepsilon}:=\sum_{j\ge J_\varepsilon}u_j.
\]
Then for any $s'>0$,
\begin{equation}\label{lh}
\|u^{\ell,\varepsilon}\|_{\dot B^s_{p,r}}\le
\|u\|_{\dot B^s_{p,r}}^{\ell,\varepsilon}
\le 2^{J_\varepsilon s'}\|u\|_{\dot B^{s-s'}_{p,r}}^{\ell,\varepsilon},
\qquad
\|u^{h,\varepsilon}\|_{\dot B^s_{p,r}}
\le \|u\|_{\dot B^s_{p,r}}^{h,\varepsilon}
\le 2^{-(J_\varepsilon-1)s'}\|u\|_{\dot B^{s+s'}_{p,r}}^{h,\varepsilon}.
\end{equation}
Throughout the paper, we choose \vspace{-2mm}
\begin{equation}\label{Jvar}
J_\varepsilon:=-[\log_2\varepsilon]-m_0,
\end{equation}
so that $2^{J_\varepsilon}\sim\varepsilon^{-1}$, where $m_0$ is a sufficiently
large integer independent of $\varepsilon$ (see \eqref{m0}).

	We collect some useful technical lemmas that will be used frequently in our analysis. The first lemma concerns the Bernstein inequalities (see \cite[Lemma 2.1]{bahouri1}).
	\begin{lemma}\label{lemma21}
		Let $0<r<R$, $1\leq p\leq q\leq \infty$ and $k\in \mathbb{N}$. For any $u\in L^p$ and $\lambda>0$, it holds
		\begin{equation}\nonumber
			\left\{
			\begin{aligned}
				&{\rm{Supp}}~ \mathcal{F}(u) \subset \{\xi\in\mathbb{R}^{d}~| ~|\xi|\leq \lambda R\}\Rightarrow \|D^{k}u\|_{L^q}\lesssim\lambda^{k+d(\frac{1}{p}-\frac{1}{q})}\|u\|_{L^p},\\
				&{\rm{Supp}}~ \mathcal{F}(u) \subset \{\xi\in\mathbb{R}^{d}~|~ \lambda r\leq |\xi|\leq \lambda R\}\Rightarrow \|D^{k}u\|_{L^{p}}\sim\lambda^{k}\|u\|_{L^{p}}.
			\end{aligned}
			\right.
		\end{equation}
	\end{lemma}

	The Besov spaces have the following properties due to the Bernstein inequalities. See \cite[Pages 64-65, 79]{bahouri1} for more details.
	\begin{lemma}\label{lemma22}
  The following statements hold:
		\begin{itemize}
			\item{} For $s\in\mathbb{R}$, $1\leq p_{1}\leq p_{2}\leq \infty$ and $1\leq r_{1}\leq r_{2}\leq \infty$, it holds
			\begin{equation}\nonumber
				\begin{aligned}
					\dot{B}^{s}_{p_{1},r_{1}}\hookrightarrow \dot{B}^{s-d(\frac{1}{p_{1}}-\frac{1}{p_{2}})}_{p_{2},r_{2}};
				\end{aligned}
			\end{equation}
			\item{} For $1\leq p\leq q\leq\infty$, we have the following chain of continuous embeddings:
			\begin{equation}\nonumber
				\begin{aligned}
					\dot{B}^{0}_{p,1}\hookrightarrow L^{p}\hookrightarrow \dot{B}^{0}_{p,\infty}\hookrightarrow \dot{B}^{\sigma}_{q,\infty},\quad \sigma=-d(\frac{1}{p}-\frac{1}{q})<0;
				\end{aligned}
			\end{equation}
			\item{} If $1\leq p<\infty$, then $\dot{B}^{\frac{d}{p}}_{p,1}$ is continuously embedded in the set of continuous functions that decay to 0 at infinity;
			\item{} The following real interpolation property is satisfied for $1\leq p\leq\infty$, $s_{1}<s_{2}$ and $\theta\in(0,1)$:
			\begin{equation}
				\begin{aligned}
					&\|u\|_{\dot{B}^{\theta s_{1}+(1-\theta)s_{2}}_{p,1}}\lesssim \frac{1}{\theta(1-\theta)(s_{2}-s_{1})}\|u\|_{\dot{B}^{ s_{1}}_{p,\infty}}^{\theta}\|u\|_{\dot{B}^{s_{2}}_{p,\infty}}^{1-\theta}.\label{inter}
				\end{aligned}
			\end{equation}
			\item{}
			Let $\Lambda^{\sigma}$ be defined by $\Lambda^{\sigma}=(-\Delta)^{\frac{\sigma}{2}}u:=\mathcal{F}^{-}\big{(} |\xi|^{\sigma}\mathcal{F}(u) \big{)}$ for $\sigma\in \mathbb{R}$ and $u\in\dot{S}^{'}_{h}$. Then $\Lambda^{\sigma}$ is an isomorphism from $\dot{B}^{s}_{p,r}$ to $\dot{B}^{s-\sigma}_{p,r}$.
		\end{itemize}
	\end{lemma}

	The following Morser-type product estimates in Besov spaces play a fundamental role in the analysis of nonlinear terms (cf. \cite[Page 90]{bahouri1}) and \cite[Section 4.4]{MR1419319}).
	\begin{lemma}\label{lemma23}
		The following statements hold:
		\begin{itemize}
			\item{} Let $s>0$ and $1\leq p,r\leq \infty$. Then $\dot{B}^{s}_{p,r}\cap L^{\infty}$ is an algebra and
			\begin{equation}\label{uv1}
				\begin{aligned}
					\|g h\|_{\dot{B}^{s}_{p,r}}\lesssim \|g\|_{L^{\infty}}\|h\|_{\dot{B}^{s}_{p,r}}+ \|h\|_{L^{\infty}}\|g\|_{\dot{B}^{s}_{p,r}};
				\end{aligned}
			\end{equation}
			\item{}
			Let $s_{1}$, $s_{2}$ and $p$ satisfy $2\leq p\leq \infty$, $s_{1}\leq \frac{d}{p}$, $s_{2}\leq \frac{d}{p}$ and $s_{1}+s_{2}>0$. Then we have
			\begin{equation}\label{uv2}
				\begin{aligned}
					&\|g h\|_{\dot{B}^{s_{1}+s_{2}-\frac{d}{p}}_{p,1}}\lesssim \|g\|_{\dot{B}^{s_{1}}_{p,1}}\|h\|_{\dot{B}^{s_{2}}_{p,1}};
				\end{aligned}
			\end{equation}
			\item{} Assume that $s_{1}$, $s_{2}$ and $p$ satisfy $2\leq p\leq \infty$, $s_{1}\leq \frac{d}{p}$, $s_{2}<\frac{d}{p}$ and $s_{1}+s_{2}\geq0$. Then it holds 
			\begin{equation}\label{uv3}
				\begin{aligned}
					&\| g h\|_{\dot{B}^{s_{1}+s_{2}-\frac{d}{p}}_{p,\infty}}\lesssim \|g\|_{\dot{B}^{s_{1}}_{p,1}}\|h\|_{\dot{B}^{s_{2}}_{p,\infty}}.
				\end{aligned}
			\end{equation}
		\end{itemize}
	\end{lemma}

	We state the following result about the continuity of composition functions, which can be found in \cite[Page 104]{bahouri1}.
	\begin{lemma}\label{lemma24}
		Let $G:I\rightarrow \mathbb{R}$ be a smooth function satisfying $G(0)=0$. For any $1\leq p,r\leq \infty$, $s>0$ and $g\in\dot{B}^{s}_{p,r}\cap L^{\infty}$, there holds $G(g)\in \dot{B}^{s}_{p,r}\cap L^{\infty}$ and
		\begin{align}
			\|G(g)\|_{\dot{B}^{s}_{p,r}}\leq C_{g}\|g\|_{\dot{B}^{s}_{p,r}},\label{F1}
		\end{align}
		where the constant $C_{g}>0$ depends only on $\|g\|_{L^{\infty}}$, $\|G'\|_{L^\infty}$, $s$ and $d$.
		
		
	\end{lemma}

	Finally, the following commutator estimates will be useful  in high frequencies (see \cite[Page 112]{bahouri1}).
	\begin{lemma}\label{lemma25}
		Let $1\leq p\leq \infty$ and $-\frac{d}{p}\leq s\leq \frac{d}{p}+1$. Then it holds
		\begin{align}
			&\sum_{j\in\mathbb{Z}}2^{js}\|[g ,\dot{\Delta}_{j}]\partial_{k} h\|_{L^{p}}\lesssim\|g\|_{\dot{B}^{\frac{d}{p}+1}_{p,1}}\|h\|_{\dot{B}^{s}_{p,1}},\quad k=1,...,d,\label{commutator1}
		\end{align}
		with the commutator $[A,B]:=AB-BA$.
	\end{lemma}

We also state the optimal regularity estimates (see \cite[Page 157]{bahouri1}) for the heat equation in $\mathbb{R}^{d}$:
\begin{equation}\label{heat}
\begin{aligned}
&\partial_{t}g-\mu\Delta g=G,\quad g(0,x)=g_{0}(x).
\end{aligned}
\end{equation}
\begin{lemma}\label{lemma33}
Let $T>0$, $\mu>0$, $s\in\mathbb{R}$ and $1\leq p,r\leq\infty$. Assume that $u_{0}\in\dot{B}^{s}_{p,r}$ and $G\in L^1(0,T;\dot{B}^{s}_{p,r})$. Then the solution $g$ to \eqref{heat} satisfies
\begin{equation*}
\|g\|_{\widetilde{L}^{\infty}_{T}(\dot{B}^{s}_{p,r})}+\|g\|_{\widetilde{L}^{1}_t(\dot{B}^{s+2}_{p,r})}+\|\partial_t g\|_{\widetilde{L}^{1}_T(\dot{B}^{s}_{p,r})}\lesssim \|g_{0}\|_{\dot{B}^{s}_{p,r}}+\|G\|_{\widetilde{L}^{1}_{T}(\dot{B}^{s}_{p,r})}.
\end{equation*}
\end{lemma}

For the  damped equation in $\mathbb{R}^{d}$:
\begin{equation}\label{damped}
\begin{aligned}
&\partial_t g+\frac{1}{\var^2} g =F,\quad g(0, x)=g_0(x),
\end{aligned}
\end{equation}
it is easy to prove the following lemma.
\begin{lemma}\label{maximaldamped}
Let $s\in\mathbb{R}$, $1\leq  p,r\leq\infty$, and let $T>0$ be a given time. Assume $g_{0}\in\dot{B}^{s}_{p,r}$ and  $F\in  \widetilde{L}^{1}_T(\dot{B}^{s}_{p,r})$. If $g$ is a solution to the Cauchy problem \eqref{damped}, then it holds that
\begin{equation}\nonumber
\begin{aligned}
&\|g\|_{\widetilde{L}^{\infty}_{T}(\dot{B}^{s}_{p,r})}+\frac{1}{\var^2}\|g\|_{\widetilde{L}^{1}_{T}(\dot{B}^{s}_{p,r})}
\lesssim 
\|g_{0}\|_{\dot{B}^{s}_{p,1}}+\|F\|_{\widetilde{L}^{1}_{T}(\dot{B}^{s}_{p,r})}.
\end{aligned}
\end{equation}
\end{lemma}


	\section{Proof of Theorem \ref{thm2}}\label{S4}
    
	This section is devoted to the study of the global well-posedness of the system \eqref{main1} in Besov space (Theorem \ref{thm2}). Since local well-posedness can be proven by standard methods, we omit the details here. To extend the local solution to a global one, the key issue is to establish the {\emph{a priori}} estimate. 
	
	To this end, we introduce the error functional as follows:
	\begin{align}\label{Fn1}
		\delta X(t)
		&=\frac{1}{\varepsilon}\|\delta a\|_{\widetilde{L}_t^\infty(\dot{B}_{2,1}^{\frac{d}{2}-1})}+\frac{1}{\varepsilon}\|\delta a\|_{\widetilde{L}_t^1(\dot{B}_{2,1}^{\frac{d}{2}+1})}+\frac{1}{\varepsilon}\|\partial_t\delta a \|_{\widetilde{L}_t^1(\dot{B}_{2,1}^{\frac{d}{2}-1})} \nonumber\\
        &\quad +\frac{1}{\varepsilon}\|\delta u\|_{\widetilde{L}_t^\infty(\dot{B}_{2,1}^{\frac{d}{2}-1})}+\frac{1}{\varepsilon}\|\delta u\|_{\widetilde{L}_t^1(\dot{B}_{2,1}^{\frac{d}{2}+1})}+\frac{1}{\varepsilon}\|\partial_t\delta u\|_{\widetilde{L}_t^1(\dot{B}_{2,1}^{\frac{d}{2}-1})} \nonumber\\
		&\quad+\|\mathcal{P}w^\varepsilon\|_{ \widetilde{L}_t^\infty(\dot{B}_{2,1}^{\frac{d}{2}-1})}+\|w^\varepsilon\|_{\widetilde{L}^{\infty}_t(\dot{B}_{2,1}^{\frac{d}{2}})}^{\ell,\var}+\frac{1}{\varepsilon}\|w^\varepsilon\|_{\widetilde{L}^2_t(\dot{B}_{2,1}^{\frac{d}{2}})}^{\ell,\var}+\frac{1}{\varepsilon}\|w^\varepsilon\|_{\widetilde{L}_t^1(\dot{B}_{2,1}^{\frac{d}{2}+1})}^{\ell,\var}\nonumber\\
        &\quad+\varepsilon\|( a^\varepsilon, w^\varepsilon)\|_{\widetilde{L}_t^\infty(\dot{B}_{2,1}^{\frac{d}{2}+1})}^{h,\varepsilon}+\frac{1}{\varepsilon}\|( a^\varepsilon, w^\varepsilon)\|_{\widetilde{L}_t^1(\dot{B}_{2,1}^{\frac{d}{2}+1})}^{h,\varepsilon}\nonumber\\
        &\quad
		+\frac{1}{\varepsilon^2}\|\mathcal{Z}^\varepsilon\|_{\widetilde{L}_t^1(\dot{B}_{2,1}^{\frac{d}{2}})}+\frac{1}{\varepsilon^2}\|\mathcal{R}^\varepsilon\|_{\widetilde{L}_t^1(\dot{B}_{2,1}^{\frac{d}{2}-1}\cap \dot{B}_{2,1}^{\frac{d}{2}})}.
	\end{align}
	Then, our task is to prove the uniform boundedness of the energy functional $\delta X(t)$.  At the beginning, we provide the following \emph{a priori} assumption that 
	\begin{align}\label{asss}
		\delta X(t)\leq \zeta_1,
	\end{align}
	where $\zeta_1$ is  a small positive constant determined in \eqref{var1}. 
	
	\begin{prop}\label{P31}
		Let $(\rho^*, u^*)$ be the global solution of the Cauchy problem \eqref{main2}-\eqref{main2d} with the initial data $(\rho^*_0,u_0^*)$. There exist two constants $\zeta_1,\var^*$ independent of $\var$ such that if \eqref{asss} and $\var\leq \var^*$ hold, then it holds 
		\begin{align}
			\delta X(t)\leq C_0( E_0+\delta E_0),\label{3.3}
		\end{align}
	with a generic constant $C_0>0$. Here $E_0$ and $\delta E_0$ are given by \eqref{a0} and \eqref{deltaE0} respectively. 
	\end{prop}
	
	To complete the proof of Proposition \ref{P31}, we divide the proof into the following three steps.
	\subsection{Step 1: Uniform estimates of \texorpdfstring{$(\delta u,\mathcal{P}w^\varepsilon,\mathcal{R}^\varepsilon)$}{(delta u, P w^epsilon, R^epsilon)}}
	In this subsection, we mainly establish the uniform estimates of $(\delta u,\mathcal{P}w^\varepsilon,\mathcal{R}^\varepsilon)$. At the beginning, we consider the system of $\delta u$ and $\mathcal{R}^\var$:
	\begin{equation}\label{main7}
		\left\{
		\begin{aligned}
			&\partial_t (\delta u)-\mu\Delta (\delta u)=-\mathcal{P}(u^\varepsilon\cdot\nabla \delta u+\delta u\cdot\nabla u^*)+\frac{1}{\var}\mathcal{P}(a^\var(\mathcal{Z}^\var+\mathcal{R}^\var))+\frac{1}{\varepsilon}\mathcal{R}^\varepsilon,\\
			&	\partial_t\mathcal{R}^\varepsilon+(1+\frac{1}{\varepsilon^2})\mathcal{R}^\varepsilon=-\varepsilon\mu\Delta u^{\varepsilon}+\varepsilon\mathcal{P}(u^\varepsilon\cdot\nabla u^\varepsilon)-\mathcal{P}(a^\var(\mathcal{Z}^\var+\mathcal{R}^\var))-\frac{1}{\varepsilon}\mathcal{P}(w^\varepsilon\cdot\nabla w^\varepsilon),\\
            &(\delta u, \mathcal{R}^\varepsilon)(0,x)=(\delta u_0, \mathcal{R}_0^\varepsilon)(x):=(u_0^\var-u_0, \mathcal{P} w_0^\var-\var u_0^\var)(x).
		\end{aligned}
		\right.
	\end{equation}
	\begin{lemma}\label{P3.2}
		Assume that $(a^\varepsilon, w^\varepsilon,u^\varepsilon)$ is a solution to the system \eqref{main3}. Then it holds
		\begin{align}
			&\frac{1}{\varepsilon}\|\delta u\|_{{\widetilde{L}_t^\infty(\dot{B}_{2,1}^{\frac{d}{2}-1})}}+\frac{1}{\varepsilon}
			\|\delta u\|_{{\widetilde{L}_t^1(\dot{B}_{2,1}^{\frac{d}{2}+1})}}+\|\mathcal{P}w^\varepsilon\|_{\widetilde{L}_t^\infty(\dot{B}_{2,1}^{\frac{d}{2}-1})}
			\nonumber\\
			&\quad+\|\mathcal{P}w^\varepsilon\|_{\widetilde{L}_t^\infty(\dot{B}_{2,1}^{\frac{d}{2}})}^{\ell,\var}+\frac{1}{\var}\|\mathcal{P}w^\varepsilon\|_{{\widetilde{L}^2_t(\dot{B}_{2,1}^{\frac{d}{2}})}}^{\ell,\varepsilon}+\frac{1}{\varepsilon}
			\|\mathcal{P}w^\varepsilon\|_{{\widetilde{L}_t^1(\dot{B}_{2,1}^{\frac{d}{2}+1})}}^{\ell,\varepsilon}\nonumber\\
			&\quad +\frac{1}{\varepsilon^2}\|\mathcal{R}^\varepsilon\|_{{\widetilde{L}_t^1(\dot{B}_{2,1}^{\frac{d}{2}-1})}}+\frac{1}{\varepsilon^2}\|\mathcal{R}^\varepsilon\|_{{\widetilde{L}_t^1(\dot{B}_{2,1}^{\frac{d}{2}})}}^{\ell,\varepsilon}\nonumber\\
			&\lesssim E_0+\delta E_0+\var\delta X(t)+(\delta X(t))^2.\label{step1:es}
		\end{align}	
	\end{lemma}
	\begin{proof}
		Applying Lemma \ref{lemma33} for the equation $\eqref{main7}_1$ yields that 
		\begin{align}
			&\frac{1}{\varepsilon}\|\delta u\|_{\widetilde{L}_t^\infty(\dot{B}_{2,1}^{\frac{d}{2}-1})}+\frac{1}{\varepsilon}\|\delta u\|_{\widetilde{L}_t^1(\dot{B}_{2,1}^{\frac{d}{2}+1})}+\frac{1}{\varepsilon}\|\partial_t\delta u\|_{\widetilde{L}_t^1(\dot{B}_{2,1}^{\frac{d}{2}-1})}\nonumber\\
			&\lesssim \frac{1}{\varepsilon}\|\delta u_0\|_{\dot{B}_{2,1}^{\frac{d}{2}-1}}+\frac{1}{\varepsilon} \|(u^\varepsilon\cdot\nabla\delta u,\delta u\cdot\nabla u^*)\|_{\widetilde{L}_t^1(\dot{B}_{2,1}^{\frac{d}{2}-1})}+\frac{1}{\var^2}\|a^\var(\mathcal{Z}^\var+\mathcal{R}^\var)\|_{\widetilde{L}_t^1(\dot{B}_{2,1}^{\frac{d}{2}-1})}\nonumber\\
			&\quad+\frac{1}{\varepsilon^2}\|\mathcal{R}^\varepsilon\|_{{\widetilde{L}_t^1(\dot{B}_{2,1}^{\frac{d}{2}-1})}}.
			\label{A1}
		\end{align}	
		To control the last term on the right-hand side of \eqref{A1}, we deduce from $\eqref{main7}_2$ that 
		\begin{align}\label{A2}
			&\|\mathcal{R}^\varepsilon\|_{\widetilde{L}_t^\infty(\dot{B}_{2,1}^{\frac{d}{2}-1})}+\frac{1}{\varepsilon^2}\|\mathcal{R}^\varepsilon\|_{\widetilde{L}_t^1(\dot{B}_{2,1}^{\frac{d}{2}-1})}\nonumber\\
			&\lesssim \|\mathcal{R}^\varepsilon_0\|_{\dot{B}_{2,1}^{\frac{d}{2}-1}}+\varepsilon \|u^\varepsilon\|_{\widetilde{L}_t^1(\dot{B}_{2,1}^{\frac{d}{2}+1})}
			+\varepsilon\|u^\varepsilon\cdot\nabla u^\varepsilon\|_ {\widetilde{L}_t^1(\dot{B}_{2,1}^{\frac{d}{2}-1})}\nonumber\\
			&\quad+\|a^\var(\mathcal{Z}^\var+\mathcal{R}^\var)\|_{\widetilde{L}_t^1(\dot{B}_{2,1}^{\frac{d}{2}-1})}+\frac{1}{\varepsilon}\|w^\varepsilon\cdot\nabla w^\varepsilon\|_{\widetilde{L}_t^1(\dot{B}_{2,1}^{\frac{d}{2}-1})}.	
		\end{align}
	Multiplying \eqref{A1} by a small positive constant, then adding the resultant to \eqref{A2} yields that 
		\begin{align}
			&\frac{1}{\varepsilon}\|\delta u\|_{\widetilde{L}_t^\infty(\dot{B}_{2,1}^{\frac{d}{2}-1})}+\frac{1}{\varepsilon}\|\delta u\|_{\widetilde{L}_t^1(\dot{B}_{2,1}^{\frac{d}{2}+1})}+\frac{1}{\varepsilon}\|\partial_t\delta u\|_{\widetilde{L}_t^1(\dot{B}_{2,1}^{\frac{d}{2}-1})}\nonumber\\
            &\quad+\|\mathcal{R}^\varepsilon\|_{\widetilde{L}_t^\infty(\dot{B}_{2,1}^{\frac{d}{2}-1})}+\frac{1}{\varepsilon^2}\|\mathcal{R}^\varepsilon\|_{\widetilde{L}_t^1(\dot{B}_{2,1}^{\frac{d}{2}-1})}\nonumber\\
			&\lesssim \frac{1}{\varepsilon}\|\delta u_0\|_{\dot{B}_{2,1}^{\frac{d}{2}-1}}+\|\mathcal{R}^\varepsilon_0\|_{\dot{B}_{2,1}^{\frac{d}{2}-1}}+\frac{1}{\varepsilon} \|(u^\varepsilon\cdot\nabla\delta u,\delta u\cdot\nabla u^*)\|_{\widetilde{L}_t^1(\dot{B}_{2,1}^{\frac{d}{2}-1})}+\frac{1}{\var^2}\|a^\var(\mathcal{Z}^\var+\mathcal{R}^\var)\|_{\widetilde{L}_t^1(\dot{B}_{2,1}^{\frac{d}{2}-1})}\nonumber\\
			&\quad +\varepsilon \|u^\varepsilon\|_{\widetilde{L}_t^1(\dot{B}_{2,1}^{\frac{d}{2}+1})}
			+\varepsilon\|u^\varepsilon\cdot\nabla u^\varepsilon\|_ {\widetilde{L}_t^1(\dot{B}_{2,1}^{\frac{d}{2}-1})}+\frac{1}{\varepsilon}\|w^\varepsilon\cdot\nabla w^\varepsilon\|_{\widetilde{L}_t^1(\dot{B}_{2,1}^{\frac{d}{2}-1})}.	\label{A111}
		\end{align}	
 Recall $E_0\leq 1$ and $\var\leq 1$. The first and second terms on the right-hand side of \eqref{A111} are estimated as
		\begin{align*}
			\frac{1}{\varepsilon}\|\delta u_0\|_{\dot{B}_{2,1}^{\frac{d}{2}-1}}+\|\mathcal{R}^\varepsilon_0\|_{\dot{B}_{2,1}^{\frac{d}{2}-1}}
			&\leq \delta E_0+ \|\mathcal{P}w_0^\varepsilon\|_{\dot{B}_{2,1}^{\frac{d}{2}-1}}+\varepsilon\|\delta u_0\|_{\dot{B}_{2,1}^{\frac{d}{2}-1}}+\varepsilon\|u_0^*\|_{\dot{B}_{2,1}^{\frac{d}{2}-1}}\\
			&\leq E_0+\delta E_0.
		\end{align*}		
Then, by Lemma \ref{lemma23}, the estimate of the third term on the right-hand side of \eqref{A111}  is given by
		\begin{align}
			&\frac{1}{\varepsilon}\|(u^\varepsilon\cdot\nabla\delta u, \delta u\cdot\nabla u^*)\|_{{\widetilde{L}_t^1(\dot{B}_{2,1}^{\frac{d}{2}-1})}}\nonumber\\
			&\lesssim \frac{1}{\varepsilon}\|u^\varepsilon\|_{{\widetilde{L}_t^\infty(\dot{B}_{2,1}^{\frac{d}{2}-1})}}\|\nabla\delta u\|_{\widetilde{L}_t^1(\dot{B}_{2,1}^{\frac{d}{2}})}+\frac{1}{\varepsilon}\|\delta u\|_{\widetilde{L}_t^\infty(\dot{B}_{2,1}^{\frac{d}{2}-1})}
			\|\nabla u^*\|_{\widetilde{L}_t^1(\dot{B}_{2,1}^{\frac{d}{2}})}\nonumber\\
			&\lesssim \frac{1}{\varepsilon}(\|u^*\|_{{\widetilde{L}_t^\infty(\dot{B}_{2,1}^{\frac{d}{2}-1})}}+\|\delta u\|_{{\widetilde{L}_t^\infty(\dot{B}_{2,1}^{\frac{d}{2}-1})}})\|\delta u\|_{\widetilde{L}_t^1(\dot{B}_{2,1}^{\frac{d}{2}+1})}
			+\frac{1}{\varepsilon}\|\delta u\|_{{\widetilde{L}_t^\infty(\dot{B}_{2,1}^{\frac{d}{2}-1})}}\| u^*\|_{{\widetilde{L}_t^1(\dot{B}_{2,1}^{\frac{d}{2}+1})}}\nonumber\\
			&\lesssim E_0\delta X(t)+\varepsilon (\delta X(t))^2\nonumber\\
			&\lesssim E_0+(\delta X(t))^2.\label{090801}
		\end{align}
The fourth term on the right-hand side of \eqref{A111} satisfies
		\begin{align}
			&\frac{1}{\var^2}\|a^\var(\mathcal{Z}^\var+\mathcal{R}^\var)\|_{\widetilde{L}_t^1(\dot{B}_{2,1}^{\frac{d}{2}-1})}\nonumber\\
			&\lesssim  \frac{1}{\var^2}\|a^\var\|_{\widetilde{L}_t^\infty(\dot{B}_{2,1}^{\frac{d}{2}-1})}
			\|\mathcal{Z}^\var\|_{\widetilde{L}_t^1(\dot{B}_{2,1}^{\frac{d}{2}})}+\frac{1}{\var^2}\|a^\var\|_{\widetilde{L}_t^\infty(\dot{B}_{2,1}^{\frac{d}{2}})}\|\mathcal{R}^\var\|_{\widetilde{L}_t^1(\dot{B}_{2,1}^{\frac{d}{2}-1})}\nonumber\\
			&\lesssim  \frac{1}{\var^2}(\|a^*\|_{\widetilde{L}_t^\infty(\dot{B}_{2,1}^{\frac{d}{2}-1})}+\|\delta a\|_{\widetilde{L}_t^\infty(\dot{B}_{2,1}^{\frac{d}{2}-1})}^{\ell,\var}+\var^2\| a^\var\|_{\widetilde{L}_t^\infty(\dot{B}_{2,1}^{\frac{d}{2}+1})}^{h,\var})
			\|\mathcal{Z}^\var\|_{\widetilde{L}_t^1(\dot{B}_{2,1}^{\frac{d}{2}})}\nonumber\\
			&\quad+\frac{1}{\var^2}(\|a^*\|_{\widetilde{L}_t^\infty(\dot{B}_{2,1}^{\frac{d}{2}})}+\frac{1}{\var}\|\delta a\|_{\widetilde{L}_t^\infty(\dot{B}_{2,1}^{\frac{d}{2}-1})}^{\ell,\var}+\var\| a^\var\|_{\widetilde{L}_t^\infty(\dot{B}_{2,1}^{\frac{d}{2}+1})}^{h,\var})\|\mathcal{R}^\var\|_{\widetilde{L}_t^1(\dot{B}_{2,1}^{\frac{d}{2}-1})}\nonumber\\
			&\lesssim (E_0+\var\delta X(t)+\delta X(t))\delta X(t)\nonumber\\
			&\lesssim E_0+(\delta X(t))^2.\label{090901}
		\end{align}
By Lemma \ref{lemma23}, the remaining terms on the right-hand side of \eqref{A111} are handled as follows.
		\begin{align}\label{T2}
			&\varepsilon\|u^\varepsilon\|_{{\widetilde{L}_t^1(\dot{B}_{2,1}^{\frac{d}{2}+1})}}+\varepsilon \|u^\varepsilon\cdot\nabla u^\varepsilon\|_{{\widetilde{L}_t^1(\dot{B}_{2,1}^{\frac{d}{2}-1})}}+	\frac{1}{\varepsilon^2} \|w^\varepsilon\cdot\nabla w^\varepsilon\|_{{\widetilde{L}_t^1(\dot{B}_{2,1}^{\frac{d}{2}-1})}}\nonumber\\
			&\lesssim \varepsilon\|u^*\|_{{\widetilde{L}_t^1(\dot{B}_{2,1}^{\frac{d}{2}+1})}}+\varepsilon \|\delta u\|_{{\widetilde{L}_t^1(\dot{B}_{2,1}^{\frac{d}{2}+1})}} \nonumber\\
			&\quad+\varepsilon( \|u^*\|_{{\widetilde{L}_t^\infty(\dot{B}_{2,1}^{\frac{d}{2}-1})}}+\|\delta u\|_{{\widetilde{L}_t^\infty(\dot{B}_{2,1}^{\frac{d}{2}-1})}})(\| u^*\|_{\widetilde{L}_t^1(\dot{B}_{2,1}^{\frac{d}{2}+1})}+\|\delta u\|_{\widetilde{L}_t^1(\dot{B}_{2,1}^{\frac{d}{2}+1})}) \nonumber\\
			&\quad +\frac{1}{\varepsilon^2}\| w^\varepsilon\|_{{\widetilde{L}^2_t(\dot{B}_{2,1}^{\frac{d}{2}})}}\|\nabla w^\varepsilon\|_{{\widetilde{L}^2_t(\dot{B}_{2,1}^{\frac{d}{2}-1})}}\nonumber\\
			&\lesssim \varepsilon E_0+\varepsilon^2\delta X(t)+\varepsilon(E_0+\varepsilon\delta X(t))^2+\var(\delta X(t))^2\nonumber\\
			&\lesssim E_0+\var\delta X(t)+(\delta X(t))^2,
		\end{align}
		where we have used that following inequality from \eqref{lh}: 
		\begin{align*}
			\frac{1}{\varepsilon}\|w^\varepsilon\|_{\widetilde{L}_t^2(\dot{B}_{2,1}^{\frac{d}{2}})}&\lesssim 
			\frac{1}{\varepsilon}\|w^\varepsilon\|_{\widetilde{L}_t^2(\dot{B}_{2,1}^{\frac{d}{2}})}^{\ell,\var}+\| w^\varepsilon\|_{\widetilde{L}_t^2(\dot{B}_{2,1}^{\frac{d}{2}+1})}^{h,\varepsilon}\lesssim \delta X(t).
		\end{align*}
Consequently, by substituting the above estimates into \eqref{A111} we discover that
		\begin{align}\label{T11}
			&\frac{1}{\varepsilon}\|\delta u\|_{{\widetilde{L}_t^\infty(\dot{B}_{2,1}^{\frac{d}{2}-1})}}+\frac{1}{\varepsilon}
			\|\delta u\|_{{\widetilde{L}_t^1(\dot{B}_{2,1}^{\frac{d}{2}+1})}}
			+\|\mathcal{R}^\varepsilon\|_{{\widetilde{L}_t^\infty(\dot{B}_{2,1}^{\frac{d}{2}-1})}}+\frac{1}{\varepsilon^2}\|\mathcal{R}^\varepsilon\|_{{\widetilde{L}_t^1(\dot{B}_{2,1}^{\frac{d}{2}-1})}}\nonumber\\
			&\lesssim E_0+\delta E_0+\varepsilon\delta X(t)+(\delta X(t))^2.
		\end{align}
	
We furthermore establish high-order low-frequency estimates for $\mathcal{R}^\var$. Utilizing the equation $\eqref{T8}_2$ of $\mathcal{R}^\varepsilon$ yields that
		\begin{align}\label{T7}
			&\|\mathcal{R}^\varepsilon\|_{{\widetilde{L}_t^\infty(\dot{B}_{2,1}^{\frac{d}{2}})}}^{\ell,\varepsilon}+\frac{1}{\varepsilon^2}\|\mathcal{R}^\varepsilon\|_{{\widetilde{L}_t^1(\dot{B}_{2,1}^{\frac{d}{2}})}}^{\ell,\varepsilon}\nonumber\\
			&\lesssim \|\mathcal{R}^\varepsilon_0\|_{\dot{B}_{2,1}^{\frac{d}{2}}}^{\ell,\varepsilon}+\varepsilon
			\|\delta u\|_{{\widetilde{L}_t^1(\dot{B}_{2,1}^{\frac{d}{2}})}}^{\ell,\varepsilon}+\varepsilon\| u^\varepsilon\cdot\nabla u^\varepsilon\|_{{\widetilde{L}_t^1(\dot{B}_{2,1}^{\frac{d}{2}})}}^{\ell,\varepsilon}+\frac{1}{\varepsilon}\| w^\varepsilon\cdot\nabla w^\varepsilon\|_{{\widetilde{L}_t^1(\dot{B}_{2,1}^{\frac{d}{2}})}}^{\ell,\varepsilon}
            +\|a^\var(\mathcal{Z}^\var+\mathcal{R}^\var)\|_{\widetilde{L}_t^1(\dot{B}_{2,1}^{\frac{d}{2}})}^{\ell,\var}\nonumber\\
			&\lesssim \|\mathcal{P}w_0^\var\|_{\dot{B}_{2,1}^{\frac{d}{2}}}^{\ell,\varepsilon}+\var \|u_0^\var\|_{\dot{B}_{2,1}^{\frac{d}{2}}}^{\ell,\varepsilon}+\|u^\varepsilon\|_{{\widetilde{L}_t^1(\dot{B}_{2,1}^{\frac{d}{2}+1})}}+ \|u^\varepsilon\cdot\nabla u^\varepsilon\|_{{\widetilde{L}_t^1(\dot{B}_{2,1}^{\frac{d}{2}-1})}}^{\ell,\var}+\frac{1}{\var^2}\|w^\varepsilon\cdot\nabla w^\varepsilon\|_{\widetilde{L}_t^1(\dot{B}_{2,1}^{\frac{d}{2}-1})}\nonumber\\
            &\quad+\frac{1}{\var}\|a^\var(\mathcal{Z}^\var+\mathcal{R}^\var)\|_{\widetilde{L}_t^1(\dot{B}_{2,1}^{\frac{d}{2}-1})}^{\ell,\var}\nonumber\\
			&\lesssim  \|\mathcal{P}w_0^\var\|_{\dot{B}_{2,1}^{\frac{d}{2}}}+\|\delta u_0^\var\|_{\dot{B}_{2,1}^{\frac{d}{2}-1}}+\|u_0^*\|_{\dot{B}_{2,1}^{\frac{d}{2}-1}}+\|\delta u\|_{{\widetilde{L}_t^1(\dot{B}_{2,1}^{\frac{d}{2}+1})}}+	\|u^*\|_{{\widetilde{L}_t^1(\dot{B}_{2,1}^{\frac{d}{2}+1})}}+E_0+(\delta X(t))^2\nonumber\\
			&\lesssim E_0+\delta E_0+\var\delta X(t)+(\delta X(t))^2,
		\end{align}
where we have used Lemma \ref{lemma33} and \eqref{lh}.	In terms of the definition of $\mathcal{P}w^\var$ in \eqref{RT1}, we are able to show that
		\begin{align}\label{AA5}
			\|\mathcal{P}w^\varepsilon\|_{\widetilde{L}_t^\infty(\dot{B}_{2,1}^{\frac{d}{2}-1})}&\leq \|\mathcal{R}^\varepsilon\|_{\widetilde{L}_t^\infty(\dot{B}_{2,1}^{\frac{d}{2}-1})}+\varepsilon \|\delta u\|_{\widetilde{L}_t^\infty(\dot{B}_{2,1}^{\frac{d}{2}-1})}+\varepsilon \|u^*\|_{\widetilde{L}_t^\infty(\dot{B}_{2,1}^{\frac{d}{2}-1})}\nonumber\\
			&\lesssim E_0+\delta E_0+\var \delta X(t)+(\delta X(t))^2,
		\end{align}	
    and
		\begin{align}\label{AA51}
			\|\mathcal{P}w^\varepsilon\|_{\widetilde{L}_t^\infty(\dot{B}_{2,1}^{\frac{d}{2}})}^{\ell,\var}&\leq \|\mathcal{R}^\varepsilon\|_{\widetilde{L}_t^\infty(\dot{B}_{2,1}^{\frac{d}{2}})}^{\ell,\var}+\varepsilon \|\delta u\|_{\widetilde{L}_t^\infty(\dot{B}_{2,1}^{\frac{d}{2}})}^{\ell,\var}+\varepsilon \|u^*\|_{\widetilde{L}_t^\infty(\dot{B}_{2,1}^{\frac{d}{2}})}^{\ell,\var}\nonumber\\
			&\leq \|\mathcal{R}^\varepsilon\|_{\widetilde{L}_t^\infty(\dot{B}_{2,1}^{\frac{d}{2}})}^{\ell,\var}+ \|\delta u\|_{\widetilde{L}_t^\infty(\dot{B}_{2,1}^{\frac{d}{2}-1})}+ \|u^*\|_{\widetilde{L}_t^\infty(\dot{B}_{2,1}^{\frac{d}{2}-1})}\nonumber\\
			&\lesssim E_0+\delta E_0+\var \delta X(t)+(\delta X(t))^2.
		\end{align}	
	Combining Young's inequality, interpolation inequality and the estimates obtained in \eqref{T11} and \eqref{T7}, one has 
		\begin{align}\label{A51}
			\frac{1}{\var}\|\mathcal{P}w^\varepsilon\|_{\widetilde{L}^2_t(\dot{B}_{2,1}^{\frac{d}{2}})}^{\ell,\var} 
			&\leq\frac{1}{\var}\|\mathcal{R}^\varepsilon\|_{\widetilde{L}^2_t(\dot{B}_{2,1}^{\frac{d}{2}})}^{\ell,\var}+\|\delta u\|_{\widetilde{L}^2_t(\dot{B}_{2,1}^{\frac{d}{2}})}^{\ell,\var}+ \|u^*\|_{\widetilde{L}^2_t(\dot{B}_{2,1}^{\frac{d}{2}})}^{\ell,\var}\nonumber\\
			&\lesssim \|\mathcal{R}^\varepsilon\|_{{\widetilde{L}_t^\infty(\dot{B}_{2,1}^{\frac{d}{2}})}}^{\ell,\varepsilon}+\frac{1}{\varepsilon^2}\|\mathcal{R}^\varepsilon\|_{{\widetilde{L}_t^1(\dot{B}_{2,1}^{\frac{d}{2}})}}^{\ell,\varepsilon}
			+\frac{1}{\var}\|\delta u\|_{\widetilde{L}_t^\infty(\dot{B}_{2,1}^{\frac{d}{2}-1})}^{\ell,\var}
			+\var\|\delta u\|_{\widetilde{L}_t^1(\dot{B}_{2,1}^{\frac{d}{2}+1})}^{\ell,\var}\nonumber\\
			&\quad+\|u^*\|_{\widetilde{L}_t^\infty(\dot{B}_{2,1}^{\frac{d}{2}-1})}^{\ell,\var}+\|u^*\|_{\widetilde{L}_t^1(\dot{B}_{2,1}^{\frac{d}{2}+1})}^{\ell,\var}
			\nonumber\\
			&\lesssim E_0+\delta E_0+\var\delta X(t)+(\delta X(t))^2
		\end{align}
		and 
		\begin{align}
			\frac{1}{\var}\|\mathcal{P}w^\varepsilon\|_{\widetilde{L}_t^1(\dot{B}_{2,1}^{\frac{d}{2}+1})}^{\ell,\var}
			&\leq \frac{1}{\var}\|\mathcal{R}^\varepsilon\|_{\widetilde{L}_t^1(\dot{B}_{2,1}^{\frac{d}{2}+1})}^{\ell,\var}+\|\delta u\|_{\widetilde{L}_t^1(\dot{B}_{2,1}^{\frac{d}{2}+1})}+ \|u^*\|_{\widetilde{L}_t^1(\dot{B}_{2,1}^{\frac{d}{2}+1})}\nonumber\\
			&\leq \frac{1}{\var^2}\|\mathcal{R}^\varepsilon\|_{\widetilde{L}_t^1(\dot{B}_{2,1}^{\frac{d}{2}})}^{\ell,\var}+\|\delta u\|_{\widetilde{L}_t^1(\dot{B}_{2,1}^{\frac{d}{2}+1})}+ \|u^*\|_{\widetilde{L}_t^1(\dot{B}_{2,1}^{\frac{d}{2}+1})}\nonumber\\
			&\lesssim E_0+\delta E_0+\varepsilon \delta X(t)+(\delta X(t))^2. \label{A52}
		\end{align}
Therefore, collecting the estimates obtained in \eqref{T11}-\eqref{A52}, we end up with \eqref{step1:es}. 
This completes the proof of the lemma.
			\end{proof}
			
	\subsection{Step 2: Low-frequency estimate of \texorpdfstring{$(\delta a,\mathcal{P}^\top w^\varepsilon,\mathcal{Z}^\varepsilon)$}{(delta a, P-top w-epsilon, Z-epsilon)}}
			In this subsection, we aim to establish the uniform estimates of $(\delta a,\mathcal{P}^\top w^\varepsilon,\mathcal{Z}^\varepsilon)$ in the low frequency part. At the beginning, we consider the system of $\delta a$ and $\mathcal{Z}^\var$:
			\begin{equation}\label{A6}
				\left\{
				\begin{aligned}
					&\partial_t (\delta a)-\Delta (\delta a)=- u^\varepsilon\cdot \nabla \delta a-\delta u\cdot\nabla a^*+\frac{1}{\varepsilon}Y^\varepsilon,\\
					&\partial_t\mathcal{Z}^\varepsilon+\frac{1}{\varepsilon^2}\mathcal{Z}^\varepsilon=\varepsilon\partial_t((1+h(a^\varepsilon))\nabla a^\varepsilon)-\frac{1}{\varepsilon}\mathcal{P}^\top(w^\varepsilon\cdot\nabla w^\varepsilon),\\
                    &(\delta a,\mathcal{Z}^\var)(0,x)=(\delta a_0, \mathcal{Z}_0^\var)(x):=(a_0^\var-a_0, \mathcal{P}^{\top}w_0^\var+\var (1+h(a_0^\var))\nabla a_0^\var)(x).
				\end{aligned}
				\right.
			\end{equation}
			\begin{lemma}\label{P3.3}
				Assume $(a^\varepsilon, w^\varepsilon,u^\varepsilon)$ is a solution to the system \eqref{main3}. Then, it holds
				\begin{align}
					&\frac{1}{\varepsilon}\|\delta a\|_{\widetilde{L}_t^\infty(\dot{B}_{2,1}^{\frac{d}{2}-1})}^{\ell,\varepsilon}
					+\frac{1}{\varepsilon}\|\delta a\|_{\widetilde{L}_t^1(\dot{B}_{2,1}^{\frac{d}{2}+1})}^{\ell,\varepsilon}+\frac{1}{\varepsilon}\|\partial_t \delta a\|_{\widetilde{L}_t^1(\dot{B}_{2,1}^{\frac{d}{2}-1})}^{\ell,\varepsilon}\nonumber\\
                    &\quad
					+\|\mathcal{P}^{\top}w^\varepsilon\|_{\widetilde{L}_t^\infty(\dot{B}_{2,1}^{\frac{d}{2}})}^{\ell,\varepsilon}+\frac{1}{\var}\|\mathcal{P}^{\top}w^\varepsilon\|_{\widetilde{L}_t^1(\dot{B}_{2,1}^{\frac{d}{2}+1})}^{\ell,\varepsilon}+	\frac{1}{\varepsilon}\|\mathcal{P}^\top w^\varepsilon\|_{\widetilde{L}_t^2(\dot{B}_{2,1}^{\frac{d}{2}})}^{\ell,\varepsilon}+\frac{1}{\varepsilon^2}\|\mathcal{Z}^\varepsilon\|_{\widetilde{L}_t^1(\dot{B}_{2,1}^{\frac{d}{2}})}^{\ell,\varepsilon}\nonumber\\
					&\lesssim E_0+\delta E_0+(\delta X(t))^2.\label{lowaw}
				\end{align}
			\end{lemma}
			\begin{proof}
				We deduce from Lemma \ref{lemma33} applied to  \eqref{A6} that
				\begin{align}\label{T1}
					&\frac{1}{\varepsilon}\|\delta a\|_{{\widetilde{L}_t^\infty(\dot{B}_{2,1}^{\frac{d}{2}-1})}}^{\ell,\varepsilon}+\frac{1}{\varepsilon}\|\delta a\|_{{\widetilde{L}_t^1(\dot{B}_{2,1}^{\frac{d}{2}+1})}}^{\ell,\varepsilon}+\frac{1}{\varepsilon}
					\|\partial_t\delta a\|_{{\widetilde{L}_t^1(\dot{B}_{2,1}^{\frac{d}{2}-1})}}^{\ell,\varepsilon}\nonumber\\
					&\lesssim \frac{1}{\varepsilon}\|a_0^\varepsilon-a_0^*\|_{\dot{B}_{2,1}^{\frac{d}{2}-1}}^{\ell,\varepsilon}
					+\frac{1}{\varepsilon}\|(u^\varepsilon\cdot\nabla \delta a,\delta u\cdot\nabla a^* )\|_{\widetilde{L}_t^1(\dot{B}_{2,1}^{\frac{d}{2}-1})}^{\ell,\varepsilon}+\frac{1}{\varepsilon^2}\|Y^\varepsilon\|_{\widetilde{L}_t^1(\dot{B}_{2,1}^{\frac{d}{2}-1})}^{\ell,\varepsilon}\nonumber\\
					&\lesssim \frac{1}{\varepsilon}\|\rho_0^\varepsilon-\rho_0^*\|_{\dot{B}_{2,1}^{\frac{d}{2}-1}}^{\ell,\varepsilon}
					+\frac{1}{\varepsilon}\|(u^\varepsilon\cdot\nabla \delta a,\delta u\cdot\nabla a^* )\|_{\widetilde{L}_t^1(\dot{B}_{2,1}^{\frac{d}{2}-1})}^{\ell,\varepsilon}+\frac{1}{\varepsilon^2}\|Y^\varepsilon\|_{\widetilde{L}_t^1(\dot{B}_{2,1}^{\frac{d}{2}-1})}^{\ell,\varepsilon}.
				\end{align}

               For clarity, we divide  into three parts to proceed.

			 \textbf{Part I.} First of all, we provide some useful estimates for later use. By \eqref{T1} and interpolation inequalities, we have
				\begin{align*}
				&\frac{1}{\var}\| \delta a\|_{\widetilde{L}_t^2(\dot{B}_{2,1}^{\frac{d}{2}})}+	\frac{1}{\var}\| \delta u\|_{\widetilde{L}_t^2(\dot{B}_{2,1}^{\frac{d}{2}})}\\
				&\lesssim \Big(\frac{1}{\var} \|\delta a\|_{\widetilde{L}_t^{\infty}(\dot{B}_{2,1}^{\frac{d}{2}-1})}\Big)^{\frac{1}{2}}\Big( \frac{1}{\var} \|\delta a\|_{\widetilde{L}_t^1(\dot{B}_{2,1}^{\frac{d}{2}+1})}\Big)^{\frac{1}{2}}+\Big(\frac{1}{\var} \|\delta u\|_{\widetilde{L}_t^{\infty}(\dot{B}_{2,1}^{\frac{d}{2}-1})}\Big)^{\frac{1}{2}}\Big( \frac{1}{\var} \|\delta u\|_{\widetilde{L}_t^1(\dot{B}_{2,1}^{\frac{d}{2}+1})}\Big)^{\frac{1}{2}}\\
				&\lesssim \frac{1}{\var} \|\delta a\|_{\widetilde{L}_t^{\infty}(\dot{B}_{2,1}^{\frac{d}{2}-1})}^{\ell,\var}+\var\|a^\var\|_{\widetilde{L}_t^\infty(\dot{B}_{2,1}^{\frac{d}{2}+1})}^{h,\var}+\|a^*\|_{\widetilde{L}_t^\infty(\dot{B}_{2,1}^{\frac{d}{2}})}+\frac{1}{\var} \|\delta a\|_{\widetilde{L}_t^1(\dot{B}_{2,1}^{\frac{d}{2}+1})}^{\ell,\var}\\
				&\quad +\frac{1}{\var}\|a^\var\|_{\widetilde{L}_t^1(\dot{B}_{2,1}^{\frac{d}{2}+1})}^{h,\var}
				+\|a^*\|_{\widetilde{L}_t^1(\dot{B}_{2,1}^{\frac{d}{2}+2})}+ \frac{1}{\var} \|\delta u\|_{\widetilde{L}_t^{\infty}(\dot{B}_{2,1}^{\frac{d}{2}-1})}+\frac{1}{\var} \|\delta u\|_{\widetilde{L}_t^1(\dot{B}_{2,1}^{\frac{d}{2}+1})}\\
				&\lesssim E_0+\delta X(t),
			\end{align*}
			and 
				\begin{align*}
				&\|u^\varepsilon\|_{\widetilde{L}_t^2(\dot{B}_{2,1}^{\frac{d}{2}})}+	\|a^*\|_{\widetilde{L}_t^2(\dot{B}^{\frac{d}{2}}_{2,1})}\\
				&\lesssim \|\delta u\|_{\widetilde{L}_t^\infty(\dot{B}_{2,1}^{\frac{d}{2}-1})}	
				+\|\delta u\|_{\widetilde{L}_t^1(\dot{B}_{2,1}^{\frac{d}{2}+1})}	+\| u^*\|_{\widetilde{L}_t^\infty(\dot{B}_{2,1}^{\frac{d}{2}-1})}+\| u^*\|_{\widetilde{L}_t^1(\dot{B}_{2,1}^{\frac{d}{2}+1})}\\&\quad+				\|a^*\|_{\widetilde{L}_t^{\infty}(\dot{B}^{\frac{d}{2}-1}_{2,1})}+ \|a^*\|_{\widetilde{L}_t^1(\dot{B}^{\frac{d}{2}+1}_{2,1})}\\
				&\lesssim \var\delta X(t)+E_0.			
			\end{align*}
		Then, the second term on  the right-hand side of \eqref{T1} can be estimated as follows,
				\begin{align}
					&\frac{1}{\varepsilon}\|(u^\varepsilon\cdot\nabla \delta a,\delta u\cdot\nabla a^* )\|_{\widetilde{L}_t^1(\dot{B}_{2,1}^{\frac{d}{2}-1})}^{\ell,\varepsilon}\nonumber\\
					&\lesssim \frac{1}{\varepsilon}\|\nabla \delta a\|_{\widetilde{L}_t^2(\dot{B}_{2,1}^{\frac{d}{2}-1})}\|u^\varepsilon\|_{\widetilde{L}_t^2(\dot{B}_{2,1}^{\frac{d}{2}})}+\frac{1}{\varepsilon} \|\delta u\|_{\widetilde{L}_t^2(\dot{B}_{2,1}^{\frac{d}{2}})} \|\nabla a^*\|_{\widetilde{L}_t^2(\dot{B}_{2,1}^{\frac{d}{2}-1})}\nonumber\\
					&\lesssim \frac{1}{\varepsilon}\| \delta a\|_{\widetilde{L}_t^2(\dot{B}_{2,1}^{\frac{d}{2}})}\|u^\varepsilon\|_{\widetilde{L}_t^2(\dot{B}_{2,1}^{\frac{d}{2}})}+\frac{1}{\varepsilon} \|\delta u\|_{\widetilde{L}_t^2(\dot{B}_{2,1}^{\frac{d}{2}})} \|a^*\|_{\widetilde{L}_t^2(\dot{B}_{2,1}^{\frac{d}{2}})}\nonumber\\
					&\lesssim E_0\delta X(t)+(\delta X(t))^2\nonumber\\
					&\lesssim E_0+\delta X(t)^2.\label{TRR1}
				\end{align}				
 Furthermore, we find that 
 \begin{align}
 	\frac{1}{\var^2}\|\mathcal{R}^\varepsilon\|_{\widetilde{L}_t^1(\dot{B}_{2,1}^{\frac{d}{2}})}&\leq 
 	\frac{1}{\var^2}\|\mathcal{R}^\varepsilon\|_{\widetilde{L}_t^1(\dot{B}_{2,1}^{\frac{d}{2}})}^{\ell,\var}+\frac{1}{\var^2}\|\mathcal{R}^\varepsilon\|_{\widetilde{L}_t^1(\dot{B}_{2,1}^{\frac{d}{2}})}^{h,\var}\nonumber\\
 	&\lesssim E_0+\delta E_0+\frac{1}{\var^2}\|\mathcal{P}w^\varepsilon\|_{\widetilde{L}_t^1(\dot{B}_{2,1}^{\frac{d}{2}})}^{h,\var}
 	+\frac{1}{\var}\|u^\var\|_{\widetilde{L}_t^1(\dot{B}_{2,1}^{\frac{d}{2}})}^{h,\var}\nonumber\\
 	&\lesssim E_0+\delta E_0+\frac{1}{\var}\|w^\varepsilon\|_{\widetilde{L}_t^1(\dot{B}_{2,1}^{\frac{d}{2}+1})}^{h,\var}
 	+\frac{1}{\var}\|\delta u\|_{\widetilde{L}_t^1(\dot{B}_{2,1}^{\frac{d}{2}})}^{h,\var}+\|u^*\|_{\widetilde{L}_t^1(\dot{B}_{2,1}^{\frac{d}{2}+1})}\nonumber\\
 	&\lesssim E_0+\delta E_0+\var\delta X(t)+(\delta X(t))^2.\label{TYU7}
 \end{align}	
Using \eqref{TYU7} and the \emph{a priori} assumption \eqref{asss}, the estimate of the third term on the right-hand side of \eqref{T1} can be derived as 
			\begin{align}
				\frac{1}{\varepsilon^2}\|Y^\varepsilon\|_{{\widetilde{L}_t^1(\dot{B}_{2,1}^{\frac{d}{2}-1})}}^{\ell,\varepsilon}
				&\lesssim \frac{1}{\varepsilon^2}\|(1+a^\varepsilon)(\mathcal{Z}^\varepsilon+\mathcal{R}^\varepsilon)\|_{\widetilde{L}_t^1(\dot{B}_{2,1}^{\frac{d}{2}})}^{\ell,\varepsilon}\nonumber\\
				&\lesssim  \frac{1}{\varepsilon^2}\|(\mathcal{Z}^\varepsilon,\mathcal{R}^\varepsilon)\|_{\widetilde{L}_t^1(\dot{B}_{2,1}^{\frac{d}{2}})}^{\ell,\varepsilon}+\frac{1}{\varepsilon^2}\|a^\varepsilon\|_{\widetilde{L}_t^\infty(\dot{B}_{2,1}^{\frac{d}{2}})}\|(\mathcal{Z}^\varepsilon,\mathcal{R}^\varepsilon)\|_{\widetilde{L}_t^1(\dot{B}_{2,1}^{\frac{d}{2}})}\nonumber\\
				&\lesssim 
				E_0+\delta E_0+(\delta X(t))^2+\var\delta X(t)+\frac{1}{\varepsilon^2}\|\mathcal{Z}^\varepsilon\|_{\widetilde{L}_t^1(\dot{B}_{2,1}^{\frac{d}{2}})}^{\ell,\varepsilon},\label{TRR2}
			\end{align}
			where we have used the fact that 
			\begin{align}
				\|a^\var\|_{\widetilde{L}_t^\infty(\dot{B}_{2,1}^{\frac{d}{2}})}&\lesssim \|\delta a\|_{\widetilde{L}_t^\infty(\dot{B}_{2,1}^{\frac{d}{2}})}+\|a^*\|_{\widetilde{L}_t^\infty(\dot{B}_{2,1}^{\frac{d}{2}})}\nonumber\\
				&\lesssim  \frac{1}{\var}\|\delta a\|_{\widetilde{L}_t^\infty(\dot{B}_{2,1}^{\frac{d}{2}-1})}^{\ell,\var}+\var\|\delta a\|_{\widetilde{L}_t^\infty(\dot{B}_{2,1}^{\frac{d}{2}+1})}^{h,\var}+\|a^*\|_{\widetilde{L}_t^\infty(\dot{B}_{2,1}^{\frac{d}{2}})}\nonumber\\
				&\lesssim E_0+\delta X(t).\label{ayj}
			\end{align}
		Substituting \eqref{TRR1} and \eqref{TRR2} into \eqref{T1} yields that 
			\begin{align}\label{T6}
				&\frac{1}{\varepsilon}\|\delta a \|_{\widetilde{L}_t^\infty(\dot{B}_{2,1}^{\frac{d}{2}-1})}^{\ell,\varepsilon}
				+\frac{1}{\varepsilon} \|\delta a \|_{\widetilde{L}_t^1(\dot{B}_{2,1}^{\frac{d}{2}+1})}^{\ell,\varepsilon}
				+\frac{1}{\varepsilon}\|\partial_t\delta a \|_{\widetilde{L}_t^1(\dot{B}_{2,1}^{\frac{d}{2}-1})}^{\ell,\varepsilon}\nonumber\\
				&\lesssim E_0+\delta E_0+(\delta X(t))^2+\var\delta X(t)+\frac{1}{\varepsilon^2}\|\mathcal{Z}^\varepsilon\|_{\widetilde{L}_t^1(\dot{B}_{2,1}^{\frac{d}{2}})}^{\ell,\varepsilon}.
			\end{align}
			
			 \textbf{Part II.} The subsequent challenge is to control the last term on the right-hand side of \eqref{T6}. We apply Lemma \ref{lemma33}  to the equation $\eqref{A6}_2$ and derive
			\begin{align}\label{A7}
				&\|\mathcal{Z}^\varepsilon\|_{\widetilde{L}_t^\infty(\dot{B}_{2,1}^{\frac{d}{2}})}^{\ell,\varepsilon}+
				\frac{1}{\varepsilon^2}\|\mathcal{Z}^\varepsilon\|_{\widetilde{L}_t^1(\dot{B}_{2,1}^{\frac{d}{2}})}^{\ell,\varepsilon}\nonumber\\
				&\lesssim \|\mathcal{Z}^\varepsilon_0\|_{\dot{B}_{2,1}^{\frac{d}{2}}}^{\ell,\varepsilon}+\varepsilon\|\nabla\partial_ta^\varepsilon\|_{\widetilde{L}_t^1(\dot{B}_{2,1}^{\frac{d}{2}})}^{\ell,\varepsilon}+\varepsilon\|\partial_t(h(a^\varepsilon)\nabla a^\varepsilon)\|_{\widetilde{L}_t^1(\dot{B}_{2,1}^{\frac{d}{2}})}^{\ell,\varepsilon}+\frac{1}{\varepsilon}\|w^\varepsilon\cdot\nabla w^\varepsilon\|_{\widetilde{L}_t^1(\dot{B}_{2,1}^{\frac{d}{2}})}^{\ell,\varepsilon}.
			\end{align}
			Now we first handle the first term on the right-hand side of \eqref{A7}, which satisfies 
			\begin{align}
				\|\mathcal{Z}_0^\varepsilon\|_{\dot{B}_{2,1}^{\frac{d}{2}}}^{\ell,\varepsilon}
				&\lesssim \|\mathcal{P}^\top w^\varepsilon_0+\varepsilon(1+h(a_0^\varepsilon))\nabla a_0^\varepsilon\|_{\dot{B}_{2,1}^{\frac{d}{2}}}^{\ell,\varepsilon}\nonumber\\
				&\lesssim \|\mathcal{P}^\top w^\varepsilon_0\|_{\dot{B}_{2,1}^{\frac{d}{2}}}^{\ell,\varepsilon}+2^{-m_0}\|a_0^\varepsilon\|_{\dot{B}_{2,1}^{\frac{d}{2}}}^{\ell,\varepsilon}+2^{-m_0}\|h(a_0^\varepsilon)\nabla a_0^\varepsilon\|_{\dot{B}_{2,1}^{\frac{d}{2}-1}}^{\ell,\varepsilon}\nonumber\\
				&\lesssim (1+2^{-m_0}) \delta E_0+2^{-m_0}\|h(a_0^\varepsilon)\|_{\dot{B}_{2,1}^{\frac{d}{2}}} \|\nabla a_0^\varepsilon\|_{\dot{B}_{2,1}^{\frac{d}{2}-1}}\nonumber\\
				&\lesssim (1+2^{-m_0})( E_0+\delta E_0).\label{T4}
			\end{align}		
			Before analyzing $\varepsilon \partial_t(h(a^\varepsilon)\nabla a^\varepsilon)$, we first provide a control on $\partial_t a^\varepsilon$. It follows from the definition \eqref{Jvar} of the threshold $J_\var$, Bernstein's inequality (Lemma \ref{lemma21}) and \eqref{ayj} that
			\begin{align}
				\varepsilon\|\nabla\partial_ta^\varepsilon\|_{\widetilde{L}_t^1(\dot{B}_{2,1}^{\frac{d}{2}})}^{\ell,\varepsilon}&\lesssim \var \|\partial_t \delta a\|_{\widetilde{L}_t^1(\dot{B}_{2,1}^{\frac{d}{2}+1})}^{\ell,\var}+\varepsilon\|\nabla\partial_ta^*\|_{\widetilde{L}_t^1(\dot{B}_{2,1}^{\frac{d}{2}})}^{\ell,\varepsilon}\nonumber\\
				&\lesssim\frac{2^{-2m_0}}{\var} \|\partial_t \delta a\|_{\widetilde{L}_t^1(\dot{B}_{2,1}^{\frac{d}{2}-1})}^{\ell,\var}+2^{-m_0}\|\partial_ta^*\|_{\widetilde{L}_t^1(\dot{B}_{2,1}^{\frac{d}{2}})}\nonumber\\
				&\lesssim \frac{2^{-2m_0}}{\var} \|\partial_t \delta a\|_{\widetilde{L}_t^1(\dot{B}_{2,1}^{\frac{d}{2}-1})}^{\ell,\var}+ 2^{-m_0}E_0.
			\end{align}	
			Define a new function  
			\begin{align*}
				\tilde{h}(a^\varepsilon):=\int_0^{a^\var}h(s)ds=-a^\var+\ln(1+a^\var).	
			\end{align*}
            It is observed that, due to the first equation in \eqref{main3} and the product law \eqref{uv1},
\begin{align}\label{347}
\|\partial_t a^\var\|_{\widetilde{L}^1_t(\dot{B}^{\frac{d}{2}}_{2,1})}&\lesssim \frac{1}{\var}\|w^\var a^\var\|_{\widetilde{L}^1_t(\dot{B}^{\frac{d}{2}+1}_{2,1})}+\frac{1}{\var}\|w^\var \|_{\widetilde{L}^1_t(\dot{B}^{\frac{d}{2}+1}_{2,1})}\nonumber\\
&\lesssim \frac{1}{\var}\|w^\var\|_{\widetilde{L}^2_t(\dot{B}^{\frac{d}{2}}_{2,1})} \|a^\var\|_{\widetilde{L}^2_t(\dot{B}^{\frac{d}{2}+1}_{2,1})}+\frac{1}{\var}\|w^\var\|_{\widetilde{L}^1_t(\dot{B}^{\frac{d}{2}}_{2,1})} \|a^\var\|_{\widetilde{L}^{\infty}_t(\dot{B}^{\frac{d}{2}}_{2,1})}+\frac{1}{\var}\|w^\var\|_{\widetilde{L}^1_t(\dot{B}^{\frac{d}{2}+1}_{2,1})}\nonumber\\
&\lesssim E_0\delta X(t)+(\delta X(t))^2.
\end{align}
Thus, since $h(a^\varepsilon)\nabla a^\varepsilon=\nabla\tilde{h}(a^\varepsilon)$, we obtain from \eqref{uv2}, \eqref{F1} and \eqref{347} that
	\begin{align}\label{dta}
				\varepsilon\|\partial_t(h(a^\varepsilon)\nabla a^\varepsilon)\|_{\widetilde{L}_t^1(\dot{B}_{2,1}^{\frac{d}{2}})}^{\ell,\varepsilon}	& \lesssim 2^{-m_0}\|\partial_t a^\varepsilon \tilde{h}'(a^\varepsilon)\|_{\widetilde{L}_t^1(\dot{B}_{2,1}^{\frac{d}{2}})}^{\ell,\varepsilon}\nonumber\\
                &\lesssim 2^{-m_0} \|\partial_t a^\var\|_{\widetilde{L}^1_t(\dot{B}^{\frac{d}{2}}_{2,1})} \|\tilde{h}'(a^\varepsilon)\|_{\widetilde{L}^{\infty}_t(\dot{B}^{\frac{d}{2}}_{2,1})} \nonumber\\          
				&\lesssim 2^{-m_0}(E_0\delta X(t)+(\delta X(t))^2).
			\end{align}
		Similar to the proof of \eqref{T2}, we easily obtain
			\begin{align*}
				\frac{1}{\varepsilon}\|w^\varepsilon\cdot\nabla w^\varepsilon\|_{\widetilde{L}_t^1(\dot{B}_{2,1}^{\frac{d}{2}})}^{\ell,\varepsilon}
				\lesssim \frac{1}{\varepsilon^2}\|w^\varepsilon\cdot\nabla w^\varepsilon\|_{\widetilde{L}_t^1(\dot{B}_{2,1}^{\frac{d}{2}-1})}^{\ell,\varepsilon}
				\lesssim \var(\delta X(t))^2.
			\end{align*}
		Consequently, it concludes from the above estimates  that 
			\begin{align}\label{A8}
				&\|\mathcal{Z}^\varepsilon\|_{\widetilde{L}_t^\infty(\dot{B}_{2,1}^{\frac{d}{2}})}^{\ell,\varepsilon}+\frac{1}{\varepsilon^2}
				\|\mathcal{Z}^\varepsilon\|_{\widetilde{L}_t^1(\dot{B}_{2,1}^{\frac{d}{2}})}^{\ell,\varepsilon}\nonumber\\
				&\lesssim (1+2^{-m_0})(E_0+\delta E_0+(\delta X(t))^2)+\frac{2^{-2m_0}}{\var}\|\partial_t\delta a\|_{\widetilde{L}_t^1(\dot{B}_{2,1}^{\frac{d}{2}-1})}^{\ell,\varepsilon}.
			\end{align}
			
			Multiplying \eqref{T6} by a suitably chosen constant (independent of $\var$) and adding the resulting inequality to \eqref{A8}, we can find a  uniform positive constant $C^*$ such that
			\begin{align*}
				&\frac{1}{\varepsilon}\|\delta a\|_{\widetilde{L}_t^\infty(\dot{B}_{2,1}^{\frac{d}{2}-1})}^{\ell,\varepsilon}
				+\frac{1}{\varepsilon}\|\delta a\|_{\widetilde{L}_t^1(\dot{B}_{2,1}^{\frac{d}{2}+1})}^{\ell,\varepsilon}
				+\frac{1}{\varepsilon} \|\partial_t\delta a\|_{\widetilde{L}_t^1(\dot{B}_{2,1}^{\frac{d}{2}-1})}^{\ell,\varepsilon}
				+\|\mathcal{Z}^\varepsilon\|_{\widetilde{L}_t^\infty(\dot{B}_{2,1}^{\frac{d}{2}})}^{\ell,\varepsilon}+\frac{1}{\varepsilon^2}\|\mathcal{Z}^\varepsilon\|_{\widetilde{L}_t^1(\dot{B}_{2,1}^{\frac{d}{2}})}^{\ell,\varepsilon}\nonumber\\
				&\leq C^*(1+2^{-m_0}) (E_0+\delta E_0+(\delta X(t))^2)+\frac{C^*  2^{-2m_0}}{\varepsilon} \|\partial_t\delta a\|_{\widetilde{L}_t^1(\dot{B}_{2,1}^{\frac{d}{2}-1})}^{\ell,\varepsilon}.
			\end{align*}
			
			Now we take 
			\begin{align}
				m_0=[\frac{1}{2}\log_{2} 2C^*]+1\label{m0},
			\end{align}
			such that $C^*2^{-2m_0}<\frac{1}{2}$ and then 
			\begin{align}\label{A9}
				&\frac{1}{\varepsilon}\|\delta a\|_{\widetilde{L}_t^\infty(\dot{B}_{2,1}^{\frac{d}{2}-1})}^{\ell,\varepsilon}
				+\frac{1}{\varepsilon}\|\delta a\|_{\widetilde{L}_t^1(\dot{B}_{2,1}^{\frac{d}{2}+1})}^{\ell,\varepsilon}
				+\frac{1}{\varepsilon} \|\partial_t\delta a\|_{\widetilde{L}_t^1(\dot{B}_{2,1}^{\frac{d}{2}-1})}^{\ell,\varepsilon}
				+\|\mathcal{Z}^\varepsilon\|_{\widetilde{L}_t^\infty(\dot{B}_{2,1}^{\frac{d}{2}})}^{\ell,\varepsilon}+\frac{1}{\varepsilon^2}\|\mathcal{Z}^\varepsilon\|_{\widetilde{L}_t^1(\dot{B}_{2,1}^{\frac{d}{2}})}^{\ell,\varepsilon}\nonumber\\
				&\lesssim E_0+\delta E_0+ (\delta X(t))^2.
			\end{align}
            
 In what follows, we will not specify the constant $m_0$ when using frequency cutoff since it has been fixed in \eqref{m0}.
			
			 \textbf{Part III.} According to \eqref{RT1}, we are able to show that  
			\begin{align}
				\mathcal{P}^\top w^\varepsilon=\mathcal{Z}^\varepsilon-\varepsilon(1+h(a^\varepsilon))\nabla a^\varepsilon.	
			\end{align}
			In terms of \eqref{A9}, we can obtain the estimate of $\mathcal{P}^\top w^\varepsilon$ 
			\begin{align}
				\|\mathcal{P}^\top w^\varepsilon\|_{\widetilde{L}_t^\infty(\dot{B}_{2,1}^{\frac{d}{2}})}^{\ell,\varepsilon}
				&\lesssim \|\mathcal{Z}^\varepsilon\|_{\widetilde{L}_t^\infty(\dot{B}_{2,1}^{\frac{d}{2}})}^{\ell,\varepsilon}+\varepsilon
				\|\nabla a^\varepsilon\|_{\widetilde{L}_t^\infty(\dot{B}_{2,1}^{\frac{d}{2}})}^{\ell,\varepsilon}+\varepsilon
				\|h(a^\varepsilon)\nabla a^\varepsilon\|_{\widetilde{L}_t^\infty(\dot{B}_{2,1}^{\frac{d}{2}})}\nonumber\\
				&\lesssim  \|\mathcal{Z}^\varepsilon\|_{\widetilde{L}_t^\infty(\dot{B}_{2,1}^{\frac{d}{2}})}^{\ell,\varepsilon}
				+\|a^\varepsilon\|_{\widetilde{L}_t^\infty(\dot{B}_{2,1}^{\frac{d}{2}})}^{\ell,\varepsilon}+
				\|h(a^\varepsilon)\nabla a^\varepsilon\|_{\widetilde{L}_t^\infty(\dot{B}_{2,1}^{\frac{d}{2}-1})}\nonumber\\
				&\lesssim  \|\mathcal{Z}^\varepsilon\|_{\widetilde{L}_t^\infty(\dot{B}_{2,1}^{\frac{d}{2}})}^{\ell,\varepsilon}
				+\frac{1}{\var}\|\delta a\|_{\widetilde{L}_t^\infty(\dot{B}_{2,1}^{\frac{d}{2}-1})}^{\ell,\varepsilon}+\|a^*\|_{\widetilde{L}_t^\infty(\dot{B}_{2,1}^{\frac{d}{2}})}+\|h(a^\varepsilon)\nabla a^\varepsilon\|_{\widetilde{L}_t^\infty(\dot{B}_{2,1}^{\frac{d}{2}-1})}.\label{TYU2}
			\end{align}
	
	Thanks to the following fact
			\begin{align}
				\|\delta a\|_{\widetilde{L}_t^\infty(\dot{B}_{2,1}^{\frac{d}{2}})}
				&\lesssim \|\delta a\|_{\widetilde{L}_t^\infty(\dot{B}_{2,1}^{\frac{d}{2}})}^{\ell,\var}
				+\|\delta a\|_{\widetilde{L}_t^\infty(\dot{B}_{2,1}^{\frac{d}{2}})}^{h,\var}\nonumber\\
				&\lesssim 
				\frac{1}{\var}	\|\delta a\|_{\widetilde{L}_t^\infty(\dot{B}_{2,1}^{\frac{d}{2}-1})}^{\ell,\var}+
				\| a^\varepsilon\|_{\widetilde{L}_t^\infty(\dot{B}_{2,1}^{\frac{d}{2}})}^{h,\var}+
				\| a^*\|_{\widetilde{L}_t^\infty(\dot{B}_{2,1}^{\frac{d}{2}})}^{h,\var}\nonumber\\
				&\lesssim 	\frac{1}{\var}	\|\delta a\|_{\widetilde{L}_t^\infty(\dot{B}_{2,1}^{\frac{d}{2}-1})}^{\ell,\var}+
				\var\|a^\varepsilon\|_{\widetilde{L}_t^\infty(\dot{B}_{2,1}^{\frac{d}{2}+1})}^{h,\var}
				+\|a^*\|_{\widetilde{L}_t^\infty(\dot{B}_{2,1}^{\frac{d}{2}})}\nonumber\\
				&\lesssim E_0+\delta X(t),\label{TYU1}
			\end{align}
			we discover that 
			\begin{align*}
				\|h(a^\varepsilon)\nabla a^\varepsilon\|_{\widetilde{L}_t^\infty(\dot{B}_{2,1}^{\frac{d}{2}-1})}
				\lesssim \|h(a^\varepsilon)\|_{\widetilde{L}_t^\infty(\dot{B}_{2,1}^{\frac{d}{2}})}
				\|\nabla a^\varepsilon\|_{\widetilde{L}_t^\infty(\dot{B}_{2,1}^{\frac{d}{2}-1})}
				\lesssim \|a^\varepsilon\|_{\widetilde{L}_t^\infty(\dot{B}_{2,1}^{\frac{d}{2}})}^2
				\lesssim E_0+(\delta X(t))^2.
			\end{align*}
			Substituting the above estimates into \eqref{TYU2} yields that 
			\begin{align}
				\|\mathcal{P}^\top w^\varepsilon\|_{\widetilde{L}_t^\infty(\dot{B}_{2,1}^{\frac{d}{2}})}^{\ell,\varepsilon}	\lesssim E_0+\delta E_0+(\delta X(t))^2.
			\end{align}
	It follows from \eqref{A9} and \eqref{TYU1} that 
			\begin{align}
				\frac{1}{\varepsilon}\|\mathcal{P}^\top w^\varepsilon\|_{\widetilde{L}_t^1(\dot{B}_{2,1}^{\frac{d}{2}+1})}^{\ell,\varepsilon}
				&\lesssim \frac{1}{\varepsilon^2} \|\mathcal{Z}^\varepsilon\|_{\widetilde{L}_t^1(\dot{B}_{2,1}^{\frac{d}{2}})}^{\ell,\varepsilon}+\|\nabla a^\varepsilon\|_{\widetilde{L}_t^1(\dot{B}_{2,1}^{\frac{d}{2}+1})}^{\ell,\var}+\|h(a^\varepsilon)\nabla a^\varepsilon\|_{\widetilde{L}_t^1(\dot{B}_{2,1}^{\frac{d}{2}+1})}^{\ell,\varepsilon}\nonumber\\
				&\lesssim \frac{1}{\varepsilon^2} \|\mathcal{Z}^\varepsilon\|_{\widetilde{L}_t^1(\dot{B}_{2,1}^{\frac{d}{2}})}^{\ell,\varepsilon}+\frac{1}{\var}\| \delta a\|_{\widetilde{L}_t^1(\dot{B}_{2,1}^{\frac{d}{2}+1})}^{\ell,\var}+\| a^*\|_{\widetilde{L}_t^1(\dot{B}_{2,1}^{\frac{d}{2}+2})}\nonumber\\
				&\quad +\|h(a^\varepsilon)\|_{\widetilde{L}_t^\infty(\dot{B}_{2,1}^{\frac{d}{2}})}\|\nabla (a^\varepsilon)^{\ell,\var}\|_{\widetilde{L}_t^1(\dot{B}_{2,1}^{\frac{d}{2}+1})}+\frac{1}{\var}\|h(a^\varepsilon)\|_{\widetilde{L}_t^\infty(\dot{B}_{2,1}^{\frac{d}{2}})}\|\nabla (a^\varepsilon)^{h,\var}\|_{\widetilde{L}_t^1(\dot{B}_{2,1}^{\frac{d}{2}})}\nonumber\\
				&\lesssim \frac{1}{\varepsilon^2} \|\mathcal{Z}^\varepsilon\|_{\widetilde{L}_t^1(\dot{B}_{2,1}^{\frac{d}{2}})}^{\ell,\varepsilon}+\frac{1}{\var}\|\delta a\|_{\widetilde{L}_t^1(\dot{B}_{2,1}^{\frac{d}{2}+1})}^{\ell,\var}+\| a^*\|_{\widetilde{L}_t^1(\dot{B}_{2,1}^{\frac{d}{2}+2})}+\frac{1}{\var}\|a^\varepsilon\|_{\widetilde{L}_t^\infty(\dot{B}_{2,1}^{\frac{d}{2}})}\|\delta a\|_{\widetilde{L}_t^1(\dot{B}_{2,1}^{\frac{d}{2}+1})}^{\ell,\var}\nonumber\\
				&\quad+\|a^\var\|_{\widetilde{L}_t^\infty(\dot{B}_{2,1}^{\frac{d}{2}})}\|a^*\|_{\widetilde{L}_t^1(\dot{B}_{2,1}^{\frac{d}{2}+2})}+\frac{1}{\var}\|a^\var\|_{\widetilde{L}_t^\infty(\dot{B}_{2,1}^{\frac{d}{2}})}\|a^\var\|_{\widetilde{L}_t^1(\dot{B}_{2,1}^{\frac{d}{2}+1})}^{h,\var} \nonumber\\
				&\lesssim E_0+\delta E_0+(\delta X(t))^2.\label{TYU2-2}
			\end{align}
			
			Moreover, we also obtain that 
			\begin{align}
				\frac{1}{\varepsilon}\|\mathcal{P}^\top w^\varepsilon\|_{\widetilde{L}_t^2(\dot{B}_{2,1}^{\frac{d}{2}})}^{\ell,\varepsilon}
				&\lesssim \frac{1}{\varepsilon} \|\mathcal{Z}^\varepsilon\|_{\widetilde{L}_t^2(\dot{B}_{2,1}^{\frac{d}{2}})}^{\ell,\varepsilon}+\|\nabla a^\varepsilon\|_{\widetilde{L}_t^2(\dot{B}_{2,1}^{\frac{d}{2}})}^{\ell,\var}+\|h(a^\varepsilon)\nabla a^\varepsilon\|_{\widetilde{L}_t^2(\dot{B}_{2,1}^{\frac{d}{2}})}^{\ell,\varepsilon}\nonumber\\
				&\lesssim  \|\mathcal{Z}^\varepsilon\|_{\widetilde{L}_t^\infty(\dot{B}_{2,1}^{\frac{d}{2}})}^{\ell,\varepsilon}
				+\frac{1}{\varepsilon^2} \|\mathcal{Z}^\varepsilon\|_{\widetilde{L}_t^1(\dot{B}_{2,1}^{\frac{d}{2}})}^{\ell,\varepsilon}\nonumber\\
                &\quad+\frac{1}{\var}\|\delta a\|_{\widetilde{L}_t^\infty(\dot{B}_{2,1}^{\frac{d}{2}-1})}^{\ell,\var}+\frac{1}{\var}\|\delta a\|_{\widetilde{L}_t^1(\dot{B}_{2,1}^{\frac{d}{2}+1})}^{\ell,\var}+\| a^*\|_{\widetilde{L}_t^\infty(\dot{B}_{2,1}^{\frac{d}{2}})}\nonumber\\
				&\quad+\| a^*\|_{\widetilde{L}_t^1(\dot{B}_{2,1}^{\frac{d}{2}+2})}+\|a^\varepsilon\|_{\widetilde{L}_t^\infty(\dot{B}_{2,1}^{\frac{d}{2}})}\|a^\var\|_{\widetilde{L}_t^2(\dot{B}_{2,1}^{\frac{d}{2}+1})}\nonumber\\
				&\lesssim E_0+\delta E_0+(\delta X(t))^2,\label{TYU2-3}
			\end{align}
			where we have used the fact that 
			\begin{align*}
				\|a^\varepsilon\|_{\widetilde{L}_t^2(\dot{B}_{2,1}^{\frac{d}{2}+1})}&\lesssim 
				\| a^\varepsilon\|_{\widetilde{L}_t^2(\dot{B}_{2,1}^{\frac{d}{2}+1})}^{\ell,\var}+\|a^\var\|_{\widetilde{L}_t^2(\dot{B}_{2,1}^{\frac{d}{2}+1})}^{h,\var}\\
				&\lesssim \|\delta a\|_{\widetilde{L}_t^2(\dot{B}_{2,1}^{\frac{d}{2}+1})}^{\ell,\var}
				+\| a^*\|_{\widetilde{L}_t^2(\dot{B}_{2,1}^{\frac{d}{2}+1})}+\|a^\var\|_{\widetilde{L}_t^2(\dot{B}_{2,1}^{\frac{d}{2}+1})}^{h,\var} \\
				&\lesssim \frac{1}{\var}\|\delta a\|_{\widetilde{L}_t^\infty(\dot{B}_{2,1}^{\frac{d}{2}-1})}^{\ell,\var}
				+\frac{1}{\var}\|\delta a\|_{\widetilde{L}_t^1(\dot{B}_{2,1}^{\frac{d}{2}+1})}^{\ell,\var}+
				\| a^*\|_{\widetilde{L}_t^\infty(\dot{B}_{2,1}^{\frac{d}{2}})}+\| a^*\|_{\widetilde{L}_t^1(\dot{B}_{2,1}^{\frac{d}{2}+2})}\nonumber\\
				&\quad+\var\|a^\var\|_{\widetilde{L}_t^\infty(\dot{B}_{2,1}^{\frac{d}{2}+1})}^{h,\var}+\frac{1}{\var}\|a^\var\|_{\widetilde{L}_t^1(\dot{B}_{2,1}^{\frac{d}{2}+1})}^{h,\var} \\
				&\lesssim E_0+\delta X(t).
			\end{align*}

			Finally, we combine \eqref{A9}, \eqref{TYU2}, \eqref{TYU2-2} and \eqref{TYU2-3} to show \eqref{lowaw}, which completes the proof of this lemma.	
		\end{proof}
		
		\subsection{Step 3: High-frequency estimate of $(a^\varepsilon,  w^\varepsilon, \mathcal{Z}^\varepsilon)$}
		In this subsection, we aim to establish the estimates of \texorpdfstring{$(a^\varepsilon,w^\varepsilon,\mathcal{Z}^\varepsilon)$}{(a, w, Z)} in the high-frequency part.
					\begin{lemma}\label{P3.4}
			Assume $(a^\varepsilon, w^\varepsilon,u^\varepsilon)$ is a solution of the system \eqref{main3}. Then it holds
			\begin{align}
				&\varepsilon\|(a^\varepsilon,w^\varepsilon)\|_{\widetilde{L}_t^\infty(\dot{B}_{2,1}^{\frac{d}{2}+1})}^{h,\varepsilon}
				+\frac{1}{\varepsilon}\|(a^\varepsilon,w^\varepsilon)\|_{\widetilde{L}_t^1(\dot{B}_{2,1}^{\frac{d}{2}+1})}^{h,\varepsilon}
				+\frac{1}{\varepsilon^2}\|\mathcal{Z}^\varepsilon\|_{\widetilde{L}_t^1(\dot{B}_{2,1}^{\frac{d}{2}})}^{h,\varepsilon}+\frac{1}{\varepsilon^2}\|\mathcal{R}^\varepsilon\|_{\widetilde{L}_t^1(\dot{B}_{2,1}^{\frac{d}{2}})}^{h,\varepsilon}\nonumber\\
				&\lesssim  E_0+\delta E_0
				+(\delta X(t))^2+\|\delta u^\varepsilon\|_{\widetilde{L}_t^1(\dot{B}_{2,1}^{\frac{d}{2}+1})}.\label{high1}
			\end{align}
		\end{lemma}  
		\begin{proof}	For clarity we divide the proof into two parts.

                \textbf{Part I.}
			Noting that
			\begin{equation}\label{main8}
				\left\{
				\begin{aligned}
					&\partial_ta^\varepsilon+\frac{1}{\varepsilon}w^\varepsilon\cdot\nabla a^\varepsilon+\frac{1}{\varepsilon}(1+a^\varepsilon)\text{div}w^\varepsilon=0,\\
					&\partial_tw^\varepsilon+\frac{1}{\varepsilon}w^\varepsilon\cdot\nabla w^\varepsilon+\frac{1}{\varepsilon}(1+h(a^\varepsilon))\nabla a^\varepsilon+\frac{1}{\varepsilon^2}w^\varepsilon=\frac{1}{\varepsilon} u^\varepsilon.
				\end{aligned}
				\right.
			\end{equation}
			Applying $\dot{\Delta}_j$ to the system \eqref{main8}, we have
			\begin{equation}\label{main9}
				\left\{
				\begin{aligned}
					&\partial_ta_j^\varepsilon+\frac{1}{\varepsilon}w^\var\cdot\nabla a_j^\varepsilon+\frac{1}{\varepsilon}(1+a^\varepsilon)\text{div}w_j^\varepsilon=
					\frac{1}{\varepsilon}(\mathcal{N}^\varepsilon_{1j}+\mathcal{N}^\varepsilon_{2j}),\\
					&\partial_tw_j^\varepsilon+\frac{1}{\varepsilon}w^\varepsilon\cdot\nabla w^\varepsilon_j+\frac{1}{\varepsilon}(1+h(a^\varepsilon))\nabla a_j^\varepsilon+\frac{1}{\varepsilon^2}w_j^\varepsilon=\frac{1}{\varepsilon} u_j^\var+\frac{1}{\varepsilon} (\mathcal{N}_{3j}^{\var}+\mathcal{N}_{4j}^\var),
				\end{aligned}
				\right.
			\end{equation}
			where the commutators $\mathcal{N}_{1j}, \mathcal{N}_{2j},  \mathcal{N}_{3j}$ and $\mathcal{N}_{4j}$ are given by
			\begin{align*}
				&\mathcal{N}^\varepsilon_{1j}=[w^\varepsilon,\dot{\Delta}_j]\nabla a^\varepsilon,\quad 	\mathcal{N}^\varepsilon_{2j}=[a^\varepsilon,\dot{\Delta}_j]\text{div} w^\varepsilon,\\
				&\mathcal{N}^\varepsilon_{3j}=[w^\varepsilon,\dot{\Delta}_j]\nabla w^\varepsilon,\quad \mathcal{N}^\varepsilon_{4j}=
				[h(a^\var),\dot{\Delta}_j]\nabla a^\var.
			\end{align*}
			
			Multiplying $\eqref{main9}_1$ by $\frac{1}{1+a^\var}a_j^\varepsilon$ and integrating the resulting equation in $\mathbb{R}^d$ yields that 
			\begin{align}\label{en1}
				&\frac{1}{2}\frac{d}{dt}\int_{\mathbb{R}^d}\frac{1}{1+a^\var}|a_j^\varepsilon|^2dx+\frac{1}{\varepsilon}\int_{\mathbb{R}^d}\text{div}w_j^\varepsilon a_j^\varepsilon dx\nonumber\\
				&=\int_{\mathbb{R}^d}\Big(-\frac{1}{2}\frac{\partial_ta^\varepsilon}{(1+a^\var)^2}|a_j^\varepsilon|^2+\frac{1}{2\varepsilon}\text{div}(\frac{w^\var}{1+a^\var})|a_j^\varepsilon|^2+\frac{1}{\varepsilon(1+a^\var)}(\mathcal{N}^\varepsilon_{1j}+\mathcal{N}^\varepsilon_{2j})a_j^\varepsilon\Big)dx\nonumber\\
				&\lesssim \Big\|\frac{\partial_ta^\varepsilon}{(1+a^\var)^2}\Big\|_{L^\infty}\|a_j^\varepsilon\|_{L^2}^2+\frac{1}{\varepsilon}\Big\|\text{div}(\frac{w^\var}{1+a^\var})\Big\|_{L^\infty}\|a_j^\varepsilon\|_{L^2}^2
				\nonumber\\
                &\quad +\frac{1}{\varepsilon}\Big\|\frac{1}{1+a^\var}\Big\|_{L^\infty}	\|(\mathcal{N}^\varepsilon_{1j},\mathcal{N}^\varepsilon_{2j})\|_{L^2}\|a_j^\var\|_{L^2}.
			\end{align}
			Multiplying $\eqref{main9}_2$ by $w_j^\varepsilon$ and then integrating the resulting equation over the whole space gives 
			\begin{align}\label{en2}
				&\frac{1}{2}\frac{d}{dt}\|w_j^\varepsilon\|_{L^2}^2+\frac{1}{\varepsilon}\int_{\mathbb{R}^d} w_j^\varepsilon\cdot \nabla a_j^\varepsilon dx+\frac{1}{\varepsilon^2}\|w_j^\varepsilon\|_{L^2}^2\nonumber\\
				&=\int_{\mathbb{R}^d}\Big(\frac{1}{2\varepsilon}\text{div}w^\varepsilon |w_j^\varepsilon|^2+\frac{1}{\varepsilon}u_j^\varepsilon\cdot w_j^\varepsilon+\frac{1}{\varepsilon}(\mathcal{N}^\varepsilon_{3j}+\mathcal{N}^\varepsilon_{4j})w_j^\varepsilon-\frac{1}{\var}h(a^\var)\nabla a_j^\var \cdot w_j^\var\Big)dx\nonumber\\
				&\lesssim 
				\frac{1}{\varepsilon}\|\text{div}w^\varepsilon\|_{L^\infty}\|w_j^\varepsilon\|_{L^2}^2+\frac{1}{\varepsilon}\|u_j^\varepsilon\|_{L^2}\|w_j^\varepsilon\|_{L^2} \nonumber\\
                &\quad+\frac{1}{\varepsilon}(\|\mathcal{N}^\varepsilon_{3j}\|_{L^2}+\|\mathcal{N}^\varepsilon_{4j}\|_{L^2})\|w_j^\varepsilon\|_{L^2}
				+\frac{1}{\var}\|a_j^\var\|_{L^2}\|\nabla a^\var\|_{L^\infty}\|w_j^\varepsilon\|_{L^2}.
			\end{align}
			To capture the dissipation of $a_j^\var$, we multiply $\eqref{main9}_2$ by $\varepsilon \nabla a_j^\varepsilon$ and then integrate the resulting equation in $\mathbb{R}^d$ so that
			\begin{align}\label{en3}
				&\varepsilon\frac{d}{dt}\int_{\mathbb{R}^d}w_j^\varepsilon\cdot\nabla a_j^\varepsilon dx+\|\nabla a_j^\varepsilon\|_{L^2}^2-\int_{\mathbb{R}^d}(1+a^\varepsilon)|\text{div}w_j^\varepsilon|^2dx+\frac{1}{\varepsilon}\int_{\mathbb{R}^d}w_j^\varepsilon\cdot\nabla a_j^\varepsilon dx\nonumber\\	
				&=\int_{\mathbb{R}^d}(-w_j^\varepsilon\cdot\nabla w^\varepsilon_j-h(a^\var)\nabla a^\var+u_j^\var+(\mathcal{N}_{3j}^\var+\mathcal{N}_{4j}^\var))\cdot\nabla a_j^\var dx\nonumber\\
				&\quad+\int_{\mathbb{R}^d}(w^\var\cdot\nabla a_j^\var\text{div}w_j^\var+\text{div}w_j^\var(\mathcal{N}_{1j}^\var+\mathcal{N}_{2j}^\var))dx.
			\end{align}
			
			Let $\eta_0$ be a small positive constant. Define 
			\begin{align*}
				L_j^{h,\varepsilon}(t)=\int_{\mathbb{R}^d}(\frac{1}{1+a^\var}|a_j^\varepsilon|^2+|w_j^\varepsilon|^2)dx+\frac{\eta_0}{2^{2j}\var}\int_{\mathbb{R}^d} w_j^\varepsilon\cdot\nabla a_j^\varepsilon dx,
			\end{align*}
			and 
			\begin{align*}
				D_j^{h,\varepsilon}(t)=\frac{1}{\varepsilon^2}\|w_j^\varepsilon\|_{L^2}^2+\frac{\eta_0}{2^{2j}\varepsilon^2 }\Big(\|\nabla a_j^\varepsilon\|_{L^2}^2-\int_{\mathbb{R}^d}(1+a^\var)|\text{div}w_j^\varepsilon|^2dx+\frac{1}{\varepsilon}\int_{\mathbb{R}^d}w_j^\varepsilon\cdot\nabla a_j^\varepsilon dx\Big).
			\end{align*}
			Thanks to the \emph{a priori} assumption \eqref{asss} and the smallness of $\zeta_1$ and $E_0$, we have 
			\begin{align*}
				\|a^\varepsilon\|_{L_t^\infty L_x^\infty}&\lesssim\|a^\varepsilon\|_{\widetilde{L}_t^\infty(\dot{B}_{2,1}^{\frac{d}{2}})}
				\lesssim \| a^\varepsilon\|_{\widetilde{L}_t^\infty(\dot{B}_{2,1}^{\frac{d}{2}})}^{\ell,\var}
				+\| a^\varepsilon\|_{\widetilde{L}_t^\infty(\dot{B}_{2,1}^{\frac{d}{2}})}^{h,\var}\\
				&\lesssim 
				\frac{1}{\var}\|\delta a\|_{\widetilde{L}_t^\infty(\dot{B}_{2,1}^{\frac{d}{2}-1})}^{\ell,\var}
				+\|a^*\|_{\widetilde{L}_t^\infty(\dot{B}_{2,1}^{\frac{d}{2}})}+\var\| a^\varepsilon\|_{\widetilde{L}_t^\infty(\dot{B}_{2,1}^{\frac{d}{2}+1})}^{h,\var}\\
				&\lesssim \delta X(t)+E_0\lesssim \zeta_1+E_0\leq \frac{1}{2},	
			\end{align*}
			which implies that 
			\begin{align*}
				\frac{2}{3}\leq \frac{1}{1+a^\var}\leq 2.
			\end{align*}
			Then we have 
			\begin{align*}
				D_j^{h,\varepsilon}(t)&\geq \frac{1}{\var^2}\|w_j^\var\|_{L^2}^2+\frac{\eta_0}{2^{2j}\var^2 }(\|\nabla a_j^\var\|_{L^2}^2-2\|\nabla w_j^\var\|_{L^2}^2-\frac{1}{2\var^2}\|w_j^\var\|_{L^2}^2-\frac{1}{2}\|\nabla a_j^\var\|_{L^2}^2)\\
				&\geq \frac{2c}{\var^2}(1-C\eta_0)(\| a_j^\var\|_{L^2}^2+\|w_j^\var\|_{L^2}^2),
			\end{align*}
			for some uniform constants $c, C>0$. Here we have used  $2^{J_\epsilon}>\frac{1}{\var}$ and the fact that 
			\begin{align*}
				\frac{1}{\varepsilon}\Big|\int_{\mathbb{R}^d}w_j^\varepsilon\cdot\nabla a_j^\varepsilon dx\Big|
				\leq \frac{1}{\varepsilon}\|w_j^\varepsilon\|_{L^2}\|\nabla a_j^\varepsilon\|_{L^2}	
				\leq \frac{1}{2\var^2}\|w_j^\varepsilon\|_{L^2}^2+\frac{1}{2}\|\nabla a_j^\varepsilon\|_{L^2}^2.
			\end{align*}
	We can choose a suitable constant $\eta_0$ independent of $\varepsilon$ such that 
			\begin{align}\label{3.44}
				L_j^{h,\var}(t)\sim \|(a_j^\varepsilon,w_j^\varepsilon)\|_{L^2}^2,\quad
				D_j^{h,\var}(t)\geq \frac{c}{\varepsilon^2}L_j^{h,\var}(t).	
			\end{align}
            
			Consequently, by \eqref{en1}-\eqref{3.44} we obtain 
			\begin{align}
				&\frac{d}{dt}L_j^{h,\var}(t)+\frac{c}{\varepsilon^2}L_j^{h,\var}(t)\nonumber\\
                &\lesssim \frac{d}{dt}L_j^{h,\var}(t)+D_j^{h,\var}(t)\nonumber\\
				&\lesssim 
				\Big((\|\nabla w^\varepsilon\|_{L^2}+\|\nabla a^\varepsilon\|_{L^\infty}+\|w^\varepsilon\|_{L^\infty})\frac{1}{\varepsilon}\|(a_j^\varepsilon,w_j^\varepsilon)\|_{L^2}^2+\|\partial_ta^\varepsilon\|_{L^\infty}
				\|a_j^\varepsilon\|_{L^2}\nonumber\\
				&\quad+\frac{1}{\varepsilon}\|(\mathcal{N}^\varepsilon_{1j},\mathcal{N}_{2j}^\var,\mathcal{N}_{3j}^\var,\mathcal{N}^\varepsilon_{4j})\|_{L^2}
				+\frac{1}{\varepsilon}\|u_j^\varepsilon\|_{L^2}	\Big)\sqrt{L_j^{h,\var}(t)}.\label{Lyhigh}
			\end{align}

            With the aid of the inequality \eqref{Lyhigh}, we establish uniform estimates in the high-frequency regime. From \eqref{Lyhigh}, we arrive at
			\begin{align*}
				&\varepsilon \|(a_j^\varepsilon,w_j^\varepsilon)\|_{L_t^\infty (L^2)}+\frac{1}{\varepsilon}\|(a_j^\varepsilon,w_j^\varepsilon)\|_{L_t^1(L^2)}\\
				&\lesssim \varepsilon \|\dot{\Delta}_j(a_{0},w_{0})\|_{L^2}
				+\Big(\|\nabla w^\varepsilon\|_{L_t^2(L^\infty)}+\|\nabla a^\varepsilon\|_{L_t^2(L^\infty)}+\|w^\varepsilon\|_{L_t^2(L^\infty)}\Big)\|(a_j^\varepsilon,w_j^\varepsilon)\|_{L_t^2(L^2)}\\	
				&\quad+\varepsilon\|\partial_ta^\varepsilon\|_{L_t^1L^\infty}\|(a_j^\varepsilon,w_j^\varepsilon)\|_{L_t^\infty(L^2)}
				+\|(\mathcal{N}^\varepsilon_{1j},\mathcal{N}^\varepsilon_{2j},\mathcal{N}^\varepsilon_{3j},\mathcal{N}^\varepsilon_{4j})\|_{L_t^1(L^2)}
				+\|u_j^\varepsilon\|_{L_t^1(L^2)}.
			\end{align*}
		Summing the above inequality over  the regime $j\geq J_\var-1$, one obtains 
			\begin{align}
				&\varepsilon \|(a^\varepsilon,w^\varepsilon)\|_{\widetilde{L}_t^\infty(\dot{B}_{2,1}^{\frac{d}{2}+1})}^{h,\varepsilon}
				+\frac{1}{\varepsilon}\|(a^\varepsilon,w^\varepsilon)\|_{\widetilde{L}_t^1(\dot{B}_{2,1}^{\frac{d}{2}+1})}^{h,\varepsilon}\nonumber\\
				&\lesssim  \|(a_0^\varepsilon,w_0^\varepsilon)\|_{\widetilde{L}_t^\infty(\dot{B}_{2,1}^{\frac{d}{2}+1})}^{h,\varepsilon}
				+\|(\nabla w^\varepsilon,w^\varepsilon,\nabla a^\varepsilon)\|_{\widetilde{L}_t^2(\dot{B}_{2,1}^{\frac{d}{2}})}^{h,\varepsilon}\|(a^\varepsilon,w^\varepsilon)\|_{\widetilde{L}_t^2(\dot{B}_{2,1}^{\frac{d}{2}+1})}^{h,\varepsilon}\nonumber\\
				&\quad+\|\partial_ta^\varepsilon\|_{\widetilde{L}_t^1(\dot{B}_{2,1}^{\frac{d}{2}})}\varepsilon\|(a^\varepsilon,w^\varepsilon)\|_{\widetilde{L}_t^\infty(\dot{B}_{2,1}^{\frac{d}{2}+1})}^{h,\varepsilon}\nonumber\\
        &\quad+\sum_{2^{J_\var}>\frac{1}{\varepsilon}}2^{(\frac{d}{2}+1)j}\|(\mathcal{N}^\varepsilon_{1j},\mathcal{N}^\varepsilon_{2j},\mathcal{N}^\varepsilon_{3j},\mathcal{N}^\varepsilon_{4j})\|_{L_t^1(L^2)}+\|u^\varepsilon\|_{\widetilde{L}_t^1(\dot{B}_{2,1}^{\frac{d}{2}+1})}^{h,\varepsilon}.\label{TYU4}
                \end{align}
Gathering the commutator estimates in Lemma \ref{lemma25} leads to          
                \begin{align}
					&\sum_{2^{J_\var}>\frac{1}{\varepsilon}}2^{(\frac{d}{2}+1)j}\|(\mathcal{N}^\varepsilon_{1j},\mathcal{N}^\varepsilon_{2j},\mathcal{N}^\varepsilon_{3j},\mathcal{N}^\varepsilon_{4j})\|_{L_t^1(L^2)}\lesssim \|w^\varepsilon\|_{\widetilde{L}_t^2(\dot{B}_{2,1}^{\frac{d}{2}+1})}\|(w^\varepsilon,a^\varepsilon)\|_{\widetilde{L}_t^2(\dot{B}_{2,1}^{\frac{d}{2}+1})}.\label{TYU41}
			\end{align}			
By using the uniform bound \eqref{r0}, the frequency cutoff property \eqref{lh} and the definition of $\delta X(t)$, we are able to demonstrate that 
			\begin{align}
				&\|w^\varepsilon\|_{\widetilde{L}_t^2(\dot{B}_{2,1}^{\frac{d}{2}+1})}+\|a^\varepsilon\|_{\widetilde{L}_t^2(\dot{B}_{2,1}^{\frac{d}{2}+1})}+	\|(a^\varepsilon,w^\varepsilon)\|_{\widetilde{L}_t^2(\dot{B}_{2,1}^{\frac{d}{2}+1})}\lesssim E_0+\delta X(t).\label{TYU3}
			\end{align}
Thus, substituting \eqref{TYU41}, \eqref{TYU3} and \eqref{347} into \eqref{TYU4} yields that 
			\begin{align}
				&\varepsilon\|(a^\varepsilon,w^\varepsilon)\|_{\widetilde{L}_t^\infty(\dot{B}_{2,1}^{\frac{d}{2}+1})}^{h,\varepsilon}
				+\frac{1}{\varepsilon}\|(a^\varepsilon,w^\varepsilon)\|_{\widetilde{L}_t^1(\dot{B}_{2,1}^{\frac{d}{2}+1})}^{h,\varepsilon}\nonumber\\
				&\lesssim  E_0+\delta E_0+E_0\delta X(t)
				+(\delta X(t))^2+\|u^*\|_{\widetilde{L}_t^1(\dot{B}_{2,1}^{\frac{d}{2}+1})}^{h,\varepsilon}+\|\delta u\|_{\widetilde{L}_t^1(\dot{B}_{2,1}^{\frac{d}{2}+1})}^{h,\varepsilon}\nonumber\\
				&\lesssim E_0+\delta E_0+(\delta X(t))^2+\|\delta u^\varepsilon\|_{\widetilde{L}_t^1(\dot{B}_{2,1}^{\frac{d}{2}+1})}.\label{091401}
			\end{align}

 \textbf{Part II.}           
The next step is to derive the estimates of $\mathcal{Z}^\varepsilon$ and $w^\varepsilon$ in the high-frequency part as follows:
			\begin{align*}
				\|\mathcal{Z}^\varepsilon\|_{\widetilde{L}_t^1(\dot{B}_{2,1}^{\frac{d}{2}})}^{h,\varepsilon}
				&\lesssim 
				\|\mathcal{P}^\top w^\varepsilon\|_{\widetilde{L}_t^1(\dot{B}_{2,1}^{\frac{d}{2}})}^{h,\varepsilon}
				+\varepsilon \|\nabla a^\varepsilon\|_{\widetilde{L}_t^1(\dot{B}_{2,1}^{\frac{d}{2}})}^{h,\varepsilon}	
				+\varepsilon \|h(a^\varepsilon)\nabla a^\varepsilon\|_{\widetilde{L}_t^1(\dot{B}_{2,1}^{\frac{d}{2}})}^{h,\varepsilon}\\
				&\lesssim 	\varepsilon\|(a^\varepsilon, w^\varepsilon)\|_{\widetilde{L}_t^1(\dot{B}_{2,1}^{\frac{d}{2}+1})}^{h,\varepsilon}+\varepsilon \|h(a^\varepsilon)\nabla a^\varepsilon\|_{\widetilde{L}_t^1(\dot{B}_{2,1}^{\frac{d}{2}})}^{h,\varepsilon}\\
				&\lesssim \var^2(E_0+\delta E_0+(\delta X(t))^2)+\varepsilon \|h(a^\varepsilon)\nabla a^\varepsilon\|_{\widetilde{L}_t^1(\dot{B}_{2,1}^{\frac{d}{2}})}^{h,\varepsilon}.
			\end{align*}
			Noting that 
			\begin{align*}
				\varepsilon \|h(a^\varepsilon)\nabla a^\varepsilon\|_{\widetilde{L}_t^1(\dot{B}_{2,1}^{\frac{d}{2}})}^{h,\varepsilon}
				&\lesssim \varepsilon \|h(a^\varepsilon)\nabla (a^\varepsilon)^\ell\|_{\widetilde{L}_t^1(\dot{B}_{2,1}^{\frac{d}{2}})}^{h,\varepsilon}
				+ \varepsilon \|h(a^\varepsilon)\nabla (a^\varepsilon)^h\|_{\widetilde{L}_t^1(\dot{B}_{2,1}^{\frac{d}{2}})}^{h,\varepsilon}\\
				&\lesssim \varepsilon^2 \|h(a^\varepsilon)\|_{\widetilde{L}_t^\infty(\dot{B}_{2,1}^{\frac{d}{2}})}\|\nabla (a^\varepsilon)^\ell\|_{\widetilde{L}_t^1(\dot{B}_{2,1}^{\frac{d}{2}+1})}
				\\
				&\quad+\varepsilon \|h(a^\varepsilon)\|_{\widetilde{L}_t^\infty(\dot{B}_{2,1}^{\frac{d}{2}})}\|\nabla (a^\varepsilon)^h\|_{\widetilde{L}_t^1(\dot{B}_{2,1}^{\frac{d}{2}})}\\
				&\lesssim  \varepsilon^2\|a^\varepsilon\|_{\widetilde{L}_t^\infty(\dot{B}_{2,1}^{\frac{d}{2}})}\|a^\varepsilon\|_{\widetilde{L}_t^1(\dot{B}_{2,1}^{\frac{d}{2}+2})}^{\ell,\varepsilon}
				+\varepsilon\|a^\varepsilon\|_{\widetilde{L}_t^\infty(\dot{B}_{2,1}^{\frac{d}{2}})}\|a^\varepsilon\|_{\widetilde{L}_t^1(\dot{B}_{2,1}^{\frac{d}{2}+1})}^{h,\varepsilon}\\
				&\lesssim \varepsilon^2(E_0+\delta E_0+(\delta X(t))^2),
			\end{align*}
			which implies that 
			\begin{align}\label{091402}
				\frac{1}{\varepsilon^2}\|\mathcal{Z}^\varepsilon\|_{\widetilde{L}_t^1(\dot{B}_{2,1}^{\frac{d}{2}})}^{h,\varepsilon}\lesssim 
				E_0+\delta E_0+(\delta X(t))^2.
			\end{align}
            Furthermore, one also obtains from \eqref{r0}, \eqref{lh} and  \eqref{high1} that
            \begin{align*}
            \frac{1}{\varepsilon^2}\|\mathcal{R}^\varepsilon\|_{\widetilde{L}_t^1(\dot{B}_{2,1}^{\frac{d}{2}})}^{h,\varepsilon}\lesssim \|(w^\var, u^\var)\|_{\widetilde{L}^1_t(\dot{B}^{\frac{d}{2}+1}_{2,1})}^{h,\var}\lesssim E_0+\delta E_0+(\delta X(t))^2+\|\delta u^\var\|_{\widetilde{L}^1_t(\dot{B}^{\frac{d}{2}+1}_{2,1})}.
            \end{align*}
Then, we combine \eqref{091401} and \eqref{091402} to show \eqref{high1} and  finish the proof of this lemma.
		\end{proof}	

        \vspace{2mm}

		\noindent{\it\textbf{Proof of Proposition \ref{P31}.}\ }		 Multiplying \eqref{high1} by a sufficiently small (uniform) constant and adding the resulting inequality to \eqref{step1:es} and \eqref{lowaw}, we obtain that there exist uniform constants $C_0^*, \var_0^*$ such that for $\var\leq \var_0^*$.
		\begin{align}
			\delta X(t)\leq C_0^*(E_0+\delta E_0+\var\delta X(t)+(\delta X(t))^2).
		\end{align}
		Then, under the condition \eqref{asss}, one has 
        \begin{align*}
        \delta X(t)\leq C_0^*(E_0+\delta E_0)+(C_0^*\var+C_0^*\zeta_1)\delta X(t).
        \end{align*}
        Now we take 
        \begin{align}\label{var1}
        \var\leq \var_0=\min\Big\{\var_0^*,\frac{1}{4C_0^*}\Big\},\quad \zeta_1=\frac{1}{4C_0^*},
        \end{align}
        such that  \eqref{3.3} 
        with $C_0=2C_0^*$ holds. To this end, we require $E_0+\delta E_0\leq \frac{1}{C_0}\zeta_1$. Therefore, we complete the proof of Proposition \ref{P31}. \hfill $\square$

\vspace{3mm}
    Now we are ready to prove Theorem \ref{thm2}.\\
     \noindent{\it\textbf{Proof of Theorem \ref{thm2}.}\ }   
	Let $(a^\var,w^\var,u^\var)$ be a local solution on $[0,T_0]\times \mathbb{R}^d$ with a positive time $T_0$ such that $\sup_{t\in[0,T_0]}\delta X(t)\leq \zeta_1$. Then, employing Proposition \ref{P31}, we have
\begin{align}
\delta X(t)\leq C_0(E_0+\delta E_0)\leq \frac{\zeta_1}{2},
\end{align}
provided $E_0+\delta E_0\leq \frac{\zeta_1}{2C_0}$. By a standard continuous method, we can extend the local solution to a global-in-time solution that satisfies $\sup_{t\in \mathbb{R}_+} \delta X(t) \leq 2C_0( E_0+\delta E_0)$. Owing to the definition of $E^\var(t), D^\var(t)$, \eqref{r0} and the uniform bound of $\delta X(t)$ at hand, one can recover the estimate of  $E^\var(t), D^\var(t)$ as follows:
		\begin{align*}	E^\varepsilon(t)+D^\varepsilon(t)&\lesssim \delta X(t)+\|a^*\|_{\widetilde{L}^{\infty}_t(\dot{B}^{\frac{d}{2}}_{2,1}\cap \dot{B}^{\frac{d}{2}-1}_{2,1})\cap \widetilde{L}^{1}_t(\dot{B}^{\frac{d}{2}+2}_{2,1}\cap \dot{B}^{\frac{d}{2}+1}_{2,1})}+\|u^*\|_{\widetilde{L}^{\infty}_t(\dot{B}^{\frac{d}{2}-1}_{2,1})\cap \widetilde{L}^{1}_t(\dot{B}^{\frac{d}{2}+1}_{2,1})}\\
        &\lesssim E_0+\delta E_0.
		\end{align*}

Furthermore, by \eqref{darcy} and \eqref{B1}, we have 
\begin{align}\label{112002}
\mathcal{P}(\var^{-1}w^\var-W^*)=\var^{-1}\mathcal{P}w^\var-u^*=
\var^{-1}\mathcal{R}^\var+u^\var-u^*,
\end{align}
and 
\begin{align}\label{112003}
\mathcal{P}^\top(\var^{-1}w^\var-W^*)=\var^{-1}\mathcal{P}^\top w^\var+\frac{\nabla \rho^*}{\rho^*}=\var^{-1}\mathcal{Z}^\var-\nabla(\log\rho^\var-\log\rho^*).
\end{align}

Thus, it follows from \eqref{3.3}, \eqref{112002} and \eqref{112003} that 
\begin{align*}
\|\mathcal{P}(\var^{-1}w^\var-W^*)\|_{\widetilde{L}_t^1(\dot{B}_{2,1}^{\frac{d}{2}}+\dot{B}_{2,1}^{\frac{d}{2}+1})}\lesssim \var^{-1}\|\mathcal{R}^\var\|_{\widetilde{L}^1_t(\dot{B}_{2,1}^{\frac{d}{2}})}+\|u^\var-u^*\|_{\widetilde{L}_t^1(\dot{B}_{2,1}^{\frac{d}{2}+1})}\lesssim \var (E_0+\delta E_0),
\end{align*}
and 
\begin{align*}
  \|\mathcal{P}^\top(\var^{-1}w^\var-W^*)\|_{\widetilde{L}_t^1(\dot{B}_{2,1}^\frac{d}{2})}\lesssim \var^{-1}\|\mathcal{Z}^\var\|_{\widetilde{L}_t^1(\dot{B}_{2,1}^{\frac{d}{2}})} +\|\rho^\var-\rho^*\|_{\widetilde{L}^1_t(\dot{B}_{2,1}^{\frac{d}{2}+1})}\lesssim \var(E_0+\delta E_0).
\end{align*}
 Thus, we complete the proof of Theorem  \ref{thm2}.\hfill $\square$

		\section{Proof of Theorem \ref{thm3}}\label{S5}
In this section, we establish the large-time behaviors of the solution and the error in Besov spaces.
		
		\subsection{Propagation of regularity in \texorpdfstring{$\dot{B}^{\sigma_1}_{2,\infty}$}{B-sigma1-2-infty}}

        		At the beginning, we establish the uniform boundedness of the solution in $\dot{B}^{\sigma_1}_{2,\infty}$, which plays a key role in deriving uniform time-decay estimates.
		
		\begin{prop}
			Assume $(a^\varepsilon, w^\varepsilon,u^\varepsilon)$ is a solution of the system \eqref{main3}. Under the assumptions in Theorem \ref{thm3}, it follows that
			\begin{align}\label{evolution}
				\|(a^\var, w^\var, u^\var)\|_{\widetilde{L}_t^{\infty}(\dot{B}^{\sigma_1}_{2,\infty})}
				+\frac{1}{\var}\|\delta a\|_{\widetilde{L}_t^{\infty}(\dot{B}^{\sigma_1}_{2,\infty})}+\frac{1}{\var}	\| \delta u\|_{\widetilde{L}_t^{\infty}(\dot{B}^{\sigma_1}_{2,\infty})} \lesssim 1.
			\end{align}
			In addition, if $-\frac{d}{2}\leq \sigma_1<\frac{d}{2}-2$, then we have
			\begin{align}\label{improve}
				\|\delta a\|_{\widetilde{L}_t^{\infty}(\dot{B}^{\sigma_1-1}_{2,\infty})}\lesssim \var.
			\end{align}
		\end{prop}
		\begin{proof}
			We define the error energy functional 
			\begin{align}
				\delta\mathcal{X}_{\sigma_1}&= \frac{1}{\var} \|\delta a\|_{\widetilde{L}_t^{\infty}(\dot{B}^{\sigma_1}_{2,\infty})}+\frac{1}{\var}\|\delta a\|_{\widetilde{L}_t^1(\dot{B}^{\sigma_1+2}_{2,\infty})}+\|w^\var\|_{\widetilde{L}^{\infty}_t(\dot{B}^{\sigma_1}_{2,\infty})}\nonumber\\
				&\quad+\frac{1}{\var}\|\delta u\|_{\widetilde{L}_t^{\infty}(\dot{B}^{\sigma_1}_{2,\infty})}+\frac{1}{\var}\| \delta u\|_{\widetilde{L}_t^1(\dot{B}^{\sigma_1+2}_{2,\infty})}+\frac{1}{\var^2}\big(\|\mathcal{Z}^\var\|_{\widetilde{L}_t^1(\dot{B}_{2,\infty}^{\sigma_1})}+\|\mathcal{R}^\var\|_{\widetilde{L}_t^1(\dot{B}_{2,\infty}^{\sigma_1})}\big).\label{X1}
			\end{align}  
Recall that $\delta a(0,x)=\delta u(0,x)=0$. Following the same line as in the proofs of Lemma \ref{P3.2}-\ref{P3.3} (replacing the product law \eqref{uv2} with \eqref{uv3}), we establish the following estimates
\begin{align}
& \frac{1}{\varepsilon}\|\delta u\|_{{\widetilde{L}_t^\infty(\dot{B}_{2,\infty}^{\sigma_1})}}+\frac{1}{\varepsilon}
			\|\delta u\|_{{\widetilde{L}_t^1(\dot{B}_{2,\infty}^{\sigma_1+2})}}+\|\mathcal{P}w^\varepsilon\|_{\widetilde{L}_t^\infty(\dot{B}_{2,\infty}^{\sigma_1})}
			\nonumber\\
			&\quad+\|\mathcal{P}w^\varepsilon\|_{\widetilde{L}_t^\infty(\dot{B}_{2,\infty}^{\sigma_1+1})}^{\ell,\var}+\frac{1}{\var}\|\mathcal{P}w^\varepsilon\|_{{\widetilde{L}^2_t(\dot{B}_{2,\infty}^{\sigma_1})}}^{\ell,\varepsilon}+\frac{1}{\varepsilon}
			\|\mathcal{P}w^\varepsilon\|_{{\widetilde{L}_t^1(\dot{B}_{2,\infty}^{\sigma_1+2})}}^{\ell,\varepsilon}\nonumber\\
			&\quad +\frac{1}{\varepsilon^2}\|\mathcal{R}^\varepsilon\|_{{\widetilde{L}_t^1(\dot{B}_{2,\infty}^{\sigma_1})}}+\frac{1}{\varepsilon^2}\|\mathcal{R}^\varepsilon\|_{{\widetilde{L}_t^1(\dot{B}_{2,\infty}^{\sigma_1+1})}}^{\ell,\varepsilon}\nonumber\\
                &\lesssim \|w_0^\var\|_{\dot{B}^{\sigma_1}_{2,\infty}}+E_0+\delta E_0+(E_0+\delta E_0)\delta\mathcal{X}_{\sigma_1}\nonumber\\
                &\quad+ \|a^*\|_{\widetilde{L}_t^\infty( {\dot{B}^{\sigma_1}_{2,\infty}})\cap \widetilde{L}_t^1( \dot{B}_{2,\infty}^{\sigma_1})}+\|u^*\|_{\widetilde{L}_t^\infty(\dot{B}_{2,\infty}^{\sigma_1})\cap \widetilde{L}_t^1(\dot{B}_{2,\infty}^{\sigma_1+2})},\label{sigma1:es}
\end{align}
and
			\begin{align}
					&\frac{1}{\varepsilon}\|\delta a\|_{\widetilde{L}_t^\infty(\dot{B}_{2,\infty}^{\sigma_1})}^{\ell,\varepsilon}
					+\frac{1}{\varepsilon}\|\delta a\|_{\widetilde{L}_t^1(\dot{B}_{2,\infty}^{\sigma_1+2})}^{\ell,\varepsilon}+\frac{1}{\varepsilon}\|\partial_t \delta a\|_{\widetilde{L}_t^1(\dot{B}_{2,\infty}^{\sigma_1})}^{\ell,\varepsilon}\nonumber\\
                    &\quad
					+\|\mathcal{P}^{\top}w^\varepsilon\|_{\widetilde{L}_t^\infty(\dot{B}_{2,\infty}^{\sigma_1+1})}^{\ell,\varepsilon}+\frac{1}{\var}\|\mathcal{P}^{\top}w^\varepsilon\|_{\widetilde{L}_t^1(\dot{B}_{2,\infty}^{\sigma_1+2})}^{\ell,\varepsilon}+	\frac{1}{\varepsilon}\|\mathcal{P}^\top w^\varepsilon\|_{\widetilde{L}_t^2(\dot{B}_{2,\infty}^{\sigma_1+1})}^{\ell,\varepsilon}+\frac{1}{\varepsilon^2}\|\mathcal{Z}^\varepsilon\|_{\widetilde{L}_t^1(\dot{B}_{2,\infty}^{\sigma_1+1})}^{\ell,\varepsilon}\nonumber\\
					&\lesssim E_0+\delta E_0+\|w_0^\var\|_{\dot{B}^{\sigma_1}_{2,\infty}}+(E_0+\delta E_0)\delta \mathcal{X}_{\sigma_1}\nonumber\\
                   &\quad + \|a^*\|_{\widetilde{L}_t^\infty( {\dot{B}^{\sigma_1}_{2,\infty}})\cap \widetilde{L}_t^1( \dot{B}_{2,\infty}^{\sigma_1})}+\|u^*\|_{\widetilde{L}_t^\infty(\dot{B}_{2,\infty}^{\sigma_1})\cap \widetilde{L}_t^1(\dot{B}_{2,\infty}^{\sigma_1+2})}.\label{sigma1:es1}
				\end{align}

Then, due to $\sigma_1<\frac{d}{2}-1$, from the uniform estimates \eqref{r0} and \eqref{r1} as well as the high-frequency cutoff property, it holds
\begin{align}\label{vvv3}
\frac{1}{\var} \|\delta a\|_{\widetilde{L}_t^{\infty}(\dot{B}^{\sigma_1}_{2,\infty})}^{h,\var}+\frac{1}{\var}\|\delta a\|_{\widetilde{L}_t^1(\dot{B}^{\sigma_1+2}_{2,\infty})}^{h,\var}&\lesssim \|(a^\var,a^*)\|_{\widetilde{L}_t^{\infty}(\dot{B}^{\frac{d}{2}}_{2,1})}^{h,\var}+\|a^*\|_{\widetilde{L}_t^{1}(\dot{B}^{\frac{d}{2}+2}_{2,1})}^{h,\var}+\frac{1}{\var}\| a^\var\|_{\widetilde{L}_t^1(\dot{B}^{\frac{d}{2}+1}_{2,1})}^{h,\var}\nonumber\\
&\lesssim E_0+\delta E_0.
\end{align}

We further establish some additional estimates of $\mathcal{Z}^\var$. In accordance with Lemma \ref{maximaldamped} applied to $\eqref{A6}_2$ and performing the usual product laws as well as interpolation inequalities, we can derive
\begin{align}
&\|\mathcal{Z}^\var\|_{\widetilde{L}^{\infty}_t(\dot{B}^{\sigma_1}_{2,\infty})}+\frac{1}{\var^2} \|\mathcal{Z}^\var\|_{\widetilde{L}^{1}_t(\dot{B}^{\sigma_1}_{2,\infty})}\nonumber\\
&\lesssim \|\mathcal{Z}_0^\var\|_{\dot{B}^{\sigma_1}_{2,\infty}}+\|\nabla \partial_t \tilde{h}(a^\var)\|_{\widetilde{L}^{1}_t(\dot{B}^{\sigma_1}_{2,\infty})}+\frac{1}{\var}\|w^\var\cdot \nabla w^\var\|_{\widetilde{L}^{1}_t(\dot{B}^{\sigma_1}_{2,\infty})} \nonumber\\
&\lesssim \|w_0^\var\|_{\dot{B}^{\sigma_1}_{2,\infty}}+\|a_0^\var\|_{\dot{B}^{\sigma_1+1}_{2,\infty}}+ \|(\partial_t \delta a,\partial_t a^*)\|_{\widetilde{L}^1_t(\dot{B}^{\sigma_1+1}_{2,\infty})}+\frac{1}{\var}\|w^\var\|_{\widetilde{L}^2_t(\dot{B}^{\sigma_1+1}_{2,\infty})}\|w^\var\|_{\widetilde{L}^2_t(\dot{B}^{\frac{d}{2}}_{2,1})} \nonumber\\
&\lesssim \|w_0^\var\|_{\dot{B}^{\sigma_1}_{2,\infty}}+E_0+\delta E_0+(E_0+\delta E_0)\delta \mathcal{X}_{\sigma_1}\nonumber \\
&\quad+ \|a^*\|_{\widetilde{L}_t^\infty( {\dot{B}^{\sigma_1}_{2,1}})\cap \widetilde{L}_t^1( \dot{B}_{2,\infty}^{\sigma_1})}+\|u^*\|_{\widetilde{L}_t^\infty(\dot{B}_{2,\infty}^{\sigma_1})\cap \widetilde{L}_t^1(\dot{B}_{2,\infty}^{\sigma_1+2})}+\|\partial_t a^*\|_{\widetilde{L}_t^1(\dot{B}_{2,\infty}^{\sigma_1})}.\label{47}
\end{align}

Combining the estimates \eqref{sigma1:es}-\eqref{47}, we obtain
\begin{align}
\delta \mathcal{X}_{\sigma_1}&\lesssim \|w_0^\var\|_{\dot{B}^{\sigma_1}_{2,\infty}}+E_0+\delta E_0+(E_0+\delta E_0)\delta \mathcal{X}_{\sigma_1}\nonumber \\
&\quad+ \|a^*\|_{\widetilde{L}_t^\infty( {\dot{B}^{\sigma_1}_{2,1}})\cap \widetilde{L}_t^1( \dot{B}_{2,\infty}^{\sigma_1})}+\|u^*\|_{\widetilde{L}_t^\infty(\dot{B}_{2,\infty}^{\sigma_1})\cap \widetilde{L}_t^1(\dot{B}_{2,\infty}^{\sigma_1+2})}+\|\partial_t a^*\|_{\widetilde{L}_t^1(\dot{B}_{2,\infty}^{\sigma_1})}.
\end{align}
Due to \eqref{AP9} and the fact that $E_0+\delta E_0$ is small enough, one arrives at the error estimate
\begin{align}
\delta \mathcal{X}_{\sigma_1}\lesssim 1.\label{deltaXboind}
\end{align}
This, together with \eqref{AP9}, yields the bound for the solution $(a^\var,w^\var,u^\var)$:
\begin{align}
\|(a^\var, w^\var, u^\var)\|_{\widetilde{L}_t^{\infty}(\dot{B}^{\sigma_1}_{2,\infty})}\lesssim \|(\delta a, w^\var, \delta u)\|_{\widetilde{L}_t^{\infty}(\dot{B}^{\sigma_1}_{2,\infty})}+\|(a^*,u^*)\|_{\widetilde{L}_t^{\infty}(\dot{B}^{\sigma_1}_{2,\infty})}\lesssim 1.
\end{align}
Consequently, \eqref{evolution} is proved.

        Furthermore, we are able to establish lower-order estimates of $\delta a$ if $-\frac{d}{2}\leq \sigma_1<\frac{d}{2}-2$. Indeed, using $\div u^\var=\div u^*=0$, one can rewrite the equation of $\delta a$ by
			\begin{equation}
				\left\{
				\begin{aligned}\label{pgggh}
					&\partial_t (\delta a)-\Delta (\delta a)=\div \Big(-u^\varepsilon \otimes \delta a-\delta u \otimes a^*+\frac{1}{\varepsilon}(1+a^\varepsilon)(\mathcal{Z}^\var+\mathcal{R}^\var)\Big),\\
					&\delta a(0,x)=0.
				\end{aligned}
				\right.
			\end{equation}
			Then, applying Lemma \ref{lemma33} to \eqref{pgggh}, we have
			\begin{align}\label{ssfgggghh}
				&\frac{1}{\var}\|\delta a\|_{\widetilde{L}_t^{\infty}(\dot{B}^{\sigma_1-1}_{2,\infty})}+\frac{1}{\var}\|\delta a\|_{\widetilde{L}_t^1(\dot{B}^{\sigma_1+1}_{2,\infty})}\nonumber\\
				&\lesssim \frac{1}{\var}\|(u^\var \otimes \delta a, \delta u\otimes a^*)\|_{\widetilde{L}_t^1(\dot{B}^{\sigma_1}_{2,\infty})}+\frac{1}{\var^2}\|(1+a^\varepsilon)(\mathcal{Z}^\var+\mathcal{R}^\var)\|_{\widetilde{L}_t^1(\dot{B}^{\sigma_1}_{2,\infty})}\nonumber\\
				&\lesssim \frac{1}{\var}\|u^\var \|_{\widetilde{L}_t^2(\dot{B}^{\frac{d}{2}}_{2,1})}\|\delta a\|_{\widetilde{L}_t^2(\dot{B}^{\sigma_1}_{2,\infty})}+\frac{1}{\var}\|\delta u\|_{\widetilde{L}_t^\infty(\dot{B}^{\sigma_1}_{2,\infty})} \|a^*\|_{\widetilde{L}_t^1(\dot{B}^{\frac{d}{2}}_{2,1})}\nonumber\\
				&\quad+\frac{1}{\var^2}(1+\|a^\var\|_{\widetilde{L}_t^{\infty}(\dot{B}^{\frac{d}{2}}_{2,1})})\|(\mathcal{Z}^\var, \mathcal{R}^\var)\|_{\widetilde{L}_t^1(\dot{B}^{\sigma_1}_{2,\infty})}\nonumber\\
				&\lesssim \frac{1}{\var}\|u^\var \|_{\widetilde{L}_t^2(\dot{B}^{\frac{d}{2}}_{2,1})}(\|\delta a\|_{\widetilde{L}_t^\infty(\dot{B}^{\sigma_1-1}_{2,\infty})}+\|\delta a\|_{\widetilde{L}_t^1(\dot{B}^{\sigma_1+1}_{2,\infty})})+\frac{1}{\var}\|\delta u\|_{\widetilde{L}_t^\infty(\dot{B}^{\sigma_1}_{2,\infty})} \|a^*\|_{\widetilde{L}_t^1(\dot{B}^{\frac{d}{2}}_{2,1})}\nonumber\\
				&\quad+\frac{1}{\var^2}(1+\|a^\var\|_{\widetilde{L}_t^{\infty}(\dot{B}^{\frac{d}{2}}_{2,1})})\|(\mathcal{Z}^\var, \mathcal{R}^\var)\|_{\widetilde{L}_t^1(\dot{B}^{\sigma_1}_{2,\infty})}.
			\end{align}
			Since $\frac{d}{2}> \sigma_1+2$ due to $-\frac{d}{2}\leq \sigma_1< \frac{d}{2}-2$, we discover by interpolation that 
			\begin{align*}
				&\|a^*\|_{\widetilde{L}_t^1(\dot{B}^{\frac{d}{2}}_{2,1})}\lesssim \|a^*\|_{\widetilde{L}_t^1(\dot{B}^{\sigma_1+2}_{2,\infty})}
				+\|a^*\|_{\widetilde{L}_t^1(\dot{B}^{\frac{d}{2}+1}_{2,1})}\lesssim E_0+1\lesssim 1.
			\end{align*}

			Consequently, making use of the known estimates \eqref{r1} and \eqref{evolution}, we know that 
			\begin{align}
				\|\delta a\|_{\widetilde{L}_t^{\infty}(\dot{B}^{\sigma_1-1}_{2,\infty})}+\|\delta a\|_{\widetilde{L}_t^1(\dot{B}^{\sigma_1+1}_{2,\infty})}\lesssim \var.
			\end{align}
			Thus, we complete the proof of this proposition.		
		\end{proof}
		\subsection{Uniform time-weighted estimates}
		With the uniform boundedness of the solution in $\dot{B}_{2,\infty}^{\sigma_1}$, we are in a position to obtain the optimal time-decay rates of  the solution.
		
		At the beginning, we define the time-weighted functional as follows
		\begin{align}
			\delta\mathcal{D}(t)=&
			\|\tau^M \delta a\|_{\widetilde{L}_t^\infty(\dot{B}^{\frac{d}{2}}_{2,1})\cap \widetilde{L}_t^1(\dot{B}^{\frac{d}{2}+2}_{2,1})}^{\ell,\var}+\frac{1}{\var}\|\tau^M \delta u\|_{\widetilde{L}_t^\infty(\dot{B}^{\frac{d}{2}}_{2,1})\cap \widetilde{L}_t^1(\dot{B}^{\frac{d}{2}+2}_{2,1})}\nonumber\\
			&+\|\tau^M w^\var\|_{\widetilde{L}_t^\infty( \dot{B}_{2,1}^{\frac{d}{2}})}+\frac{1}{\var}
			\|\tau^M w^\var\|_{\widetilde{L}_t^2(\dot{B}_{2,1}^{\frac{d}{2}})}\nonumber\\
            &
			+\var \|\tau^M(a^\var,w^\var)\|_{\widetilde{L}_t^\infty(\dot{B}_{2,1}^{\frac{d}{2}+1})}^{h,\var}
			+\frac{1}{\var}\|\tau^M(a^\var,w^\var)\|_{\widetilde{L}_t^1(\dot{B}_{2,1}^{\frac{d}{2}+1})}^{h,\var}\nonumber\\
			&+\|\tau^M(\mathcal{Z}^\var,\mathcal{R}^\var)\|_{\widetilde{L}_t^{\infty}(\dot{B}_{2,1}^{\frac{d}{2}})}+\frac{1}{\var^2}\|\tau^M(\mathcal{Z}^\var,\mathcal{R}^\var)\|_{\widetilde{L}_t^1(\dot{B}_{2,1}^{\frac{d}{2}})}.\label{deltaDY}
		\end{align}
		
		We establish the time-weighted estimates as follows.
		
		\begin{prop}\label{prop41}
			For any given $-\frac{d}{2}\leq \sigma_1<\frac{d}{2}-1$ and 
			$M>\max\{1,\frac{1}{2}(\frac{d}{2}-\sigma_1)\}$, there holds
			\begin{align}\label{timeweighted}
				\delta\mathcal{D}(t)\lesssim \Big( E_0+\delta E_0+\frac{1}{\var}\|\delta a\|_{\widetilde{L}_t^{\infty}(\dot{B}^{\sigma_1}_{2,\infty})}+\frac{1}{\var}	\| \delta u\|_{\widetilde{L}_t^{\infty}(\dot{B}^{\sigma_1}_{2,\infty})}\Big)t^{M-\frac{1}{2}(\frac{d}{2}-\sigma_1)},\quad t>0.
			\end{align}
		\end{prop}
		\begin{proof}
			We carry out similar computations in Section \ref{S4} and use interpolation tricks to derive the time-weighted estimate \eqref{timeweighted}. The main difference is that we establish the $\dot{B}^{d/2}_{2,1}$-regularity instead of the $\dot{B}^{d/2-1}_{2,1}$-regularity for $\delta a$ in low frequencies and  $\delta u$. The proof will be divided into the following three steps.

\noindent
		 {\textbf{Step 1: Estimates of $(\delta u,\mathcal{P}w^\var,\mathcal{R}^\var)$.}}  Multiplying \eqref{main7} by $t^M$, we obtain
			\begin{equation}\label{416}
				\left\{
				\begin{aligned}
					&\partial_t(t^M \delta u)-\mu\Delta (t^M \delta u)= Mt^{M-1}\delta u+\frac{1}{\var}t^M\mathcal{P}(a^\var(\mathcal{Z}^\var+\mathcal{R}^\var))\\
					&\quad-t^M\mathcal{P}(u^\var\cdot\nabla \delta u+\delta u\cdot\nabla u^*)+\frac{1}{\var}t^M\mathcal{R}^\var,\\
					&\partial_t(t^M\mathcal{R}^\var)+(1+\frac{1}{\var^2})(t^M\mathcal{R}^\var)=Mt^{M-1}\mathcal{R}^\var-\var\mu\Delta(t^M\delta u)-\var \mu\Delta(t^M u^*)\\
					&\qquad+\var t^M\mathcal{P}(u^\var\cdot\nabla u^\var)-t^M\mathcal{P}(a^\var(\mathcal{Z}^\var+\mathcal{R}^\var))-\frac{1}{\var}t^M\mathcal{P}(w^\var\cdot\nabla w^\var),\\
					&(t^M\delta u,t^M \mathcal{R}^\var)(x,0)=(0,0).
				\end{aligned}
				\right.
			\end{equation}
			Then, making use of Lemma \ref{lemma33} for $\eqref{416}$, we have 
			\begin{align*}
				&\frac{1}{\var}\|\tau^M\delta u\|_{\widetilde{L}_t^\infty(\dot{B}_{2,1}^{\frac{d}{2}})}+\frac{1}{\var}\|\tau^M\delta u\|_{ \widetilde{L}_t^1(\dot{B}_{2,1}^{\frac{d}{2}+2})}\nonumber\\
				& \lesssim \frac{1}{\var}\|\tau^{M-1}\delta u\|_{\widetilde{L}_t^1(\dot{B}_{2,1}^{\frac{d}{2}})}+\frac{1}{\var^2}\|\tau^M\mathcal{P}(a^\var(\mathcal{Z}^\var+\mathcal{R}^\var))\|_{\widetilde{L}_t^1(\dot{B}_{2,1}^{\frac{d}{2}})}\\
				&\quad+\frac{1}{\var}\|\tau^M\mathcal{P}(u^\var\cdot\nabla \delta u+\delta u\cdot\nabla u^*)\|_{\widetilde{L}_t^1(\dot{B}_{2,1}^{\frac{d}{2}})}+\frac{1}{\var^2} \|\tau^M\mathcal{R}^\var\|_{\widetilde{L}_t^1(\dot{B}_{2,1}^{\frac{d}{2}})},
			\end{align*}
			and
			\begin{align*}
&\|\tau^M\mathcal{R}^\var\|_{\widetilde{L}_t^\infty(\dot{B}_{2,1}^{\frac{d}{2}})}
				+\frac{1}{\var^2}\|\tau^M\mathcal{R}^\var\|_{\widetilde{L}_t^1(\dot{B}_{2,1}^{\frac{d}{2}})}\nonumber\\
				&\lesssim \|\tau^{M-1}\mathcal{R}^\var\|_{\widetilde{L}_t^1(\dot{B}_{2,1}^{\frac{d}{2}})}+\var \|\tau^M\delta u\|_{\widetilde{L}_t^1(\dot{B}_{2,1}^{\frac{d}{2}+2})}+\var \|\tau^M u^*\|_{\widetilde{L}_t^1(\dot{B}_{2,1}^{\frac{d}{2}+2})}\nonumber\\
				&\quad+\var\|\tau^M\mathcal{P}(u^\var\cdot\nabla u^\var)\|_{\widetilde{L}_t^1(\dot{B}_{2,1}^{\frac{d}{2}})}+\|\tau^M\mathcal{P}(a^\var(\mathcal{Z}^\var+\mathcal{R}^\var))\|_{\widetilde{L}_t^1(\dot{B}_{2,1}^{\frac{d}{2}})}+\frac{1}{\var}\|\tau^M\mathcal{P}(w^\var\cdot\nabla w^\var)\|_{\widetilde{L}_t^1(\dot{B}_{2,1}^{\frac{d}{2}})}.
			\end{align*}
			Due to the two inequalities mentioned above and the smallness of $\var$, one has 
			\begin{align}\label{timeener}
				&\frac{1}{\var}\|\tau^M\delta u\|_{\widetilde{L}_t^\infty(\dot{B}_{2,1}^{\frac{d}{2}})}+\frac{1}{\var}\|\tau^M\delta u\|_{ \widetilde{L}_t^1(\dot{B}_{2,1}^{\frac{d}{2}+2})}+\|\tau^M\mathcal{R}^\var\|_{\widetilde{L}_t^\infty(\dot{B}_{2,1}^{\frac{d}{2}})}
				+\frac{1}{\var^2}\|\tau^M\mathcal{R}^\var\|_{\widetilde{L}_t^1(\dot{B}_{2,1}^{\frac{d}{2}})}\nonumber\\
				&\lesssim \var \|\tau^M u^*\|_{\widetilde{L}_t^1(\dot{B}_{2,1}^{\frac{d}{2}+2})}+\frac{1}{\var}\|\tau^{M-1}\delta u\|_{\widetilde{L}_t^1(\dot{B}_{2,1}^{\frac{d}{2}})}+\|\tau^{M-1}\mathcal{R}^\var\|_{\widetilde{L}_t^1(\dot{B}_{2,1}^{\frac{d}{2}})}\nonumber\\
				&\quad+\frac{1}{\var}\|\tau^M\mathcal{P}(u^\var\cdot\nabla \delta u,\delta u\cdot\nabla u^*)\|_{\widetilde{L}_t^1(\dot{B}_{2,1}^{\frac{d}{2}})}+\var\|\tau^M\mathcal{P}(u^\var\cdot\nabla u^\var)\|_{\widetilde{L}_t^1(\dot{B}_{2,1}^{\frac{d}{2}})}\nonumber\\
				&\quad+\frac{1}{\var^2}\|\tau^M\mathcal{P}(a^\var(\mathcal{Z}^\var+\mathcal{R}^\var))\|_{\widetilde{L}_t^1(\dot{B}_{2,1}^{\frac{d}{2}})}+\frac{1}{\var}\|\tau^M\mathcal{P}(w^\var\cdot\nabla w^\var)\|_{\widetilde{L}_t^1(\dot{B}_{2,1}^{\frac{d}{2}})}.
			\end{align}
			According to Appendix A, we have
			\begin{align*}
				\var\|\tau^M u^*\|_{\widetilde{L}_t^1(\dot{B}_{2,1}^{\frac{d}{2}+2})}\lesssim t^{M-\frac{1}{2}(\frac{d}{2}-\sigma_1)}(\|(a^*, u^*)\|_{\widetilde{L}_t^\infty(\dot{B}_{2,\infty}^{\sigma_1})}
				+E_0).
			\end{align*}
			Moreover, let  $\theta_1\in (0,1)$ be given by $\frac{d}{2}=\theta_1\sigma_1+(1-\theta_1)(\frac{d}{2}+2)$, and $\eta_1>0$ be a constant to be chosen later. By the interpolation inequality, we obtain 
			\begin{align*}
				\frac{1}{\var}\|\tau^{M-1}\delta u\|_{\widetilde{L}_t^1(\dot{B}_{2,1}^{\frac{d}{2}})}
				&\lesssim \frac{1}{\var}\int_0^t \tau^{M-1}\|\delta u\|_{\dot{B}_{2,\infty}^{\sigma_1}}^{\theta_1}\|\delta u\|_{\dot{B}_{2,1}^{\frac{d}{2}+2}}^{1-\theta_1}d\tau\\
				&\lesssim \frac{1}{\var}\|\delta u\|_{\widetilde{L}_t^\infty(\dot{B}_{2,\infty}^{\sigma_1})}^{\theta_1}
				\Big(\int_0^t \tau^{\frac{M-1-(1-\theta_1)M}{\theta_1}}d\tau\Big)^{\theta_1}\Big(\int_0^t \tau^M\|\delta u\|_{\dot{B}_{2,1}^{\frac{d}{2}+2}}d\tau\Big)^{1-\theta_1}\\
				&\lesssim \frac{1}{\var}\Big(t^{M-\frac{1}{2}(\frac{d}{2}-\sigma_1)}\|\delta u\|_{\widetilde{L}_t^\infty(\dot{B}_{2,\infty}^{\sigma_1})}\Big)^{\theta_1}\|\tau^M \delta u\|_{\widetilde{L}_t^1(\dot{B}^{\frac{d}{2}+2}_{2,1})}^{1-\theta_1}\\
				&\lesssim \eta_1^{-1}t^{M-\frac{1}{2}(\frac{d}{2}-\sigma_1)}\frac{1}{\var}\|\delta u\|_{\widetilde{L}_t^\infty(\dot{B}_{2,\infty}^{\sigma_1})}+\frac{\eta_1}{\var}\|\tau^M\delta u\|_{\widetilde{L}_t^1(\dot{B}_{2,1}^{\frac{d}{2}+2})}.
			\end{align*}
			Similarly, we have 
			\begin{align*}
				\frac{1}{\var}\|\tau^{M-1}\mathcal{R}^\var\|_{\widetilde{L}_t^1(\dot{B}_{2,1}^{\frac{d}{2}})}
				&\lesssim \frac{1}{\var}\int_0^t \tau^{M-1}\| \mathcal{R}^\var\|_{\dot{B}_{2,1}^{\frac{d}{2}}}^{\theta_1}\| \mathcal{R}^\var\|_{\dot{B}_{2,1}^{\frac{d}{2}}}^{1-\theta_1}d\tau\\
				&\lesssim \eta_1^{-1}t^{M-\frac{1}{2}(\frac{d}{2}-\sigma_1)}\frac{1}{\var}\| \mathcal{R}^\var\|_{\widetilde{L}_t^\infty(\dot{B}_{2,1}^{\frac{d}{2}})}+\frac{\eta_1}{\var}\|\tau^M\mathcal{R}^\var\|_{\widetilde{L}_t^1(\dot{B}_{2,1}^{\frac{d}{2}})}.
			\end{align*}

			The nonlinear terms in \eqref{timeener} can be addressed as follows. First, one deduces from \eqref{r0}, \eqref{r1} and \eqref{uv2} that
			\begin{align*}
				&\frac{1}{\var}\|\tau^M(u^\var\cdot\nabla \delta u,\delta u\cdot\nabla u^*)\|_{\widetilde{L}_t^1(\dot{B}_{2,1}^{\frac{d}{2}})}\\
				&\lesssim \frac{1}{\var}\|\tau^M(u^\var\otimes\delta u)\|_{\widetilde{L}_t^1(\dot{B}_{2,1}^{\frac{d}{2}+1})}+\frac{1}{\var}\|\tau^M(\delta u\otimes u^*)\|_{\widetilde{L}_t^1(\dot{B}_{2,1}^{\frac{d}{2}+1})}\\
				&\lesssim \|u^\var\|_{\widetilde{L}_t^2(\dot{B}_{2,1}^{\frac{d}{2}})}
				\frac{1}{\var}\|\tau^M\delta u\|_{\widetilde{L}_t^2(\dot{B}_{2,1}^{\frac{d}{2}+1})}
				+\|u^*\|_{\widetilde{L}_t^2(\dot{B}_{2,1}^{\frac{d}{2}})}
				\frac{1}{\var}\|\tau^M\delta u\|_{\widetilde{L}_t^2(\dot{B}_{2,1}^{\frac{d}{2}+1})}\\
				&\lesssim (E_0+\delta E_0) \delta\mathcal{D}(t).
			\end{align*}
      We also discover that       
			\begin{align*}
				\var \|\tau^M u^\var\cdot\nabla u^\var\|_{\widetilde{L}_t^1(\dot{B}_{2,1}^{\frac{d}{2}})} 
				&\lesssim \var \|\tau^M u^*\cdot\nabla u^*\|_{\widetilde{L}_t^1(\dot{B}_{2,1}^{\frac{d}{2}})}
				+\var \|\tau^M u^\var\cdot\nabla \delta u\|_{\widetilde{L}_t^1(\dot{B}_{2,1}^{\frac{d}{2}})}
				+\var\|\tau^M \delta u\cdot\nabla u^*\|_{\widetilde{L}_t^1(\dot{B}_{2,1}^{\frac{d}{2}})}\\ 
				&\lesssim \var \|\tau^M u^\var\otimes u^*\|_{\widetilde{L}_t^1(\dot{B}_{2,1}^{\frac{d}{2}+1})}
				+\var \|\tau^M u^\var\otimes\delta u\|_{\widetilde{L}_t^1(\dot{B}_{2,1}^{\frac{d}{2}+1})}
				+\var\|\tau^M \delta u\otimes u^*\|_{\widetilde{L}_t^1(\dot{B}_{2,1}^{\frac{d}{2}+1})}\\ 
				&\lesssim \|u^*\|_{\widetilde{L}_t^2(\dot{B}_{2,1}^{\frac{d}{2}+1})}\|\tau^M u^*\|_{\widetilde{L}_t^2(\dot{B}_{2,1}^{\frac{d}{2}})}
				+\var\|(u^\var,u^*)\|_{\widetilde{L}_t^2(\dot{B}_{2,1}^{\frac{d}{2}})}\|\tau^M\delta u\|_{\widetilde{L}_t^2(\dot{B}_{2,1}^{\frac{d}{2}+1})}\\
				&\lesssim (E_0+\delta E_0)t^{M-\frac{1}{2}(\frac{d}{2}-\sigma_1)}+(E_0+\delta E_0) \delta\mathcal{D}(t),
			\end{align*}
			and 
			\begin{align*}
				\frac{1}{\var} \|\tau^M w^\var\cdot\nabla w^\var\|_{\widetilde{L}_t^1(\dot{B}_{2,1}^{\frac{d}{2}})}\lesssim \frac{1}{\var} \|w^\var\|_{\widetilde{L}_t^2(\dot{B}^{\frac{d}{2}}_{2,1})}\|\tau^M w^\var\|_{\widetilde{L}_t^2(\dot{B}_{2,1}^{\frac{d}{2}+1})}
				\lesssim (E_0+\delta E_0)\delta\mathcal{D}(t).
			\end{align*}
			
			We now substitute the above estimates into \eqref{timeener} and choose a suitably small $\eta_1$ such that
			\begin{align}\label{time:1}
				&\frac{1}{\var}\|\tau^M\delta u\|_{\widetilde{L}_t^\infty(\dot{B}_{2,1}^{\frac{d}{2}})}+\frac{1}{\var}\|\tau^M\delta u\|_{\widetilde{L}_t^1(\dot{B}_{2,1}^{\frac{d}{2}+2})}
				+\|\tau^M\mathcal{R}^\var\|_{\widetilde{L}_t^\infty(\dot{B}_{2,1}^{\frac{d}{2}})}
				+\frac{1}{\var^2}\|\tau^M \mathcal{R}^\var\|_{\widetilde{L}_t^1(\dot{B}_{2,1}^{\frac{d}{2}})}\nonumber\\
				&\lesssim (\|(\delta u, a^*, u^*)\|_{\widetilde{L}_t^\infty(\dot{B}_{2,\infty}^{\sigma_1})}
				+E_0+\delta E_0)t^{M-\frac{1}{2}(\frac{d}{2}-\sigma_1)}
				+(E_0+\delta E_0)\delta\mathcal{D}(t)\nonumber\\
				&\lesssim (E_0+\delta E_0+1)t^{M-\frac{1}{2}(\frac{d}{2}-\sigma_1)}+(E_0+\delta E_0)\delta\mathcal{D}(t).
			\end{align}
			
				By \eqref{time:1} and $\mathcal{P}w^\var=\mathcal{R}^\var+\var u^\var$, we can recover the estimates of $\mathcal{P}w^\var$ as follows:
			\begin{align}\label{11}
				& \|\tau^M \mathcal{P}w^\var\|_{\widetilde{L}_t^\infty(\dot{B}_{2,1}^{\frac{d}{2}})}+\frac{1}{\var}\|\tau^M \mathcal{P}w^\var\|_{\widetilde{L}_t^2(\dot{B}_{2,1}^{\frac{d}{2}+1})}\nonumber\\
				&\lesssim (E_0+\delta E_0+1)t^{M-\frac{1}{2}(\frac{d}{2}-\sigma_1)}+(E_0+\delta E_0)\delta\mathcal{D}(t).
			\end{align}

            \noindent
  {\textbf{Step 2: Low-frequency estimates for $(\delta a,\mathcal{P}^\top w^\var,\mathcal{Z}^\var)$.}}	By \eqref{A6}, we have 
			\begin{equation}\label{4-1}
				\left\{
				\begin{aligned}
					&\partial_t(t^M \delta  a)-\Delta (t^M \delta a)=Mt^{M-1} \delta  a-t^M( u^\var\cdot\nabla \delta a+\delta u\cdot\nabla a^*)+\frac{1}{\var} t^M Y^\var,\\
					&\partial_t(t^M \mathcal{Z}^\var)+\frac{1}{\var^2}(t^ M\mathcal{Z}^\var)=Mt^{M-1} \mathcal{Z}^\var+\var t^M \partial_t((1+h(a^\var))\nabla a^\var)-\frac{1}{\var}t^M \mathcal{P}^\top (w^\var\cdot\nabla w^\var),\\
                    &(t^M \delta a, t^M \mathcal{Z}^\var)(0,x)=(0,0).
				\end{aligned}
				\right.
			\end{equation} 
			Then, similarly to the proof in Lemma \ref{P3.3}, we can perform energy estimates on \eqref{4-1} and obtain 
			\begin{align}\label{sfggg}
				& \|\tau^M \delta a\|_{\widetilde{L}_t^\infty(\dot{B}_{2,1}^{\frac{d}{2}})}^{\ell,\var}+\|\tau^M \delta a\|_{\widetilde{L}_t^1(\dot{B}_{2,1}^{\frac{d}{2}+2})}^{\ell,\var}+ \|\tau^M \partial_t\delta a\|_{\widetilde{L}_t^1(\dot{B}_{2,1}^{\frac{d}{2}})}^{\ell,\var}\nonumber\\
				&\quad\lesssim \|\tau^{M-1}\delta a\|_{\widetilde{L}_t^1(\dot{B}_{2,1}^{\frac{d}{2}})}^{\ell,\var}+\|\tau^M(u^\var\cdot\nabla \delta a, \delta u\cdot\nabla a^*)\|_{\widetilde{L}_t^1(\dot{B}_{2,1}^{\frac{d}{2}})}^{\ell,\var}+\frac{1}{\var}\|\tau^M Y^\var\|_{\widetilde{L}_t^1(\dot{B}_{2,1}^{\frac{d}{2}})}^{\ell,\var},
			\end{align}
			and 
			\begin{align}\label{sfggg1}
				\frac{1}{\var}\|\tau^M Y^\var\|_{\widetilde{L}_t^1(\dot{B}_{2,1}^{\frac{d}{2}})}^{\ell,\var}\lesssim \frac{1}{\var^2} \|\tau^M(\mathcal{Z}^\var,\mathcal{R}^\var)\|_{\widetilde{L}_t^1(\dot{B}_{2,1}^{\frac{d}{2}})}^{\ell,\var}+ \frac{1}{\var^2} \|\tau^Ma^\var(\mathcal{Z}^\var+\mathcal{R}^\var)\|_{\widetilde{L}_t^1(\dot{B}_{2,1}^{\frac{d}{2}})}^{\ell,\var}.
			\end{align}
			On the other hand, one also deduces from $\eqref{4-1}_2$ that
			\begin{align}\label{sfggg2}
				&\|\tau^M \mathcal{Z}^\var\|_{\widetilde{L}_t^\infty(\dot{B}_{2,1}^{\frac{d}{2}})}^{\ell,\var}+\frac{1}{\var^2}\|\tau^M \mathcal{Z}^\var\|_{\widetilde{L}_t^1(\dot{B}_{2,1}^{\frac{d}{2}})}^{\ell,\var}\nonumber\\
				&\quad\lesssim \|\tau^{M-1} \mathcal{Z}^\var\|_{\widetilde{L}_t^1(\dot{B}_{2,1}^{\frac{d}{2}})}^{\ell,\var}+\|\tau^M\partial_t  a^\var\|_{\widetilde{L}_t^1(\dot{B}_{2,1}^{\frac{d}{2}})}^{\ell,\var}\nonumber\\
				&\quad\quad+\var\|\tau^M\partial_t(h(a^\var)\nabla a^\var)\|_{\widetilde{L}_t^1(\dot{B}_{2,1}^{\frac{d}{2}})}^{\ell,\var}+\frac{1}{\var}\|\tau^M\mathcal{P}^\top (w^\var\cdot\nabla w^\var)\|_{\widetilde{L}_t^1(\dot{B}_{2,1}^{\frac{d}{2}})}^{\ell,\var}.
			\end{align}
			By \eqref{sfggg}-\eqref{sfggg2}, we arrive at
			\begin{align}\label{83101}
				&  \|\tau^M \delta a\|_{\widetilde{L}_t^\infty(\dot{B}_{2,1}^{\frac{d}{2}})\cap \widetilde{L}_t^1(\dot{B}_{2,1}^{\frac{d}{2}+2})}^{\ell,\var}+\|\tau^M \partial_t\delta a\|_{\widetilde{L}_t^1(\dot{B}_{2,1}^{\frac{d}{2}})}^{\ell,\var}+\|\tau^M \mathcal{Z}^\var\|_{\widetilde{L}_t^\infty(\dot{B}_{2,1}^{\frac{d}{2}})}^{\ell,\var}+\frac{1}{\var^2}\|\tau^M \mathcal{Z}^\var\|_{\widetilde{L}_t^1(\dot{B}_{2,1}^{\frac{d}{2}})}^{\ell,\var}\nonumber\\
				&\lesssim \|\tau^{M-1}\delta a\|_{\widetilde{L}_t^1(\dot{B}_{2,1}^{\frac{d}{2}})}^{\ell,\var}+\|\tau^{M-1}  \mathcal{Z}^\var\|_{\widetilde{L}_t^1(\dot{B}_{2,1}^{\frac{d}{2}})}^{\ell,\var}
				+\frac{1}{\var^2}\|\tau^M a^\var(\mathcal{Z}^\var+\mathcal{R}^\var)\|_{\widetilde{L}_t^1(\dot{B}_{2,1}^{\frac{d}{2}})}^{\ell,\var}\nonumber\\
				&\quad+\|\tau^M(u^\var\cdot\nabla \delta a, \delta u\cdot\nabla a^*)\|_{\widetilde{L}_t^1(\dot{B}_{2,1}^{\frac{d}{2}})}^{\ell,\var}+\var\|\tau^M\partial_t(h(a^\var)\nabla a^\var)\|_{\widetilde{L}_t^1(\dot{B}_{2,1}^{\frac{d}{2}})}^{\ell,\var}\nonumber\\
				&\quad+\frac{1}{\var}\|\tau^M\mathcal{P}^\top (w^\var\cdot\nabla w^\var)\|_{\widetilde{L}_t^1(\dot{B}_{2,1}^{\frac{d}{2}})}^{\ell,\var}+\frac{1}{\var^2}\|\tau^M\mathcal{R}^\var\|_{\widetilde{L}_t^1(\dot{B}_{2,1}^{\frac{d}{2}})}^{\ell,\var}.
			\end{align}
			The first two terms on the right-hand side of \eqref{83101} satisfy
			\begin{align*}
				& \|\tau^{M-1}\delta a\|_{\widetilde{L}_t^1(\dot{B}_{2,1}^{\frac{d}{2}})}^{\ell,\var}+\|\tau^{M-1} \mathcal{Z}^\var\|_{\widetilde{L}_t^1(\dot{B}_{2,1}^{\frac{d}{2}})}^{\ell,\var}\\  
				&\lesssim \tilde{\eta}(\|\tau^M \delta  a\|_{\widetilde{L}_t^1(\dot{B}_{2,1}^{\frac{d}{2}+2})}^{\ell,\var}+\frac{1}{\var^2}\|\tau^M  \mathcal{Z}^\var\|_{\widetilde{L}_t^1(\dot{B}_{2,1}^{\frac{d}{2}})}^{\ell,\var})+C\tilde{\eta}^{-1}(\|\delta a\|_{\widetilde{L}_t^\infty(\dot{B}_{2,\infty}^{\sigma_1})}+E_0+\delta E_0)t^{M-\frac{1}{2}(\frac{d}{2}-\sigma_1)}\\
				&\lesssim \tilde{\eta}(\|\tau^M \delta  a\|_{\widetilde{L}_t^1(\dot{B}_{2,1}^{\frac{d}{2}+2})}^{\ell,\var}+\frac{1}{\var^2}\|\tau^M  \mathcal{Z}^\var\|_{\widetilde{L}_t^1(\dot{B}_{2,1}^{\frac{d}{2}})}^{\ell,\var})+C\tilde{\eta}^{-1}(E_0+\delta E_0)t^{M-\frac{1}{2}(\frac{d}{2}-\sigma_1)}.
			\end{align*}
			The third term on the right-hand side of \eqref{83101} satisfies
			\begin{align*}
				\frac{1}{\var^2}\|\tau^M a^\var(\mathcal{Z}^\var+\mathcal{R}^\var)\|_{\widetilde{L}_t^1(\dot{B}_{2,1}^{\frac{d}{2}})}^{\ell,\var}
				\lesssim \|a^\var\|_{\widetilde{L}_t^\infty(\dot{B}_{2,1}^{\frac{d}{2}})}\frac{1}{\var^2}\|\tau^M(\mathcal{Z}^\var,\mathcal{R}^\var)\|_{\widetilde{L}_t^1(\dot{B}_{2,1}^{\frac{d}{2}})}
				\lesssim (E_0+\delta E_0)\delta\mathcal{D}(t).
			\end{align*}
			As in \eqref{dta}, the fourth term on the right-hand side of \eqref{83101} satisfies
			\begin{align*}
				& \var\|\tau^M \partial_t(h(a^\var)\nabla a^\var)\|_{\widetilde{L}_t^1(\dot{B}_{2,1}^{\frac{d}{2}})}^{\ell,\var}\\ 
				&\lesssim \|a^\var\|_{\widetilde{L}_t^\infty(\dot{B}_{2,1}^{\frac{d}{2}})} \Big( \|\tau^M\partial_t a^*\|_{\widetilde{L}_t^1(\dot{B}_{2,1}^{\frac{d}{2}})} +\|\tau^M\partial_t\delta a\|_{\widetilde{L}_t^1(\dot{B}_{2,1}^{\frac{d}{2}})}^{\ell,\var}+\|\tau^M\partial_t\delta a\|_{\widetilde{L}_t^1(\dot{B}_{2,1}^{\frac{d}{2}})}^{h,\var}\Big)\\
				&\lesssim (\|(a^*, u^*)\|_{\dot{B}_{2,\infty}^{\sigma_1}}+E_0) t^{M-\frac{1}{2}(\frac{d}{2}-\sigma_1)}
				+(E_0+\delta E_0)\delta\mathcal{D}(t)\\
				&\lesssim (E_0+1) t^{M-\frac{1}{2}(\frac{d}{2}-\sigma_1)}
				+(E_0+\delta E_0)\delta\mathcal{D}(t).
			\end{align*}
			Therefore, substituting the above estimates into \eqref{83101}, one has 
			\begin{align}\label{21}
				&\|\tau^M \delta a\|_{\widetilde{L}_t^\infty(\dot{B}_{2,1}^{\frac{d}{2}})}^{\ell,\var}+\|\tau^M \delta a\|_{\widetilde{L}_t^1(\dot{B}_{2,1}^{\frac{d}{2}+2})}^{\ell,\var}+\|\tau^M \mathcal{Z}^\var\|_{\widetilde{L}_t^\infty(\dot{B}^{\frac{d}{2}}_{2,1})}^{\ell,\var}+ \frac{1}{\var^2}\|\tau^M \mathcal{Z}^\var\|_{\widetilde{L}_t^1(\dot{B}_{2,1}^{\frac{d}{2}})}^{\ell,\var}\nonumber\\
				&\lesssim (E_0+\delta E_0+1) t^{M-\frac{1}{2}(\frac{d}{2}-\sigma_1)}+(E_0+\delta E_0)\delta\mathcal{D}(t),
			\end{align}
			which, together with $\mathcal{P}^\top w^\var=\mathcal{Z}^\var-\var(1+h(a^\var))\nabla a^\var$, leads to 
			\begin{align}\label{22}
				& \|\tau^M \mathcal{P}^\top w^\var\|_{\widetilde{L}_t^\infty(\dot{B}_{2,1}^{\frac{d}{2}})}^{\ell,\var}+\frac{1}{\var}\|\tau^M \mathcal{P}^\top w^\var\|_{\widetilde{L}_t^2(\dot{B}_{2,1}^{\frac{d}{2}})}^{\ell,\var}\nonumber\\
				&\lesssim \|\tau^Ma^*\|_{\widetilde{L}_t^\infty(\dot{B}_{2,1}^{\frac{d}{2}})\cap \widetilde{L}_t^2(\dot{B}_{2,1}^{\frac{d}{2}+1})}+\|\delta a\|_{\widetilde{L}_t^\infty(\dot{B}_{2,1}^{\frac{d}{2}})\cap \widetilde{L}_t^2(\dot{B}_{2,1}^{\frac{d}{2}+1})}
				+\|\tau^M \mathcal{Z}^\var\|_{\widetilde{L}_t^\infty(\dot{B}_{2,1}^{\frac{d}{2}})}^{\ell,\var}\nonumber\\
				&\quad+\frac{1}{\var}\|\tau^M \mathcal{Z}^\var\|_{\widetilde{L}_t^2(\dot{B}_{2,1}^{\frac{d}{2}})}^{\ell,\var}+\|a^\var\|_{\widetilde{L}_t^\infty(\dot{B}_{2,1}^{\frac{d}{2}})}\|\tau^M a^\var\|_{\widetilde{L}_t^\infty(\dot{B}_{2,1}^{\frac{d}{2}})\cap \widetilde{L}_t^2(\dot{B}_{2,1}^{\frac{d}{2}+1})}\nonumber\\
				&\lesssim (E_0+\delta E_0+1) t^{M-\frac{1}{2}(\frac{d}{2}-\sigma_1)}+(E_0+\delta E_0)\delta\mathcal{D}(t).
			\end{align}
			
\noindent
\textbf{Step 3: High-frequency estimates of $(a^\var, w^\var, \mathcal{Z}^\var)$}. Multiplying \eqref{main8} by $t^{2M}$ leads to 
			\begin{align*}
				& \frac{d}{dt}(t^{2M}L_j^{h,\var}(t) )+\frac{1}{\var^2}t^{2M} L_j^{h,\var}(t)\\
				&\lesssim 2M t^{2M-1} L_j^{h,\var}(t)+t^M \Big((\|\nabla w^\var\|_{L^\infty}+\|\nabla a^\var\|_{L^\infty}+\|w^\var\|_{L^\infty})\frac{1}{\var}\|(a_j^\var, w_j^\var)\|_{L^2}\\
				&\quad+\|\partial_t a^\var\|_{ L^\infty}\|a_j^\var\|_{L^2}+\frac{1}{\var}\|(\mathcal{N}_{1j}^\var,\mathcal{N}_{2j}^\var,\mathcal{N}_{3j}^\var,\mathcal{N}_{4j}^\var)\|_{L^2}+\frac{1}{\var}\|u_j^\var\|_{L^2} \Big)\sqrt{L_j^{h,\var}(t)},
			\end{align*}
			which implies that 
			\begin{align*}
				&\var \|\tau^M (a_j^\var,w_j^\var)\|_{L_t^\infty (L^2)}+\frac{1}{\var}\|\tau^M(a_j^\var, w_j^\var)\|_{L_t^1(L^2)}\\
				&\lesssim \var\|\tau^{M-1}(a_j^\var,w_j^\var)\|_{L_t^1(L^2)}+(\|\nabla w^\var\|_{L^2_t(L^\infty)}+\|\nabla a^\var\|_{L^2_t(L^\infty)}+\|w^\var\|_{L^2_t(L^\infty)})\|\tau^M(a_j^\var,w_j^\var)\|_{L^2_t(L^2)}\\
				&\quad+\var\|\partial_ta^\var\|_{L_t^1(L^\infty)} \|\tau^M a_j^\var\|_{L_t^\infty(L^2)}
				+\sum_{j\geq J_{\var}-1}2^{(\frac{d}{2}+1)j}\|\tau^M(\mathcal{N}_{1j}^\var,\mathcal{N}_{2j}^\var,\mathcal{N}_{3j}^\var,\mathcal{N}_{4j}^\var)\|_{L_t^1(L^2)}.
			\end{align*}
			By using Gr\"{o}nwall's inequality, we have 
			\begin{align}\label{31}
				&\var \|\tau^M(a^\var, w^\var)\|_{\widetilde{L}_t^\infty(\dot{B}_{2,1}^{\frac{d}{2}+1})}^{h,\var}+\frac{1}{\var}\|\tau^M(a^\var, w^\var)\|_{\widetilde{L}_t^1(\dot{B}_{2,1}^{\frac{d}{2}+1})}^{h,\var}\nonumber\\
				&\lesssim \|\tau^{M-1}(a^\var, w^\var)\|_{\widetilde{L}_t^1(\dot{B}_{2,1}^{\frac{d}{2}+1})}^{h,\var}
				+\|(\nabla w^\var, w^\var,\nabla a^\var)\|_{\widetilde{L}_t^2(\dot{B}_{2,1}^{\frac{d}{2}})}\|\tau^M (a^\var,w^\var)\|_{\widetilde{L}_t^2(\dot{B}_{2,1}^{\frac{d}{2}+1})}\nonumber\\
				&\quad+\|\partial_ta^\var\|_{\widetilde{L}_t^1(\dot{B}_{2,1}^{\frac{d}{2}})}\var \|\tau^M (a^\var, w^\var)\|_{\widetilde{L}_t^\infty(\dot{B}_{2,1}^{\frac{d}{2}+1})}^{h,\var}+\sum_{j\geq J_{\var}-1}2^{(\frac{d}{2}+1)j} \|\tau^M(\mathcal{N}_{1j}^\var,\mathcal{N}_{2j}^\var,\mathcal{N}_{3j}^\var,\mathcal{N}_{4j}^\var)\|_{L_t^1(L^2)}.
			\end{align}
			Then, similarly to the previous computations, one has 
			\begin{align*}
				&\var\|\tau^M (a^\var, w^\var)\|_{\widetilde{L}_t^\infty(\dot{B}_{2,1}^{\frac{d}{2}+1})}^{h,\var}+\frac{1}{\var} \|\tau^M(a^\var, w^\var)\|_{\widetilde{L}_t^1(\dot{B}_{2,1}^{\frac{d}{2}+1})}^{h,\var}\\
				&\lesssim (E_0+\delta E_0+1) t^{M-\frac{1}{2}(\frac{d}{2}-\sigma_1)}+(E_0+\delta E_0)\delta\mathcal{D}(t),
			\end{align*}
			which in particular implies that 
			\begin{align}\label{3}
				\|\tau^M\mathcal{Z}^\var\|_{\widetilde{L}_t^1(\dot{B}_{2,1}^{\frac{d}{2}})}^{h,\var}+\|\tau^M\mathcal{R}^\var\|_{\widetilde{L}_t^1(\dot{B}_{2,1}^{\frac{d}{2}})}^{h,\var}
				\lesssim (E_0+\delta E_0+1) t^{M-\frac{1}{2}(\frac{d}{2}-\sigma_1)}+(E_0+\delta E_0)\delta\mathcal{D}(t).
			\end{align}

			Finally, collecting the estimates \eqref{time:1}, \eqref{11}, \eqref{21}, \eqref{22}, \eqref{31} and \eqref{3} and recalling the definition of $\delta\mathcal{D}(t)$, it is proved that
			\begin{align*}
				\delta\mathcal{D}(t) 	\lesssim (E_0+\delta E_0+1) t^{M-\frac{1}{2}(\frac{d}{2}-\sigma_1)}+(E_0+\delta E_0)\delta\mathcal{D}(t),
			\end{align*}
			and then, due to the smallness of $E_0$ and $\delta E_0$, we have 
			\begin{align}
				\delta\mathcal{D}(t) 	\lesssim 	(E_0+\delta E_0+1) t^{M-\frac{1}{2}(\frac{d}{2}-\sigma_1)}.
			\end{align}
			This finishes the proof of Proposition \ref{prop41}.
		\end{proof}

		\subsection{Proof of Theorem \ref{thm3}: Decay estimates \texorpdfstring{\eqref{decay0}-\eqref{decay2}}{decay0-decay2}}

    Relying on the time-weighted estimate \eqref{timeweighted}, we establish the decay rates of the error and the solution in Theorem \ref{thm3} as follows.
According to Proposition \ref{prop41}, we have
			\begin{align}\label{ssgbbb}
				&\|\tau^M \delta a\|_{\widetilde{L}_t^\infty(\dot{B}^{\frac{d}{2}}_{2,1})\cap \widetilde{L}_t^1(\dot{B}^{\frac{d}{2}+2}_{2,1})}^{\ell,\var}+\frac{1}{\var}\|\tau^M \delta u\|_{\widetilde{L}_t^\infty(\dot{B}^{\frac{d}{2}}_{2,1})\cap \widetilde{L}_t^1(\dot{B}^{\frac{d}{2}+2}_{2,1})}
				+\|\tau^M w^\var\|_{\widetilde{L}_t^\infty( \dot{B}_{2,1}^{\frac{d}{2}})}\nonumber\\
				&\quad+\frac{1}{\var}
				\|\tau^M w^\var\|_{\widetilde{L}_t^2(\dot{B}_{2,1}^{\frac{d}{2}})}+\var \|\tau^M(a^\var,w^\var)\|_{\widetilde{L}_t^\infty(\dot{B}_{2,1}^{\frac{d}{2}+1})}^{h,\var}
				+\frac{1}{\var}\|\tau^M(a^\var,w^\var)\|_{\widetilde{L}_t^1(\dot{B}_{2,1}^{\frac{d}{2}+1})}^{h,\var}\nonumber\\
    &\quad+\|\tau^M(\mathcal{Z}^\var,\mathcal{R}^\var)\|_{\widetilde{L}_t^{\infty}(\dot{B}_{2,1}^{\frac{d}{2}})}+\frac{1}{\var^2}\|\tau^M(\mathcal{Z}^\var,\mathcal{R}^\var)\|_{\widetilde{L}_t^1(\dot{B}_{2,1}^{\frac{d}{2}})}\nonumber\\
				&\lesssim t^{M-\frac{1}{2}(\frac{d}{2}-\sigma_1)},
			\end{align}	
			which yields for $t\geq 1$ that
			\begin{align}\label{sgggg}
				&\frac{1}{\var}\|\delta u(t)\|_{\dot{B}^{\frac{d}{2}}_{2,1}}+\|\delta a(t)\|_{\dot{B}^{\frac{d}{2}}_{2,1}}^{\ell,\var}+\|w^\var(t)\|_{ \dot{B}_{2,1}^{\frac{d}{2}}}
				+\var \|(a^\var,w^\var)(t)\|_{\dot{B}_{2,1}^{\frac{d}{2}+1}}^{h,\var}
				+\|(\mathcal{Z}^\var,\mathcal{R}^\var)(t)\|_{\dot{B}^{\frac{d}{2}}_{2,1}}\nonumber\\
				&\lesssim t^{-\frac{1}{2}(\frac{d}{2}-\sigma_1)}\lesssim (1+t)^{-\frac{1}{2}(\frac{d}{2}-\sigma_1)}.
			\end{align}
			Using \eqref{sgggg}, the interpolation inequality \eqref{inter} and high-frequency norm of $\delta a$, for $\sigma_1<\sigma\leq \frac{d}{2}$ and $t\geq1$, we have
			\begin{align}
				\|\delta u(t)\|_{\dot{B}^{\sigma}_{2,1}}&\lesssim \|\delta u\|_{\dot{B}^{\sigma_1}_{2,\infty}}^{1-\theta_2} \|\delta u\|_{\dot{B}^{\frac{d}{2}}_{2,1}}^{\theta_2}\lesssim \var (1+t)^{-\frac{1}{2}(\sigma-\sigma_1)},\label{430}\\
				\|(\delta a,w^\var,\mathcal{Z}^\var,\mathcal{R}^\var)\|_{\dot{B}^{\sigma}_{2,1}}&\lesssim \|(\delta a,w^\var,\mathcal{Z}^\var,\mathcal{R}^\var)\|_{\dot{B}^{\sigma_1}_{2,\infty}}^{1-\theta_2} \|(\delta a,w^\var,\mathcal{Z}^\var,\mathcal{R}^\var)\|_{\dot{B}^{\frac{d}{2}}_{2,1}}^{\theta_2} \lesssim (1+t)^{-\frac{1}{2}(\sigma-\sigma_1)}, \label{431}
			\end{align}
			where $\theta_2\in (0,1)$ is given by $\sigma=\sigma_1(1-\theta_2)+\frac{d}{2} \theta_2$. 

Combining \eqref{decayNSKS} and \eqref{430}-\eqref{431}, we can derive the decay of $(a^\var, u^\var)$ as follows: 
\begin{align}
\|(a^\var,u^\var)(t)\|_{\dot{B}^{\sigma}_{2,1}}\lesssim \|(\delta a,\delta u)(t)\|_{\dot{B}^{\sigma}_{2,1}}+\|(a^*,u^*)(t)\|_{\dot{B}^{\sigma}_{2,1}}\lesssim (1+t)^{-\frac{1}{2}(\sigma-\sigma_1)},\label{4341}
\end{align}
for $\sigma_1<\sigma\leq \frac{d}{2}$ and $t\geq1$.

Note that the decay of $\|\delta a\|_{\dot{B}^{\sigma}_{2,1}}$ $(\sigma_1<\sigma\le \frac d2-1)$ can be further improved with an additional convergence factor~$\varepsilon$. Indeed, considering the weighted equations \eqref{416}-\eqref{4-1} for any sufficiently large $M$, and arguing as in Lemma~\ref{P3.2}, one obtains 
\begin{align*}
\frac1\varepsilon\|\tau^M\delta a\|_{\widetilde{L}_t^\infty(\dot{B}_{2,1}^{\frac d2-1})}^{\ell,\varepsilon}
	\lesssim t^{\,M-\frac12(\frac d2-1-\sigma_1)}.
\end{align*}
The proof relies on controlling a weighted functional similar to~\eqref{Fn1}. Since the argument parallels that of Proposition~\ref{P31}, together with the interpolation argument in Proposition~\ref{prop41} (replacing the $\dot{B}^{\frac d2}_{2,1}$–energy by $\dot{B}^{\frac d2-1}_{2,1}$ for $(\delta u,\delta a)$), we omit the details. Consequently,
\begin{align}
\|\delta a(t)\|_{\dot{B}_{2,1}^{\frac d2-1}}^{\ell,\varepsilon}
	\lesssim \varepsilon (1+t)^{-\frac12(\frac d2-1-\sigma_1)},\qquad t\ge1.
\end{align}
For the high–frequency part, using \eqref{sgggg} and the decay of $a^*$, we deduce for $t\ge1$ that
\begin{align}\label{11113r}
\|\delta a(t)\|_{\dot{B}_{2,1}^{\frac d2-1}}^{h,\varepsilon}
	\lesssim \varepsilon\|a^{\varepsilon}(t)\|_{\dot{B}_{2,1}^{\frac d2}}^{h,\varepsilon}
	+ \varepsilon\|a^*(t)\|_{\dot{B}_{2,1}^{\frac d2}}^{h,\varepsilon}
	\lesssim \varepsilon (1+t)^{-\frac12(\frac d2-\sigma_1)}.
\end{align}
Combining \eqref{inter}, \eqref{evolution} and \eqref{11113r}  yields
\begin{align}\label{additionaldeltaa}
\|\delta a(t)\|_{\dot{B}^{\sigma}_{2,1}}
	\lesssim \varepsilon (1+t)^{-\frac12(\sigma-\sigma_1)},
\end{align}
for $\sigma\in (\sigma_1,\frac{d}{2}-1]$ and $t\geq1$.
			
			The next step is to prove faster decay rates of $\mathcal{Z}^\var$. Recalling that $\mathcal{Z}^\var$ and $\mathcal{R}^\var$ satisfy \eqref{T8}, we have
			\begin{align}\label{100701}
				\mathcal{Z}^\varepsilon=e^{-\frac{t}{\varepsilon^2}}\mathcal{Z}_0^\var+\int_0^t e^{-\frac{t-\tau}{\varepsilon^2}}\Big(\var\partial_t \nabla \int_0^{a^\var}  (1+h(s))\,ds-\frac{1}{\varepsilon}\mathcal{P}^{\top}(w^\varepsilon\cdot\nabla w^\varepsilon)\Big)\,d\tau,
			\end{align}
			and
			\begin{align}
				\mathcal{R}^\varepsilon&=e^{-(1+\frac{1}{\var^2})t}\mathcal{R}_0^\var+\int_0^t e^{-(1+\frac{1}{\var^2})(t-\tau)}\Big(-\varepsilon\mu\Delta u+\varepsilon\mathcal{P}(u^\varepsilon\cdot\nabla u^\varepsilon)\nonumber\\
				&\quad-\mathcal{P}(a^\var(\mathcal{Z}^\var+\mathcal{R}^\var))-\frac{1}{\varepsilon}\mathcal{P}(w^\varepsilon\cdot\nabla w^\varepsilon)\Big)d\tau.\label{100702}
			\end{align}
			Taking the low-frequency $\dot{B}^{\sigma}_{2,1}$-norm of \eqref{100701}, we have
			\begin{align}
				\|\mathcal{Z}^\varepsilon(t)\|_{\dot{B}^{\sigma}_{2,1}}^{\ell,\var}&\lesssim e^{-\frac{t}{\varepsilon^2}} \|\mathcal{Z}_0^\var\|_{\dot{B}^{\sigma}_{2,1}}^{\ell,\var}+\int_0^t  e^{-\frac{t-\tau}{\varepsilon^2}} \Big(  \|\partial_t\int_0^{a^\var}  (1+h(s))\,ds\|_{\dot{B}^{\sigma}_{2,1}}^{\ell,\var}+\frac{1}{\var}\|w^\var\cdot\nabla w^\var\|_{\dot{B}^{\sigma}_{2,1}}^{\ell,\var}\Big)\,d\tau\nonumber\\
				&\lesssim e^{-\frac{t}{\varepsilon^2}} \|\mathcal{Z}_0^\var\|_{\dot{B}^{\sigma}_{2,1}}^{\ell,\var}+\int_0^t  e^{-\frac{t-\tau}{\varepsilon^2}} \Big( \|\partial_t a^\var\|_{\dot{B}^{\sigma}_{2,1}}+\frac{1}{\var}\|w^\var\|_{\dot{B}^{\frac{d}{2}}_{2,1}}\|\nabla w^\var\|_{\dot{B}^{\sigma}_{2,1}}\Big)\,d\tau.
			\end{align}
			For $\sigma_1<\sigma\leq \frac{d}{2}-1$, by $\eqref{Main:incom}_1$ it holds 
			\begin{align}
				\|\partial_t a^\var\|_{\dot{B}^{\sigma}_{2,1}}\lesssim  \frac{1}{\var}(1+\|a^\var\|_{\dot{B}^{\frac{d}{2}}_{2,1}}) \|w^\var\|_{\dot{B}^{\sigma+1}_{2,1}}.
			\end{align}
			Since $\|w^\var\|_{\dot{B}^{\sigma+1}_{2,1}}\lesssim (1+t)^{-\frac{1}{2}(\sigma-\sigma_1+1)}$, for $\sigma_1<\sigma\leq\frac{d}{2}-1$ we arrive at
			\begin{align}\label{1005}
				\|\mathcal{Z}^\varepsilon(t)\|_{\dot{B}^{\sigma}_{2,1}}^{\ell,\var}&\lesssim e^{-\frac{t}{\varepsilon^2}} \|\mathcal{Z}_0^\var\|_{\dot{B}^{\sigma}_{2,1}}^{\ell,\var}+\frac{1}{\var}\int_0^t  e^{-\frac{t-\tau}{\varepsilon^2}} (1+\tau)^{-\frac{1}{2}(\sigma-\sigma_1+1)}\,d\tau\lesssim \var (1+t)^{-\frac{1}{2}(\sigma-\sigma_1+1)}.
			\end{align}
	For high-frequency, it holds by \eqref{431} that
			\begin{align}
				\|\mathcal{Z}^\varepsilon(t)\|_{\dot{B}^{\sigma}_{2,1}}^{h,\var}
				\lesssim \var\|\mathcal{Z}^\varepsilon\|_{\dot{B}^{\sigma+1}_{2,1}}^{h,\var}
				\lesssim \var\|\mathcal{Z}^\varepsilon\|_{\dot{B}^{\sigma+1}_{2,1}}\lesssim
				\var (1+t)^{-\frac{1}{2}(\sigma-\sigma_1+1)},
			\end{align}
			which, together with \eqref{1005}, implies that
			\begin{align}
				&\|\mathcal{Z}^\varepsilon(t)\|_{\dot{B}^{\sigma}_{2,1}}\lesssim \var (1+t)^{-\frac{1}{2}(\sigma-\sigma_1+1)},\quad \sigma_1<\sigma\leq \frac{d}{2}-1.\label{Zfaster}
			\end{align}

			To capture faster rates of $\mathcal{R}^\var$, we similarly deduce from \eqref{100702} that
			\begin{align}
				\|\mathcal{R}^\varepsilon(t)\|_{\dot{B}^{\sigma}_{2,1}}^{\ell,\var}&\lesssim e^{-(1+\frac{1}{\var^2})t}\|\mathcal{R}_0^\var\|_{\dot{B}^{\sigma}_{2,1}}+\int_0^t e^{-(1+\frac{1}{\var^2})(t-\tau)} \Big(\var \|u^\var\|_{\dot{B}^{\sigma+2}_{2,1}}^{\ell,\var}+\var\|u^\var\|_{\dot{B}^{\frac{d}{2}}_{2,1}}\|u^\var\|_{\dot{B}^{\sigma+1}_{2,1}}\nonumber\\
				&\quad\quad+\|a^\var\|_{\dot{B}^{\frac{d}{2}}_{2,1}}(\|\mathcal{Z}^\var\|_{\dot{B}^{\sigma}_{2,1}}+\|\mathcal{R}^\var\|_{\dot{B}^{\sigma}_{2,1}})+\frac{1}{\var}\|w^\var\|_{\dot{B}^{\frac{d}{2}}_{2,1}}\|w^\var\|_{\dot{B}^{\sigma+1}_{2,1}}\Big)\,d\tau.
			\end{align}
            for $\sigma_1<\sigma\leq \frac{d}{2}-1$. Now we define
			\begin{align}
				\mathcal{X}_{\sigma}=\sup_{t>0}\{(1+t)^{\frac{1}{2}(\sigma-\sigma_1+1)}\|\mathcal{R}^\var\|_{\dot{B}^{\sigma}_{2,1}}^{\ell,\var}\}.
			\end{align}	
We obtain from \eqref{430}-\eqref{431} that 
			\begin{align*}
				&\|a^\var\|_{\dot{B}^{\frac{d}{2}}_{2,1}}(\|\mathcal{Z}^\var\|_{\dot{B}^{\sigma}_{2,1}}+\|\mathcal{R}^\var\|_{\dot{B}^{\sigma}_{2,1}})\\
                &\lesssim \|a^\var\|_{\dot{B}^{\frac{d}{2}}_{2,1}}(\|\mathcal{Z}^\var\|_{\dot{B}^{\sigma}_{2,1}}^{\ell,\var}+\|\mathcal{R}^\var\|_{\dot{B}^{\sigma}_{2,1}}^{\ell,\var})+\|a^\var\|_{\dot{B}^{\frac{d}{2}}_{2,1}}(\var^{\frac{d}{2}-\sigma}\|\mathcal{Z}^\var\|_{\dot{B}^{\frac{d}{2}}_{2,1}}^{h,\var}+\var^{\frac{d}{2}-\sigma}\|\mathcal{R}^\var\|_{\dot{B}^{\frac{d}{2}}_{2,1}}^{h,\var})\\
				&\lesssim (E_0+\delta E_0)(\|\mathcal{Z}^\var\|_{\dot{B}^{\sigma}_{2,1}}^{\ell,\var}+ \|\mathcal{R}^\var\|_{\dot{B}^{\sigma}_{2,1}}^{\ell,\var})+\var(1+t)^{-\frac{1}{2}(\sigma-\sigma_1+1)}\\
				&\lesssim \var(1+t)^{-\frac{1}{2}(\sigma-\sigma_1+1)}+(E_0+\delta E_0)(1+t)^{-\frac{1}{2}(\sigma-\sigma_1+1)}\mathcal{X}_{\sigma}.
			\end{align*}         
According to \eqref{430}-\eqref{431}, we can clearly see that
			\begin{align*}
				\var\|u^\var\|_{\dot{B}^{\sigma+2}_{2,1}}^{\ell,\var}\lesssim \|\delta u\|_{\dot{B}^{\sigma+1}_{2,1}}^{\ell,\var}+\var\|a^*\|_{\dot{B}^{\sigma+2}_{2,1}}\lesssim
				\|\delta u\|_{\dot{B}^{\sigma+1}_{2,1}}+\var\|a^*\|_{\dot{B}^{\sigma+2}_{2,1}}\lesssim 
				\var(1+t)^{-\frac{1}{2}(\sigma-\sigma_1+1)},
			\end{align*}
			and 
			\begin{align*}
				\var\|u^\var\|_{\dot{B}^{\frac{d}{2}}_{2,1}}\|u^\var\|_{\dot{B}^{\sigma+1}_{2,1}}+\frac{1}{\var}\|w^\var\|_{\dot{B}^{\frac{d}{2}}_{2,1}}\|w^\var\|_{\dot{B}^{\sigma+1}_{2,1}}\lesssim 
				\var(1+t)^{-\frac{1}{2}(\sigma-\sigma_1+1)}.
			\end{align*}
			
			It can be concluded from the above estimates that 
			\begin{align}
				\|\mathcal{R}^\varepsilon(t)\|_{\dot{B}^{\sigma}_{2,1}}^{\ell,\var}&\lesssim e^{-(1+\frac{1}{\var^2})t}\|\mathcal{R}_0^\var\|_{\dot{B}^{\sigma}_{2,1}}
				+\int_0^t e^{-(1+\frac{1}{\var^2})(t-\tau)}\Big(\var(1+\tau)^{-\frac{1}{2}(\sigma-\sigma_1+1)}\nonumber\\
                &\quad+(E_0+\delta E_0)(1+\tau)^{-\frac{1}{2}(\sigma-\sigma_1+1)}\mathcal{X}_{\sigma}\Big)\,d\tau\nonumber\\
				&\lesssim \var(1+t)^{-\frac{1}{2}(\sigma-\sigma_1+1)}+(E_0+\delta E_0)(1+t)^{-\frac{1}{2}(\sigma-\sigma_1+1)}\mathcal{X}_{\sigma}.
			\end{align}
			
			According to the definition of $\mathcal{X}_{\sigma}$, it holds 
			\begin{align*}
				\mathcal{X}_{\sigma}\lesssim \var+(E_0+\delta E_0)\mathcal{X}_{\sigma}.
			\end{align*}
			Due to the small size of $E_0$ and $\delta E_0$, one has 
			\begin{align*}
				\mathcal{X}_{\sigma}\lesssim \var.
			\end{align*}
			Thus, we obtain that 
			\begin{align*}
				\|\mathcal{R}^\varepsilon(t)\|_{\dot{B}^{\sigma}_{2,1}}^{\ell,\var}\lesssim \var(1+t)^{-\frac{1}{2}(\sigma-\sigma_1+1)}.
			\end{align*}
			The high frequency part is given by 
			\begin{align*}
				\|\mathcal{R}^\varepsilon(t)\|_{\dot{B}^{\sigma}_{2,1}}^{h,\var}\lesssim	\var\|\mathcal{R}^\varepsilon(t)\|_{\dot{B}^{\sigma+1}_{2,1}}^{h,\var}\lesssim \var(1+t)^{-\frac{1}{2}(\sigma-\sigma_1+1)}.
			\end{align*}
			Combining these two estimates yields that, for $\sigma_1<\sigma\leq\frac{d}{2}-1$ and $t\geq1$, 
			\begin{align}\label{Rfaster}
				\|\mathcal{R}^\varepsilon(t)\|_{\dot{B}^{\sigma}_{2,1}}\lesssim \var(1+t)^{-\frac{1}{2}(\sigma-\sigma_1+1)}.
			\end{align}

			Noting that $\mathcal{R}^\var=\mathcal{P}w^\varepsilon-\varepsilon u^\varepsilon$, it holds for $\sigma_1<\sigma\leq\frac{d}{2}-1$ and $t\geq1$ that
			\begin{align}
				\|(\mathcal{P}w^\varepsilon-\varepsilon u^\varepsilon)(t)\|_{\dot{B}^{\sigma}_{2,1}}\lesssim \var(1+t)^{-\frac{1}{2}(\sigma-\sigma_1+1)},
			\end{align}
			and 
			\begin{align}
				\|	\mathcal{P}^\top w^\varepsilon(t)\|_{\dot{B}^{\sigma}_{2,1}}&\lesssim \|\mathcal{Z}^\varepsilon(t)\|_{\dot{B}^{\sigma}_{2,1}}+\varepsilon\|\nabla a^\varepsilon(t)\|_{\dot{B}^{\sigma}_{2,1}}+\varepsilon\|\nabla a^\varepsilon(t)\|_{\dot{B}^{\sigma}_{2,1}}\lesssim \var(1+t)^{-\frac{1}{2}(\sigma-\sigma_1+1)}.
			\end{align}	
			Consequently, by the triangle inequality, for $\sigma_1<\sigma\leq\frac{d}{2}-1$ and  $t\geq1$, one has 
			\begin{align}\label{wdfgg}
				\|(w^\varepsilon-\var u^\var)(t)\|_{\dot{B}^{\sigma}_{2,1}}&\lesssim  	\|\mathcal{P}^\top w^\varepsilon(t)\|_{\dot{B}^{\sigma}_{2,1}}+\|(\mathcal{P} w^\varepsilon-\var u^\var)(t)\|_{\dot{B}^{\sigma}_{2,1}}\lesssim \var(1+t)^{-\frac{1}{2}(\sigma-\sigma_1+1)}.
			\end{align}	
			Therefore, by \eqref{430}-\eqref{4341}, \eqref{additionaldeltaa}, \eqref{Zfaster}, \eqref{Rfaster} and \eqref{wdfgg},   we complete the proof of \eqref{decay1}-\eqref{decay2} in Theorem \ref{thm3}.
	\
		\subsection{Proof of Theorem \ref{thm3}: Improved decay estimate \eqref{decay3} for \texorpdfstring{$\delta a$}{delta a}}
		We are going to establish faster decay rates of $\delta a$ when $-\frac{d}{2}\leq \sigma_1<\frac{d}{2}-2$ ($d\geq3$). In this case, we know $\|\delta a\|_{\dot{B}^{\sigma_1-1}_{2,\infty}}$ has a $\mathcal{O}(\varepsilon)$-bounds (see \eqref{improve}).
			Similarly to prove Proposition \ref{prop41}, we replace the $\dot{B}^{\frac{d}{2}}_{2,1}$ estimates for $\delta a$ with the $\dot{B}^{\frac{d}{2}-1}_{2,1}$ estimates and define
			\begin{align}
				\delta\mathcal{D}^*(t)&=
				\frac{1}{\var}\|\tau^M \delta a\|_{\widetilde{L}_t^\infty(\dot{B}^{\frac{d}{2}-1}_{2,1})\cap \widetilde{L}_t^1(\dot{B}^{\frac{d}{2}+1}_{2,1})}^{\ell,\var}+\frac{1}{\var}\|\tau^M \delta u\|_{\widetilde{L}_t^\infty(\dot{B}^{\frac{d}{2}}_{2,1})\cap \widetilde{L}_t^1(\dot{B}^{\frac{d}{2}+2}_{2,1})}\nonumber\\
				&\quad+\|\tau^M w^\var\|_{\widetilde{L}_t^\infty( \dot{B}_{2,1}^{\frac{d}{2}})}+\frac{1}{\var}
				\|\tau^M w^\var\|_{\widetilde{L}_t^2(\dot{B}_{2,1}^{\frac{d}{2}})}\nonumber\\
				&\quad+\var \|\tau^M(a^\var,w^\var)\|_{\widetilde{L}_t^\infty(\dot{B}_{2,1}^{\frac{d}{2}+1})}^{h,\var}
				+\frac{1}{\var}\|\tau^M(a^\var,w^\var)\|_{\widetilde{L}_t^1(\dot{B}_{2,1}^{\frac{d}{2}+1})}^{h,\var}\nonumber\\
				&\quad+\|\tau^M(\mathcal{Z}^\var,\mathcal{R}^\var)\|_{\widetilde{L}_t^{\infty}(\dot{B}_{2,1}^{\frac{d}{2}})}+\frac{1}{\var^2}\|\tau^M(\mathcal{Z}^\var,\mathcal{R}^\var)\|_{\widetilde{L}_t^1(\dot{B}_{2,1}^{\frac{d}{2}})}.\label{DT*}
			\end{align}
			
			We follow a similar method developed in Proposition \ref{prop41} to demonstrate that 
			\begin{align}
				\delta\mathcal{D}^*(t)\lesssim t^{M-\frac{1}{2}(\frac{d}{2}-\sigma_1)}+(E_0+\delta E_0)\delta\mathcal{D}^*(t). 
			\end{align}
			Due to the smallness of $E_0$ and $\delta E_0$, one has 
			\begin{align}
				\delta\mathcal{D}^*(t)\lesssim t^{M-\frac{1}{2}(\frac{d}{2}-\sigma_1)},
			\end{align}
			which implies that 
			\begin{align}
				\|\delta a(t)\|_{\dot{B}^{\frac{d}{2}-1}_{2,1}}^{\ell,\var}\lesssim\var t^{-\frac{1}{2}(\frac{d}{2}-\sigma_1)}. 
			\end{align}
			Then we have 
			\begin{align}
				\|\delta a(t)\|_{\dot{B}^{\frac{d}{2}-1}_{2,1}}\lesssim 	\|\delta a(t)\|_{\dot{B}^{\frac{d}{2}-1}_{2,1}}^{\ell,\var}+\var\| a^\var(t)\|_{\dot{B}^{\frac{d}{2}}_{2,1}}^{h,\var}+\var\| a^*(t)\|_{\dot{B}^{\frac{d}{2}}_{2,1}}^{h,\var}\lesssim\var t^{-\frac{1}{2}(\frac{d}{2}-\sigma_1)}. 
			\end{align}
			By the interpolation inequality, for $\sigma_1-1<\sigma<\frac{d}{2}-1$ and $t\geq1$, one has
			\begin{align}
				\|\delta a(t)\|_{\dot{B}^{\sigma}_{2,1}}\lesssim \|\delta a(t)\|_{\dot{B}^{\sigma_1-1}_{2,1}}^{1-\theta_3}\|\delta a(t)\|_{\dot{B}^{\frac{d}{2}-1}_{2,1}}^{\theta_3} \lesssim \var (1+t)^{-\frac{1}{2}(\sigma-\sigma_1+1)},
			\end{align}
			where $\theta_3$ satisfies that $\sigma=(1-\theta_3)(\sigma_1-1)+\theta_3(\frac{d}{2}-1)$.

            Consequently, we derive \eqref{decay3} and complete the proof of Theorem \ref{thm3}.
		
		\appendix

\section{Proof of Proposition \ref{thm1}}\label{APA}



\subsection{Global existence}
We consider the sequence of approximate solutions ($n\geq 1$) defined by
\begin{equation}\label{AP1}
	\left\{
	\begin{aligned}
		&\partial_t a_{n+1}^* - \Delta a_{n+1}^* = -u_{n}^* \cdot \nabla a_{n}^*,\\
		&\partial_t u_{n+1}^* - \mu\Delta u_{n+1}^* = -\mathcal{P}(u_{n}^* \cdot \nabla u_{n}^*),\\
		&{\rm div}\, u_{n+1}^* = 0,\\
		&(a_{n+1}^*,u_{n+1}^*)(0,x) = (a_0^*,u_0^*).
	\end{aligned}
	\right.
\end{equation}
Set $(a_1^*,u_1^*) = (e^{t\Delta}a_0^*,e^{t\mu\Delta}u_0^*)$  serves as the first (linear) approximate solution in our iteration.

\noindent\textbf{Step 1: Construction of approximate solutions.}	 Assume that for $n\geq 1$, the pair $(a_{n}^*,u_{n}^*)$ exists and satisfies 
\begin{align*}
	&a_{n}^* \in C(\mathbb{R}_+; \dot{B}^{\frac{d}{2}-1}_{2,1}\cap \dot{B}^{\frac{d}{2}}_{2,1})
	   \cap L^1(\mathbb{R}_+;\dot{B}^{\frac{d}{2}+1}_{2,1}\cap \dot{B}^{\frac{d}{2}+2}_{2,1}),\quad u_{n}^* \in C(\mathbb{R}_+; \dot{B}^{\frac{d}{2}-1}_{2,1})
	   \cap L^1(\mathbb{R}_+;\dot{B}^{\frac{d}{2}+1}_{2,1}).
\end{align*}
In terms of the product law \eqref{uv2}, the nonlinear terms in \eqref{AP1} satisfy 
\begin{align}
	\|u_{n}^* \cdot \nabla a_{n}^*\|_{\widetilde{L}_t^1(\dot{B}^{\frac{d}{2}}_{2,1})}
	&\lesssim \|u_{n}^*\|_{\widetilde{L}_t^2(\dot{B}^{\frac{d}{2}}_{2,1})}
	         \|\nabla a_{n}^*\|_{\widetilde{L}_t^2(\dot{B}^{\frac{d}{2}}_{2,1})},\\
	\|u_{n}^* \cdot \nabla a_{n}^*\|_{\widetilde{L}_t^1(\dot{B}^{\frac{d}{2}-1}_{2,1})}
	&\lesssim \|u_{n}^*\|_{\widetilde{L}_t^\infty(\dot{B}^{\frac{d}{2}-1}_{2,1})}
	         \|\nabla a_{n}^*\|_{\widetilde{L}_t^2(\dot{B}^{\frac{d}{2}}_{2,1})},\\
	\|u_{n}^* \cdot \nabla u_{n}^*\|_{\widetilde{L}_t^1(\dot{B}^{\frac{d}{2}-1}_{2,1})}
	&\lesssim \|u_{n}^*\|_{\widetilde{L}_t^\infty(\dot{B}^{\frac{d}{2}-1}_{2,1})}
	         \|\nabla u_{n}^*\|_{\widetilde{L}_t^2(\dot{B}^{\frac{d}{2}}_{2,1})}.
\end{align}
Therefore, by the standard theory for parabolic equations, there exists a unique solution $(a_{n+1}^*,u_{n+1}^*)$ to \eqref{AP1} such that 
\begin{align}
	&a_{n+1}^* \in C(\mathbb{R}_+; \dot{B}^{\frac{d}{2}-1}_{2,1}\cap \dot{B}^{\frac{d}{2}}_{2,1})
	   \cap \widetilde{L}_t^1(\mathbb{R}_+;\dot{B}^{\frac{d}{2}+1}_{2,1}\cap \dot{B}^{\frac{d}{2}+2}_{2,1}),\\
	&u_{n+1}^* \in C(\mathbb{R}_+; \dot{B}^{\frac{d}{2}-1}_{2,1})
	   \cap \widetilde{L}_t^1(\mathbb{R}_+;\dot{B}^{\frac{d}{2}+1}_{2,1}).
\end{align}

\medskip
\noindent\textbf{Step 2: Uniform estimates.} We claim that there exist constants $\tilde{C}_0>0$ and $\var_0>0$ such that if 
\begin{align}
	E_n(t)=\|a_{n}^*\|_{\widetilde{L}_t^\infty(\dot{B}^{\frac{d}{2}-1}_{2,1}\cap \dot{B}^{\frac{d}{2}}_{2,1})}
	 +\|a_{n}^*\|_{\widetilde{L}_t^1(\dot{B}^{\frac{d}{2}+1}_{2,1}\cap \dot{B}^{\frac{d}{2}+2}_{2,1})}+\|u_{n}^*\|_{\widetilde{L}_t^\infty(\dot{B}^{\frac{d}{2}-1}_{2,1})}
	 +\|u_{n}^*\|_{\widetilde{L}_t^1(\dot{B}^{\frac{d}{2}+1}_{2,1})} 
	\leq \var_0,	
\end{align}	
then it holds for all $n\geq 0$ that 
\begin{align}\label{AP2}
	&E_n(t)\leq \tilde{C}_0\Big(\|a_0^*\|_{\dot{B}^{\frac{d}{2}}_{2,1}\cap \dot{B}^{\frac{d}{2}-1}_{2,1}}
                  +\|u_0^*\|_{\dot{B}^{\frac{d}{2}-1}_{2,1}}\Big).
\end{align}

We have already analyzed the source terms in \eqref{AP1}. According to maximal regularity estimates for the heat equation, one has 
\begin{align}
	&\|a_{n+1}^*\|_{\widetilde{L}_t^\infty(\dot{B}^{\frac{d}{2}-1}_{2,1}\cap \dot{B}^{\frac{d}{2}}_{2,1})}
	   +\|a_{n+1}^*\|_{\widetilde{L}_t^1(\dot{B}^{\frac{d}{2}+1}_{2,1}\cap \dot{B}^{\frac{d}{2}+2}_{2,1})}
	   +\|u_{n+1}^*\|_{\widetilde{L}_t^\infty(\dot{B}^{\frac{d}{2}-1}_{2,1})}
	   +\|u_{n+1}^*\|_{\widetilde{L}_t^1(\dot{B}^{\frac{d}{2}+1}_{2,1})}\nonumber\\
	&\leq C_1\Big(\|a_0^*\|_{\dot{B}^{\frac{d}{2}-1}_{2,1}\cap \dot{B}^{\frac{d}{2}}_{2,1}}
	              +\|u_0^*\|_{\dot{B}^{\frac{d}{2}-1}_{2,1}}\Big)
	   +C_1\Big(\|u_{n}^* \cdot \nabla a_{n}^*\|_{\widetilde{L}_t^1(\dot{B}^{\frac{d}{2}-1}_{2,1}\cap \dot{B}^{\frac{d}{2}+1}_{2,1})}
	           +\|u_{n}^* \cdot \nabla u_{n}^*\|_{\widetilde{L}_t^1(\dot{B}^{\frac{d}{2}}_{2,1})}\Big)\nonumber\\
	&\leq C_1\Big(\|a_0^*\|_{\dot{B}^{\frac{d}{2}-1}_{2,1}\cap \dot{B}^{\frac{d}{2}}_{2,1}}
	              +\|u_0^*\|_{\dot{B}^{\frac{d}{2}-1}_{2,1}}\Big)
	   +C_2\var_0\Big(\|a_0^*\|_{\dot{B}^{\frac{d}{2}-1}_{2,1}\cap \dot{B}^{\frac{d}{2}}_{2,1}}
	                  +\|u_0^*\|_{\dot{B}_{2,1}^{\frac{d}{2}-1}}\Big).
\end{align}
Let $\tilde{C}_0 = 2C_1$ and choose $\var_0$ so small that $C_2\var_0 \leq \tilde{C}_0$. Then \eqref{AP2} holds for $n+1$, and by induction, it holds for all $n\geq 0$.

\medskip
\noindent\textbf{Step 3: Convergence.} We now show that the approximate sequence converges. Taking the difference between $(a_{n}^*, u_{n}^*)$ and $(a_{n-1}^*,u_{n-1}^*)$ for $n\geq 2$, we obtain
\begin{equation}\label{AP3}
	\left\{
	\begin{aligned}
		&\partial_t(a_{n}^* - a_{n-1}^*) - \Delta (a_{n}^* - a_{n-1}^*) \\
		&\qquad= -u_{n-1}^* \cdot \nabla (a_{n-1}^* - a_{n-2}^*)
		         -(u_{n-1}^* - u_{n-2}^*) \cdot \nabla a_{n-2}^*,\\[1mm]
		&\partial_t(u_{n}^* - u_{n-1}^*) -\mu\Delta (u_{n}^* - u_{n-1}^*) \\
		&\qquad= -u_{n-1}^* \cdot \nabla(u_{n-1}^* - u_{n-2}^*)
		         -(u_{n-1}^* - u_{n-2}^*) \cdot \nabla u_{n-2}^*,\\[1mm]
		&(a_{n}^* - a_{n-1}^*,u_{n}^* - u_{n-1}^*)(0,x) = (0,0).
	\end{aligned}
	\right.
\end{equation}

Applying maximal regularity estimates, the product law \eqref{uv2} and \eqref{AP2} yields 
\begin{align}
	&\|(a_{n}^* - a_{n-1}^*,u_{n}^* - u_{n-1}^*)\|_{\widetilde{L}_t^\infty(\dot{B}^{\frac{d}{2}-1}_{2,1})\cap \widetilde{L}_t^1(\dot{B}^{\frac{d}{2}+1}_{2,1})}	\nonumber\\
	&\lesssim \|u_{n-1}^* \cdot \nabla(a_{n-1}^* - a_{n-2}^*)\|_{\widetilde{L}_t^1(\dot{B}^{\frac{d}{2}-1}_{2,1})}
	      +\|(u_{n-1}^* - u_{n-2}^*) \cdot \nabla a_{n-2}^*\|_{\widetilde{L}_t^1(\dot{B}_{2,1}^{\frac{d}{2}-1})}\nonumber\\
	&\quad +\|u_{n-1}^* \cdot \nabla(u_{n-1}^* - u_{n-2})\|_{\widetilde{L}_t^1(\dot{B}_{2,1}^{\frac{d}{2}-1})}
	      +\|(u_{n-1}^* - u_{n-2}^*) \cdot \nabla u_{n-2}^*\|_{\widetilde{L}_t^1(\dot{B}_{2,1}^{\frac{d}{2}-1})}\nonumber\\
	&\lesssim \|(a_{n-1}^*,u_{n-1}^*)\|_{\widetilde{L}_t^2(\dot{B}_{2,1}^{\frac{d}{2}})}
	          \|(a_{n-1}^* - a_{n-2}^*, u_{n-1}^* - u_{n-2}^*)\|_{\widetilde{L}_t^2(\dot{B}_{2,1}^{\frac{d}{2}})}\nonumber\\
	&\leq C_3\var_0 \|(a_{n-1}^* - a_{n-2}^*, u_{n-1}^* - u_{n-2}^*)\|_{\widetilde{L}_t^\infty(\dot{B}_{2,1}^{\frac{d}{2}-1})\cap \widetilde{L}_t^1(\dot{B}_{2,1}^{\frac{d}{2}+1})}.
\end{align}
Let $\var_0\leq \min\{\frac{1}{C_2},\frac{1}{2C_3}\}$. Then, for $n\geq2$, we obtain 
\begin{align}
	&\|(a_{n-1}^* - a_{n-2}^*,u_{n-1}^* - u_{n-2}^*)\|_{\widetilde{L}_t^\infty(\dot{B}_{2,1}^{\frac{d}{2}-1})\cap \widetilde{L}_t^1(\dot{B}_{2,1}^{\frac{d}{2}+1})}\nonumber\\
	&\leq \frac{1}{2^{n-2}}\|(a_2^* - a_1^*, u_2^* - u_1^*)\|_{\widetilde{L}_t^\infty(\dot{B}_{2,1}^{\frac{d}{2}-1})\cap \widetilde{L}_t^1(\dot{B}_{2,1}^{\frac{d}{2}+1})}.\label{mao}
\end{align}
Since $(a_1^*,u_1^*) = (e^{t\Delta}a_0^*,e^{t\mu\Delta}u_0^*)$ and $(a_{n}^*, u_{n}^*) =\sum\limits_{k=2}^{n}(a_k^* - a_{k-1}^*, u_k^* - u_{k-1}^*)+(a_1^*,u_1^*)$, 
one has
\begin{align*}
	\sum_{i=2}^n\|(a_{i}^* - a_{i-1}^*, u_{i}^* - u_{i-1}^*)\|_{\widetilde{L}_t^\infty(\dot{B}_{2,1}^{\frac{d}{2}-1})\cap \widetilde{L}_t^1(\dot{B}_{2,1}^{\frac{d}{2}+1})}
	\lesssim \|a_0^*\|_{\dot{B}^{\frac{d}{2}-1}_{2,1}\cap \dot{B}^{\frac{d}{2}}_{2,1}}
	+\|u_0^*\|_{\dot{B}_{2,1}^{\frac{d}{2}-1}}.
\end{align*}
This allows us to define 
\begin{align*}
	(a^*,u^*): =\sum_{k=2}^{\infty}(a_k^* - a_{k-1}^*, u_k^* - u_{k-1}^*)+(a_1^*,u_1^*).	
\end{align*}
Then we discover that
\begin{align}
	&\|(a_{n}^* - a^*, u_{n}^* - u^*)\|_{\widetilde{L}_t^\infty(\dot{B}_{2,1}^{\frac{d}{2}-1})\cap \widetilde{L}_t^1(\dot{B}_{2,1}^{\frac{d}{2}+1})}	\nonumber\\
	&\lesssim \sum_{k=n+1}^{\infty}\|(a_k^* - a_{k-1}^*,u_{k}^* - u_{k-1}^*)\|_{\widetilde{L}_t^\infty(\dot{B}_{2,1}^{\frac{d}{2}-1})\cap \widetilde{L}_t^1(\dot{B}_{2,1}^{\frac{d}{2}+1})}\lesssim \sum_{k=n+1}^{\infty}\frac{1}{2^k} \rightarrow 0\quad \text{as}\quad n\rightarrow +\infty.
\end{align}

Consequently, we see that the sequence $(a_n^*, u_n^*)$ converges strongly to a limit $(a^*,u^*)$ as $n\rightarrow\infty$. After passing to the limit in \eqref{AP1} in the sense of distributions, one can conclude that $(\rho^*,u^*)$ with $\rho^*=1+a^*$ is a global solution to the Cauchy problem \eqref{main2}-\eqref{main2d} for the KS-NS equations. The proof of uniqueness is similar to \eqref{AP3}-\eqref{mao}, so the details are omitted. This completes the global well-posedness part of Proposition \ref{thm1}.

\subsection{Large-time behavior}
We now proceed to study the large-time behavior of the solution to the KS-NS equations. In addition to \eqref{a0},
if we further assume $
			(\rho_0^*-1, u_0^*)\in \dot{B}_{2,\infty}^{\sigma_1}$ with $\sigma_1\in [-\frac{d}{2},\frac{d}{2}-1)$, then we are able to show that the solution $(a, u)$ to KS-NS satisfies the decay estimates \eqref{decayNSKS},
which in particular implies that $(a, u)$ becomes smooth for any positive time $t>0$.

We divide the proof into three steps.

\medskip
\noindent\textbf{Step 1: Time-weighted  estimates.} For any $M>\max\{1,\frac{1}{2}(\sigma-\sigma_1)\}$, it holds that
\begin{equation}\label{AP4}
	\left\{
	\begin{aligned}
		&\partial_t(t^M a^*) - \Delta (t^M a^*)
		 = Mt^{M-1}a^* + t^M u^* \cdot \nabla a^*,\\
		&\partial_t(t^M u^*) - \mu\Delta (t^M u^*)
		 = Mt^{M-1}u^* + t^M u^* \cdot \nabla u^*,\\
		&(t^M a^*, t^M u^*)(0,x) = (0,0).
	\end{aligned}
	\right.
\end{equation}
By maximal regularity estimates, one has 
\begin{align}\label{TP1}
	&\|\tau^M(a^*,u^*)\|_{\widetilde{L}_t^\infty(\dot{B}_{2,1}^\sigma)}
	 +\int_0^t\|\tau^M(a^*,u^*)(\tau)\|_{\dot{B}_{2,1}^{\sigma+2}}\,d\tau\nonumber\\
	&\lesssim \int_0^t \Big(\tau^{M-1}\|(a^*, u^*)(\tau)\|_{\dot{B}_{2,1}^\sigma}
	   +\|\tau^M(u^*\cdot\nabla a^*, u^*\cdot\nabla u^*)(\tau)\|_{\dot{B}_{2,1}^\sigma}\Big)\,d\tau.	
\end{align}

Using the interpolation and Hölder's inequalities, we obtain 
\begin{align}\label{AP5}
	&\int_0^t \tau^{M-1}\|(a^*, u^*)(\tau)\|_{\dot{B}_{2,1}^\sigma}\,d\tau\nonumber\\	
	&\lesssim \int_0^t \tau^{M-1}
	  \|(a^*, u^*)\|_{\dot{B}_{2,\infty}^{\sigma_1}}^\theta
	  \|(a^*, u^*)\|_{\dot{B}_{2,1}^{\sigma+2}}^{1-\theta}\,d\tau\nonumber\\
	&\lesssim \|(a^*,u^*)\|_{\widetilde{L}_t^\infty(\dot{B}_{2,\infty}^{\sigma_1})}^\theta
	  \Big(\int_0^t \tau^{\frac{M-1-(1-\theta)M}{\theta}}\,d\tau\Big)^\theta
	  \|\tau^M (a^*, u^*)\|_{\widetilde{L}_t^1(\dot{B}_{2,1}^{\sigma+2})}^{1-\theta}\nonumber\\
	&\lesssim \eta_2\|\tau^M(a^*, u^*)\|_{\widetilde{L}_t^1(\dot{B}_{2,1}^{\sigma+2})}
	  +\frac{1}{\eta_2}
	   \|(a^*, u^*)\|_{\widetilde{L}_t^\infty(\dot{B}_{2,\infty}^{\sigma_1})}
	   t^{M-\frac{1}{2}(\sigma-\sigma_1)},
\end{align}
where the index $\sigma$ satisfies $\sigma=\theta \sigma_1+(1-\theta)(\sigma+2)$ and $\eta_2>0$ is a constant that will be chosen later.

Concerning the nonlinear terms in \eqref{TP1}, we divide into two cases:
\begin{itemize}
	\item If $\sigma_1<\sigma\leq \frac{d}{2}$, then by the product law \eqref{uv2}, it holds 
	\begin{align}\label{AP6}
		\|\tau^M(u^*\cdot\nabla a^*, u^*\cdot\nabla u^*)\|_{\widetilde{L}_t^1(\dot{B}_{2,1}^\sigma)}	
		&\lesssim \int_0^t \tau^M\|u^*\|_{\dot{B}_{2,1}^\sigma}
		  \|(\nabla a^*, \nabla u^*)\|_{\dot{B}_{2,1}^{\frac{d}{2}}}\,d\tau\nonumber\\
		&\lesssim \int_0^t\|(a^*, u^*)\|_{\dot{B}_{2,1}^{\frac{d}{2}+1}}
		  \|\tau^M(a^*, u^*)\|_{\dot{B}_{2,1}^\sigma}\,d\tau.
	\end{align} 
	
	\item If $\sigma>\frac{d}{2}$, then the Moser type estimate \eqref{uv1} and the interpolation \eqref{inter} ensure that
	\begin{align}\label{AP7}
		&\|\tau^M(u^*\cdot\nabla a^*, u^*\cdot\nabla u^*)\|_{\widetilde{L}_t^1(\dot{B}_{2,1}^\sigma)}	\nonumber\\
		&\lesssim \int_0^t \tau^M\Big(\|u^*\|_{L_t^\infty}
		 \|(\nabla a^*, \nabla u^*)\|_{\dot{B}_{2,1}^\sigma}
		 +\|u^*\|_{\dot{B}_{2,1}^\sigma}\|(\nabla a^*,\nabla u^*)\|_{L_t^\infty}\Big)\,d\tau\nonumber\\
		&\lesssim \int_0^t\tau^M\Big(\|u^*\|_{\dot{B}_{2,1}^{\frac{d}{2}}}
		 \|(a^*, u^*)\|_{\dot{B}_{2,1}^\sigma}^{\frac{1}{2}}
		 \|(a^*, u^*)\|_{\dot{B}_{2,1}^{\sigma+2}}^{\frac{1}{2}}
		 +\|u^*\|_{\dot{B}_{2,1}^\sigma}\|(a^*, u^*)\|_{\dot{B}_{2,1}^{\frac{d}{2}+1}}
		 \Big)\,d\tau\nonumber\\
		&\lesssim \eta_2\|\tau^M(a^*, u^*)\|_{\widetilde{L}_t^1(\dot{B}_{2,1}^{\sigma+2})}
		  +\int_0^t\Big(\|u^*\|_{\dot{B}_{2,1}^{\frac{d}{2}}}^2
		  +\|(a^*, u^*)\|_{\dot{B}_{2,1}^{\frac{d}{2}+1}}\Big)
		   \|\tau^M(a^*, u^*)\|_{\dot{B}_{2,1}^\sigma}\,d\tau.
	\end{align} 
\end{itemize}

Combining \eqref{TP1}–\eqref{AP7} and choosing $\eta_2>0$ sufficiently small, we obtain 
\begin{align*}
	&\|\tau^M (a^*, u^*)\|_{\dot{B}_{2,1}^\sigma}
	 +\int_0^t\tau^M\|(a^*, u^*)(\tau)\|_{\dot{B}_{2,1}^{\sigma+2}}\,d\tau\nonumber\\
	&\lesssim \int_0^t\Big(\|u^*\|_{\dot{B}_{2,1}^{\frac{d}{2}}}^2
	 +\|(a^*, u^*)\|_{\dot{B}_{2,1}^{\frac{d}{2}+1}}\Big)
	 \|\tau^M(a^*, u^*)\|_{\dot{B}_{2,1}^\sigma}\,d\tau
	 +\|(a^*, u^*)\|_{\widetilde{L}_t^\infty(\dot{B}_{2,\infty}^{\sigma_1})}
	   t^{M-\frac{1}{2}(\sigma-\sigma_1)}.
\end{align*}

By Grönwall's inequality and the global existence estimates in \eqref{r0}, it follows that
\begin{align}\label{AP8}
	&\|\tau^M(a^*, u^*)\|_{\dot{B}_{2,1}^\sigma}
	 +\int_0^t \|\tau^M(a^*, u^*)(\tau)\|_{\dot{B}_{2,1}^{\sigma+2}}\,d\tau\nonumber\\
	&\lesssim \|(a^*, u^*)\|_{\widetilde{L}_t^\infty(\dot{B}_{2,\infty}^{\sigma_1})}
	  t^{M-\frac{1}{2}(\sigma-\sigma_1)}
	  \Big(1+{\rm{Exp}}\Big\{\int_0^t \Big(\|u^*\|_{\dot{B}_{2,1}^{\frac{d}{2}}}^2
	    +\|(a^*,u^*)\|_{\dot{B}_{2,1}^{\frac{d}{2}+1}}\Big)\,d\tau\Big\}\Big)\nonumber\\
        &\lesssim \|(a^*, u^*)\|_{\widetilde{L}_t^\infty(\dot{B}_{2,\infty}^{\sigma_1})}.
\end{align} 


\medskip
\noindent\textbf{Step 2: Gain of time-decay rates.}  By maximal regularity for the KS-NS system \eqref{main4}, we obtain 
\begin{align}\label{evolution00}
	&\|(a^*, u^*)\|_{\widetilde{L}_t^\infty(\dot{B}_{2,\infty}^{\sigma_1})}
	 +\|(a^*, u^*)\|_{ \widetilde{L}_t^1(\dot{B}_{2,\infty}^{\sigma_1+2})}\nonumber\\
	&\lesssim \|(a_0^*, u_0^*)\|_{\dot{B}_{2,\infty}^{\sigma_1}}
	 +\|(u^*\cdot\nabla a^*, u^*\cdot\nabla u^*)\|_{\widetilde{L}_t^1(\dot{B}_{2,\infty}^{\sigma_1})}.
\end{align}
Moreover, using \eqref{uv3} leads to
\begin{align*}
	\|(u^*\cdot\nabla a^*, u^*\cdot\nabla u^*)\|_{\widetilde{L}_t^1(\dot{B}_{2,\infty}^{\sigma_1})}
    &\lesssim \int_0^t \|(\nabla a^*, \nabla u^*)\|_{\dot{B}_{2,1}^{\frac{d}{2}}}
	   \|u^*\|_{\dot{B}_{2,\infty}^{\sigma_1}}\,d\tau.
\end{align*}
By Grönwall's inequality and the fact $(\nabla a^*, \nabla u^*)\in L^1(\mathbb{R}^+;\dot{B}^{\frac{d}{2}}_{2,1})$, we deduce 
\begin{align}\label{AP9}
	&\|(a^*, u^*)\|_{\widetilde{L}_t^\infty(\dot{B}_{2,\infty}^{\sigma_1})\cap \widetilde{L}_t^1(\dot{B}_{2,\infty}^{\sigma_1+2})} 
	\lesssim \|(a_0^*, u_0^*)\|_{\dot{B}_{2,\infty}^{\sigma_1}}.
\end{align}
Substituting the estimate \eqref{AP9} into \eqref{AP8}, we obtain
\begin{align*}
	t^M\|(a^*,u^*)\|_{\dot{B}_{2,1}^\sigma}
	\lesssim t^{M-\frac{1}{2}(\sigma-\sigma_1)},
\end{align*}
which immediately implies the desired decay estimate \eqref{decayNSKS}.  This completes the proof of Proposition \ref{thm1}.

		\section*{Acknowledgements}
M. Fei is supported partly by the National Science Foundation of China (Grant No.12271004, No.12471222) and the Natural Science Foundation of Anhui Province of China (Grant No.2308085J10). H. Tang is supported by the National Science Foundation of China (Grant No.12501293) and the  Natural Science Foundation  of Anhui Province of China  (Grant No.2408085QA031). L. Shou is supported by the National Science Foundation of China (Grant No.12301275).

	\bibliographystyle{abbrv}  
	\bibliography{reference}

		\vspace{5ex}

		(M.-W. Fei)\par\nopagebreak
		\noindent\textsc{School of Mathematics and Statistics, Anhui Normal University, Wuhu 241002, P. R. China}
		
		Email address: {\texttt{mwfei@ahnu.edu.cn}}
		
		\vspace{3ex}

		(L.-Y. Shou)\par\nopagebreak
		\noindent\textsc{School of Mathematical Sciences, Ministry of Education Key Laboratory of NSLSCS, 
and Key Laboratory of Jiangsu Provincial Universities of FDMTA, Nanjing Normal University, Nanjing 210023, P. R. China}
		
		Email address: {\texttt{shoulingyun11@gmail.com}}
		
		\vspace{3ex}

		(H.-Z. Tang)\par\nopagebreak
		\noindent\textsc{School of Mathematics and Statistics, Anhui Normal University, Wuhu 241002, P. R. China}
		
		Email address: {\texttt{houzhitang@ahnu.edu.cn}}

	\end{document}